\documentclass[12pt, a4paper]{amsart}
\usepackage{amscd, amssymb}

\usepackage{tensor}
\usepackage{wrapfig}
\usepackage{tikz}

\allowdisplaybreaks

\usepackage{xcolor}

\def\Z{\mathbb{Z}}

\def\AA{\mathcal{A}}
\def\BB{\mathcal{B}}
\def\CC{\mathcal{C}}
\def\DD{\mathcal{D}}

\def\LL{\mathcal{L}}
\def\MM{\mathcal{M}}
\def\NN{\mathcal{N}}

\def\RR{\mathcal{R}}

\def\TT{\mathcal{T}}

\newtheorem{thm}{Theorem}[section]
\newtheorem{cor}[thm]{Corollary}

\newtheorem{lemma}[thm]{Lemma}
\newtheorem{prop}[thm]{Proposition}

\theoremstyle{definition}
\newtheorem{definition}[thm]{Definition}

\theoremstyle{remark}
\newtheorem{remark}[thm]{Remark}

\newtheorem{example}[thm]{Example}

\numberwithin{equation}{section}

\tikzstyle{vertex}=[circle]
\tikzstyle{goto}=[->,shorten >=1pt,>=stealth,semithick]

\usetikzlibrary{arrows}
\usepackage{color,soul}
\usepackage{comment}
\usepackage[linguistics]{forest}
\usetikzlibrary {decorations.pathreplacing}
\usetikzlibrary{arrows.meta}
\usepackage{mathtools}

\usepackage{hyperref}
\hypersetup{
    colorlinks=true,
    linkcolor=blue,
    filecolor=blue,      
    urlcolor=blue,
    citecolor=blue,
    }

\usepackage{float}

\begin{document}
 \title[Geometric properties of the golden ratio Thompson's group]{Geometric properties of the golden ratio Thompson's group}
\author{Denys Svetelik}

\begin{abstract}
    We show that all three golden ratio Thompson's groups $F_\tau$, $T_\tau$ and $V_\tau$  embed in the asynchronous rational group. We prove properties of the Cayley graph of the monoid $M = \langle L, R : LR^2 = RL^2 \rangle$, whose topological full group is $V_\tau$. In particular, we compute a distance function for the Cayley graph of the monoid $M$. Additionally, we prove that this Cayley graph is hyperbolic in the sense of Gromov. Our analysis reveals that the horofunction boundary of this graph is homeomorphic to a space resembling a Cantor-like set, with additional isolated points situated between each pair of breakpoints.
\end{abstract}
\maketitle

\tableofcontents

\section{Introduction}
\subsection{Overview} In 1965, Thompson introduced three groups $F$, $T$, $V$, which were used in~\cite{McT} for the construction of finitely presented groups with unsolvable word problems. In~\cite{T2} Thompson showed that the groups $T$ and $V$ are finitely presented, infinite simple groups and used the group $V$ to prove that a group with finitely many generators has a solvable word problem if and only if it can be embedded into a finitely generated simple subgroup of a finitely presented group. The group $F$ was used in a number of works related to homotopy idempotents \cite{Dy2, FrH}. In \cite{BroG} it was proven that $F$ is the first known example of a torsion-free infinite-dimensional $FP_\infty$ group. Later in \cite{Mi} it was proven that $F$ is simply connected at infinity, thus showing that the group does not have any homotopy at infinity.

The Thompson's groups $F, T$, and $V$ can be seen as groups of piecewise linear homeomorphisms of the unit interval, the unit circle, and the Cantor set, respectively, that map dyadic rational numbers to dyadic rational numbers, and are differentiable everywhere except for a finite number of dyadic numbers, and the derivative on the interval of differentiability is always a power of 2. For a detailed introduction on Thompson's groups, we refer the reader to \cite{Cannon F} and \cite{Burillo Thompson}.

In 2000 S. Cleary in his paper \cite{Clearly} introduced the irrational slope Thompson's groups. These irrational slope Thompson's groups are a variation of the classic Thompson's group, with the break points now in $\mathbb{Z}[\tau]$ and the slopes being powers of $\tau$, where $\tau = \frac{\sqrt{5}-1}{2}$ is the positive square root of the equation~$x^2+x=1$ and is called the \textit{small golden ratio}. In~\cite{Ttau} and \cite{Burillo T} J. Burillo, B. Nucinkis, and L. Reeves proved that the commutator subgroup of $F_\tau$ is simple but $T_\tau$ and $V_\tau$ have index-2 normal subgroups; they also gave a finite presentation of these groups and represented them in terms of binary trees.

The approach we adopt for \textit{rationality} is based on the work of Grigorchuk, Nekrashevych, and Sushchanskii \cite{Grigorchuk Automata}, who employed finite state machines to describe sets, relations, and functions. A central objective of this framework is to define homeomorphisms in the Cantor set $\{0, 1\}^\omega$ using asynchronous machines, where a single symbol is read as input and a finite string, composed of elements of $\{0, 1\}$, is produced as output during each step of the computation. These homeomorphisms correspond to \textit{rational functions}.

A homeomorphism of the Cantor set $\{0, 1\}^\omega$ is called \textit{rational} if it can be realized by an asynchronous transducer (or asynchronous Mealy machine) that acts on infinite binary strings. In \cite{Grigorchuk Automata}, Grigorchuk, Nekrashevych, and Sushchanskii observed that the collection of all rational homeomorphisms of $\{0, 1\}^\omega$ forms a group $\RR$ under the operation of composition, which they called the \textit{rational group}. Additionally, they point out that the group of rational homeomorphisms of $A^\omega$, for any finite alphabet $A$ with at least two symbols, is isomorphic to $\RR$.

During the last few decades, relatively little focus has been on the class of groups generated by asynchronous transducers, particularly the full asynchronous rational group~$\RR$. It is established that the group~$\RR$ is simple but not finitely generated \cite{Belk async}. Furthermore, finitely generated subgroups of~$\RR$ have a solvable word problem; however, there is no solution to the periodicity problem for elements of~$\RR$~\cite{Grigorchuk Automata,Belk some}. Additionally,~$\RR$ contains several well-known groups: the Thompson groups $F$, $T$, and $V$, the Brin-Thompson groups~$nV$ and groups like the Röver group~$ V_\Gamma$ \cite{Grigorchuk Automata,Belk some,Rover}. Furthermore, any group generated by synchronous automata can be embedded into $\RR$.

In 2023 Belk, Bleak, Matucci and Zaremsky proved that all hyperbolic groups satisfy the Boone–Higman conjecture in \cite{Belk Boone-Higman}. The Boone–Higman conjecture was proposed in 1973 and states that a finitely generated group with a solvable word problem can be embedded in a finitely presented simple group. The authors proved that hyperbolic groups meet this criterion by demonstrating that each hyperbolic group embeds within a finitely presented simple group. In their proof, they introduced a new class of groups, which they called \textit{rational similarity groups}. Rational similarity groups generalize self-similar groups by allowing for rational, "canonical similarity" transformations within subshifts of finite type. Specifically, the authors showed that each hyperbolic group is embedded in a full, contracting rational similarity group, which is embedded in a finitely presented simple group. Additionally, rational similarity groups emerged as fundamental objects in the study of asynchronous group actions, suggesting their potential broader applicability in embedding problems within geometric and combinatorial group theory.

In \cite{Belk Rational embeddings}, Belk, Bleak and Matucci proved that every Gromov hyperbolic group $G$ embeds in the rational group $\RR$, where by Gromov hyperbolic group we mean a group with a Cayley graph $\Gamma$ that is hyperbolic in Gromov sense. This embedding uses a framework in which elements of $G$ act on a boundary space (horofunction boundary~$\partial_h G$) of binary sequences by asynchronous transducers. In the same paper, the authors proved that for any hyperbolic group $G$, the action of $G$ on its horofunction boundary $\partial_h G$ is rational. In the construction, the authors introduce a specific tree structure within the hyperbolic graph $\Gamma$, which they call the \textit{tree of atoms}. They observed that when a group~$G$ acts properly and cocompactly on $\Gamma$ then the tree of atoms has a self-similar structure and is in addition is naturally homeomorphic to the horofunction boundary $\partial_h \Gamma$ (also known as the metric boundary). In \cite{Webster boundary} it was shown that the horofunction boundary~$\partial_h \Gamma$ is compact, totally disconnected and has the Gromov boundary $\partial \Gamma$ as a quotient.  

The horofunction boundary of a non-elementary group was proven to be homeomorphic to the Cantor set \cite{Belk Rational embeddings}. It was still an open question if the same statement holds for non-elementary monoids. It turns out that the monoid we study in this paper is a conterexample to the statement, which proves that the horofunction boundary of non-elemenrary monoids is more complicated then expected. 

\subsection{Summary of main results} In section \ref{sec: Irrational slope} we prove the following theorem. Let $\{0,1\}^\omega$ be the set of all binary sequences. Let~$X_i, Y_i, C_{i+1}$ and~$\Pi_i$, for~$i=\{0,1\}$ be rational homeomorphisms of the Cantor set~$\{0,1\}^\omega$ defined by the automata shown in Figure~\ref{fig:v_tau automaton} (which differ only in their initial states).

\begin{thm}
The group $G$ of homeomorphisms of $\{0,1\}^\omega$ generated by $X_0,$ $X_1,$ $Y_0,$ $Y_1,$ $C_1,$ $C_2,$ $\Pi_0$ and $\Pi_1$ is isomorphic to the golden ratio Thompson group $V_\tau$. Moreover $V_\tau$ is a rational similarity group.
\end{thm}
\noindent The first part of the theorem is proven directly in Theorem \ref{v_tau automaton}. First, we identify a quotient map from the space of binary sequences $\{0,1\}^\omega$ to the unit interval $[0,1]$. We then establish that an isomorphism exists between the two groups. The second part is proven in Theorem \ref{nuc}. The proof relies heavily on the ideas presented by Belk, Bleak, Matucci, and Zaremsky in \cite{Belk Boone-Higman}. They stated the conditions for a set of maps to be a nucleus of injections (Definition~\ref{nuc def}) and the conditions for a group to be a rational similarity group with a given nucleus (Theorem~\ref{thm: 2.46}). We adapt these ideas to find the nucleus~$\NN$ for $G$, this gives a background to prove that $V_\tau$ is a rational similarity group.

From this, we derive a corollary: $F_\tau$ and $T_\tau$ are isomorphic to subgroups of $G$ formed by the generators~$X_i$, $Y_i$ and $X_i$, $Y_i$, $C_{i+1}$, respectively. This provides insight into the interaction of irrational slope Thompson groups within the context of automata and geometric representations.

Section \ref{sec: Monoid} is heavily based on studying the properties of the Cayley graph of the monoid~$M=\langle L,R: LR^2= RL^2\rangle$, whose topological full group is $V_\tau$. Let $Cay(M)$ be the respective Cayley graph; see Figure \ref{fig: MM}. Observe that the monoid $M$ acts as the vertex set in $Cay(M)$, which we impose with the path metric. Let~$x$, $m$ be vertices in $Cay(M)$, we say that~$x\in Cone(m)$, if there is a word for $x$ that starts with $m$. The main result of the section is the following theorem.
\begin{thm}
Let $x,y $ be any given pair of vertices in the Cayley graph $Cay(M)$ of the monoid $M = \langle L,R:LR^2=RL^2 \rangle$. Then the distance between $x$ and $y$ is given  by:
\[
d(x,y) =\begin{cases}
        d(x',y'), & \begin{tabular}[t]{@{}l@{}} 
        $\textit{when } x=mx', y=my' \textit{ for some nontrivial } m\in M$\\
\end{tabular}\\
        |x|+|y| -2, & \begin{tabular}[t]{@{}l@{}} 
        $\text{when } x \in Cone(LR), y \in Cone(RL) \textit{ and } x,y \notin Cone(LR^2);$\\
\end{tabular}\\
        |x|+|y|,  & \begin{tabular}[t]{@{}l@{}} 
        else.\\
\end{tabular} 
    \end{cases}
\]
\end{thm}
\noindent The proof is presented in Theorem \ref{M distance}. The first step we make is to observe that every cone in $Cay(M)$ is strongly geodesically convex, where by strongly geodesically convex we mean that if $x,y \in Cone(m)$ then any geodesic between $x$ and $y$ is fully contained in $Cone(m)$. The remaining proof involves manipulating possible paths that a geodesic between $x$ and~$y$ can take. The key idea here is to observe that if $x$ starts with $L$ and $y$ starts with $R$ then a geodesic either visits the root $1$ or passes through the $Cone(LR^2)$, which is self-similar to the whole graph itself. 

This theorem provides a useful distance function in a Cayley graph, which has a complicated structure. It is not only a powerful tool in Sections 5 and 6 of this paper, but also lays a solid foundation for further study of the monoid $M$.

Section \ref{horofunction boundary section} explores the horofunction boundary of the graph $Cay(M)$, introducing a Cantor-like set $\DD_\tau$. This set includes additional isolated points between each pair of breakpoints, as shown below.
\[
\begin{tikzpicture}[scale=0.85]
\draw (1.5,0) node[anchor=north]{}
-- (3,0) node[anchor=south]{}
    (4.5,0) node[anchor=north]{}
-- (6,0) node[anchor=north]{}
    (7.5,0) node[anchor=north]{}
-- (9,0) node[anchor=north]{}
    (10.5,0) node[anchor=north]{}
-- (12,0) node[anchor=north]{}
    (13.5,0) node[anchor=north]{}
-- (15,0) node[anchor=north]{};

\node[circle,fill=black,inner sep=0pt,minimum size=0pt,label=above:{$\cdots$}] (a) at (0.5,-0.27) {};
\node[circle,fill=black,inner sep=0pt,minimum size=0pt,label=above:{$\cdots$}] (a) at (16.5,-0.27) {};

\node[circle,fill=black,inner sep=0pt,minimum size=4pt,label=above:{$x_1^+$}] (a) at (1.5,0) {};
\node[circle,fill=black,inner sep=0pt,minimum size=4pt,label=above:{$x_2^-$}] (a) at (3,0) {};
\node[circle,fill=black,inner sep=0pt,minimum size=4pt,label=above:{$x_2$}] (a) at (3.75,0) {};
\node[circle,fill=black,inner sep=0pt,minimum size=4pt,label=above:{$x_2^+$}] (a) at (4.5,0) {};
\node[circle,fill=black,inner sep=0pt,minimum size=4pt,label=above:{$x_3^-$}] (a) at (6,0) {};
\node[circle,fill=black,inner sep=0pt,minimum size=4pt,label=above:{$x_3$}] (a) at (6.75,0) {};
\node[circle,fill=black,inner sep=0pt,minimum size=4pt,label=above:{$x_3^+$}] (a) at (7.5,0) {};
\node[circle,fill=black,inner sep=0pt,minimum size=4pt,label=above:{$x_4^-$}] (a) at (9,0) {};
\node[circle,fill=black,inner sep=0pt,minimum size=4pt,label=above:{$x_4$}] (a) at (9.75,0) {};
\node[circle,fill=black,inner sep=0pt,minimum size=4pt,label=above:{$x_4^+$}] (a) at (10.5,0) {};
\node[circle,fill=black,inner sep=0pt,minimum size=4pt,label=above:{$x_5^-$}] (a) at (12,0) {};
\node[circle,fill=black,inner sep=0pt,minimum size=4pt,label=above:{$x_5$}] (a) at (12.75,0) {};
\node[circle,fill=black,inner sep=0pt,minimum size=4pt,label=above:{$x_5^+$}] (a) at (13.5,0) {};
\node[circle,fill=black,inner sep=0pt,minimum size=4pt,label=above:{$x_6^-$}] (a) at (15,0) {};
\end{tikzpicture}
\]
where every break point $x_i$ belongs to the ring $\mathbb{Z}[\tau] \cap (0,1)$, and each interval above represents a similar Cantor set with extra isolated points.

\begin{thm}
    The horofunction boundary $\partial_h Cay(M)$ of the graph $Cay(M)$ of the monoid $M = \langle L,R :LR^2=RL^2\rangle$ is naturally homeomorphic to~$\DD_\tau$.
\end{thm}
\noindent The proof is presented in Theorem \ref{horofunction boundary of $M$}. Our proof is heavily based on the ideas presented by Belk, Bleak, and Matucci in \cite{Belk Rational embeddings}. The main idea of the proof is to construct a tree of atoms. The boundary of this tree is shown to be homeomorphic to the horofunction boundary (Theorem \ref{equivalence boundary = horofunction boundary}). We start by considering vector fields, where each vector corresponds to a distance function between the two vertices, on the set of edges contained in balls of radius~$n \geq 0$ around the root of $Cay(M)$. Since $Cay(M)$ is a locally finite graph, for every $n$, there is a finite number of possible vector fields. For each distinct vector field, we associate a set of vertices from the whole graph, for which the vector field coincides with the distance function. Whenever the set is infinite, we call it an \textit{atom}. With each increment of $n$, the previous atoms decompose into new ones. Eventually, we observe a pattern in the atoms and construct a tree of atoms. This tree has a self-similar structure and can be represented as a directed multi-graph. We show that the path space of this graph is $\DD_\tau$. Then, following~\cite{Belk Rational embeddings}, we conclude that the horofunction boundary of the graph $Cay(M)$ is homeomorphic to $\DD_\tau$.

In Section \ref{sec: hyperbolicity} we embed the graph $Cay(M)$ in $\MM'$, which is the same graph but with additional edges between every pair of adjacent vertices. We prove that these two graphs are quasi-isometric. Relying on the work of Kong, Lau, and Wang \cite{Kong}, we prove certain proprieties (Theorem \ref{hyperbolic theorem}) of $\MM'$ that are equivalent to it being a hyperbolic graph, thus resulting in the following theorem, which is proved in Theorem \ref{M hyperbolic}.
\begin{thm}
    The Cayley graph $Cay(M)$ of the monoid $M = \langle L,R :LR^2=RL^2\rangle$ is $\delta$-hyperbolic.
\end{thm}

These results not only bring important insights, but also raise a number of questions to be answered:
\begin{enumerate}
    \item Are all irrational slope Thompson's groups rational similarity groups? If not, then what are the criteria for a Thompson's group to be one?
    \item Are all horofunction boundaries of the Cayley graphs of the family of monoids $M_n = \langle L,R : LR^n = RL^n \rangle$, for $n>1$, Cantor-like sets with additional isolated points in between every pair of breakpoints?
    \item Are all of these Cayley graphs hyperbolic? 
    \item What are their respective Gromov boundaries?
\end{enumerate}

\noindent \textbf{Acknowledgments.} I would like to thank my supervisors Jim Belk and Mike Whittaker for suggesting these problems and for valuable
discussions and comments.

\section{Background on the Irrational slope Thompson groups}
We assume that the reader is familiar with the original Thompson groups $T, F$ and $V$ \cite{Burillo Thompson} and their irrational slope variants $T_\tau, F_\tau$ and $V_\tau$~\cite{Burillo T}, \cite{Clearly}. 

The group $F_\tau$ shares many structural properties with the original Thompson group $F$, but with a key difference: the slopes and break points of the group elements are related to the golden ratio $\tau=\frac{\sqrt{5}-1}{2}$, also known as the \textit{small golden ratio}, which is a root of the polynomial $x^2 + x - 1$. This group has been extensively studied by Burillo, Nucinkis, and Reeves, who provided a finite presentation for $F_\tau$ and explored its combinatorial structure in terms of binary trees in \cite{Ttau}.

\begin{definition}
$F_\tau$ is the group of homeomorphisms from the unit interval to itself such that:
\begin{itemize}
    \item they are piecewise linear and orientation-preserving;
    \item is not differentiable at a finite number of points that belong to $\mathbb{Z}[\tau]\cap[0,1]$;
    \item where differentiable, the slope is a power of $\tau$.
\end{itemize}
Formally:
\[
F_\tau = G([0,1]; \mathbb{Z}[\tau], \langle\tau\rangle)
\]
\end{definition}
Like Thompson's group $F$, elements of $F_\tau$ can be represented by pairs of binary trees. Each binary tree encodes a subdivision of the interval into subintervals, and the leaves of the trees correspond to these subintervals. However, unlike in $F$, we must distinguish between two types of subdivision in $F_\tau$ one of length~$\tau$ and one of length $\tau^2$. This is represented in the tree diagram by carets with edges of different lengths, where the longer edge corresponds to the shorter interval and vise versa, to encode the length as a power of $\tau$. See \cite{Ttau} for the construction.

Similarly, the group $T_\tau$ is the group of piecewise-linear, orientation-preserving homeomorphisms of the circle, and the group $V_\tau$ is the group of piecewise-linear, orientation-preserving bijections of the interval, such that if they are not differantiable only at a finite number of points in $\mathbb{Z}[\tau]$, and having derivatives of the differential intervals as powers of~$\tau$. We consider two elements of $V_\tau$ to be the same if they agree everywhere but on a finite set of points. For a more detailed introduction to $F_\tau$ $V_\tau$, we refer the reader to \cite{Burillo T}.

\begin{definition}\label{def:caret_types}
    \cite[Definition 1.1]{Ttau} For the group $F_\tau$, we define two types of carets:
    \begin{itemize}
        \item an \textit{x-type caret} has a long left edge and a short right edge. It subdivides an interval~$[x,y]$ into~$[x, x+\tau^2(y-x)] \cup [x+\tau^2(y-x),y]$.
        \item a \textit{y-type caret} has a short left edge and a long right edge. It subdivides an interval~$[x,y]$ into~$[x, x+\tau(y-x)] \cup [x+\tau(y-x),y]$.
    \end{itemize}
\end{definition}

\begin{thm}~\cite[Theorem 4.4]{Ttau}
The group $F_\tau$ is generated by four key elements:~$x_0, x_1$ and~$y_0, y_1$  (see Figures \ref{fig:x_n gen1}, \ref{fig:y_n gen1} for the generators). 
\end{thm}

\begin{figure}
\centering
\includegraphics[width=0.8\linewidth]{"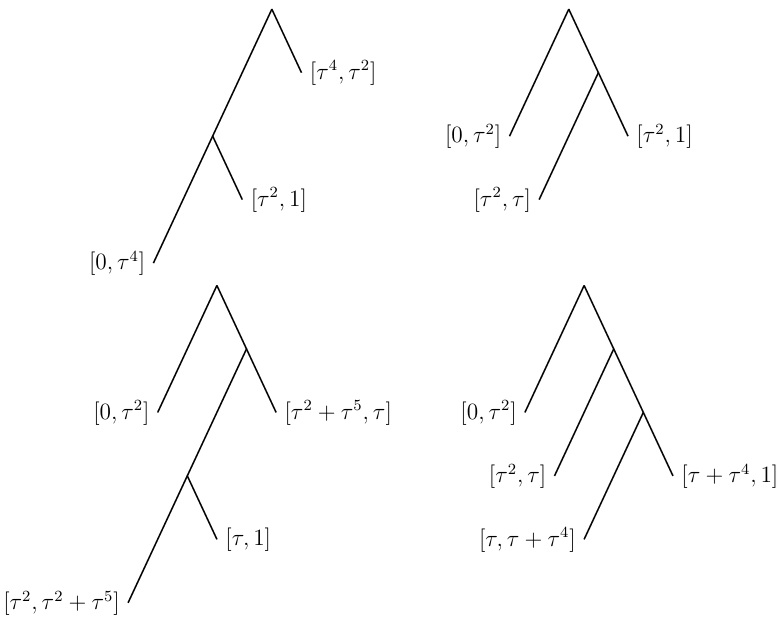"}
    \caption{The generators $x_0$ and $x_1$ as tree diagrams.}
    \label{fig:x_n gen1}
\end{figure}

\begin{figure}
\centering
\includegraphics[width=0.8\linewidth]{"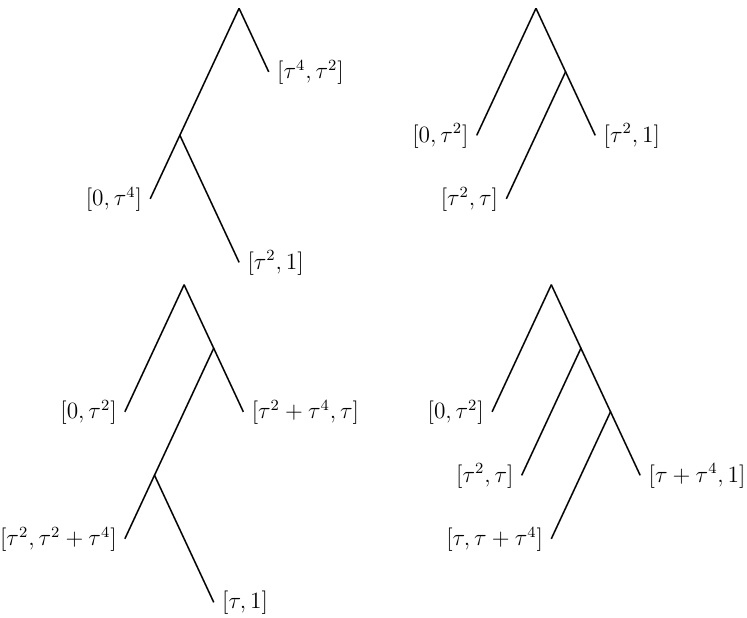"}
    \caption{The generators $y_0$ and $y_1$ as tree diagrams.}
    \label{fig:y_n gen1}
\end{figure}

\begin{thm}\label{thm t gen}
    \cite[Theorem 2.2]{Burillo T} The Thompson group $T_\tau$ is generated by the generators of the Thompson group $F_\tau$ and, in addition, $c_1$ and $c_2$. See Figure \ref{fig:c_1 c2 gen} for these generators.
\end{thm}
\begin{figure}
\centering
\includegraphics[width=0.5\linewidth]{"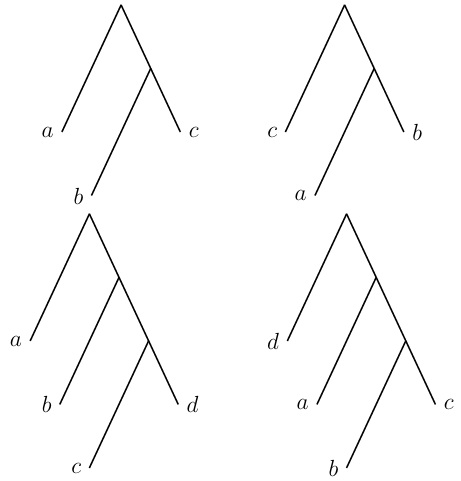"}
    \caption{The generators $c_1$ and $c_2$ as tree diagrams. The labels on the leaves determine which domain intervals map to which range intervals.}
    \label{fig:c_1 c2 gen}
\end{figure}

\begin{thm}\label{thm v gen}
    \cite[Section 4]{Burillo T} The Thompson group $V_\tau$ is generated by the generators of the Thompson group $T_\tau$ and, in addition, $\pi_0$ and $\pi_1$. See Figure \ref{fig:pi_0 pi_1 gen} for these generators.
\end{thm}
\begin{figure}
\centering
\includegraphics[width=0.5\linewidth]{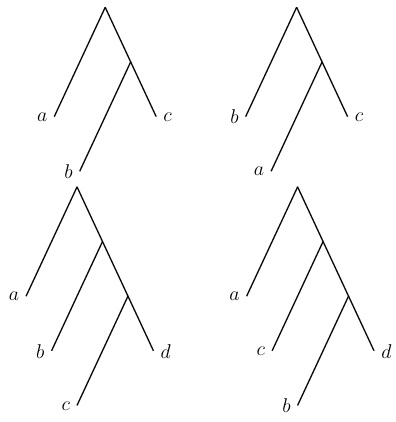}
    \caption{The generators $\pi_0$ and $\pi_1$ as tree diagrams. The labels on the leaves determine which domain intervals map to which range intervals.}
    \label{fig:pi_0 pi_1 gen}
\end{figure}

The general rule for reading a tree diagram is that the $n^{th}$ interval in the domain tree, counting from the left, is assigned to the $n^{th}$ interval in the range tree, unless specified otherwise. The length of the intervals in a tree diagram corresponds to~$\tau^n$, where~$n$ is the distance from the root to the end of the caret of the respective interval. Short carets have length 1, while long carets have length 2. Note that group $V_\tau$ has an additional generator that permutes the penultimate and ultimate intervals in the range tree.

\section[Irrational slope Thompson groups]{Irrational slope Thompson groups as an asynchronous automaton}\label{sec: Irrational slope}
In this section, we use automaton theory to explore the dynamical properties of Thompson groups~$F_\tau, T_\tau$ and~$V_\tau$. In particular, we prove that the Thompson groups~$F_\tau$,~$T_\tau$ and~$V_\tau$ are isomorphic to a certain group of rational homeomorphisms of the Cantor set. Let $X_0, X_1, Y_0, Y_1, C_1, C_2, \Pi_0$ and $\Pi_1$ be rational homeomorphisms of the Cantor set~$\{0,1\}^\omega$ defined by the automata shown in Figure \ref{fig:v_tau automaton} (which differ only in their initial states). We prove the following theorem.

\begin{thm} \label{v_tau automaton}
    The group $G$ of rational homeomorphisms of $\{0,1\}^\omega$ generated by $X_0,$ $X_1,$ $Y_0,$ $Y_1,$ $C_1,$ $C_2,$ $\Pi_0$ and $\Pi_1$ is isomorphic to $V_\tau$. Formally, $G = \langle X_0, X_1, Y_0, Y_1, C_1, C_2,$ $\Pi_0, \Pi_1 \rangle \simeq V_\tau$. See Figure \ref{fig:v_tau automaton} for the automata that generate the group $G$. 
\end{thm}
\begin{figure}
    \centering
    \includegraphics[width=0.6\linewidth]{"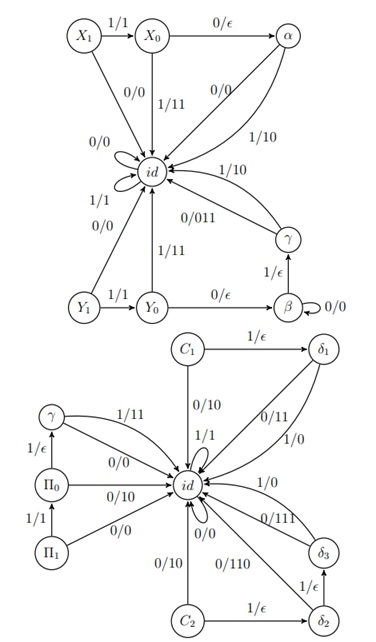"}
\caption{Automata that generate the Thompson group $V_\tau$.}
    \label{fig:v_tau automaton}
\end{figure}
\noindent Note that $G$ is a subgroup of $\mathrm{Homeo}(\{0,1\}^\omega)$. The remainder of the section is dedicated to proving this theorem.

\subsection{Preliminaries on automata theory and rational groups} We assume that the reader is familiar with  the fundamental concepts of alphabets and the languages they generate. We rely on the definitions stated in \cite[Section 2.1]{Grigorchuk Automata}.

Let $X$ be a finite set, we construct the set~$X^*$, known as the \textit{free monoid} generated by $X$. Elements of $X^*$ are finite sequences, which are called \textit{words}. The set  includes the \textit{empty word}, which we denote by $\epsilon$. The \textit{length} of a word ~$v$ is the minimum number of symbols needed to represent it and is denoted by $|v|$. The length of the empty word $\epsilon$ is 0. The set of all such infinite sequences generated by $X$ is denoted by $X^\omega$. For any $v \in X^*$ and~$u \in X^* \cup X^\omega$, we can define the concatenation (product) $vu \in X^\omega$ in a natural way. A word $v \in X^*$ is a prefix of another word~$u \in X^* \cup X^\omega$ if~$u = vy$ for some $y \in X^* \cup X^\omega$. For any given set of words $X \subseteq X^* \cup X^\omega$, there exists a unique longest common prefix of all words in~$X$, this prefix is infinite if and only if~$X$ consists of a single infinite word.

We endow the set $X$ with the discrete topology, and $X^\omega$ with the product topology. This space is homeomorphic to the Cantor set, which implies that its topological type does not depend on the choice of $X$. For any finite word~$v \in X^*$, the set $c(v) = \{vu : u \in X^\omega\}$ is open and closed in this topology. The family $\{c(v) : v \in X^*\}$ forms a basis for the topology in $X^\omega$. The sets $c(v_1)$ and $c(v_2)$ have a nonempty intersection if and only if one of $v_1$ or $v_2$ is a prefix of the other. In such a case, the set corresponding to the longer word is a subset of the other. 

\begin{definition} 
   \cite[Definition 1.1]{Belk Rational embeddings} An \textit{automaton} has the following components:
    \begin{enumerate}
        \item Two finite sets $X_{in}$ and $X_{out}$, which are called the input and output alphabet.
        \item A finite set of states $Q$, where each element represents a distinct state.
        \item An initial state $q_0$ that belongs to the set $Q$.
        \item A transition function $t$ that maps a state $q \in Q$ and an input symbol $a \in X_{in}$ to a new state $q' \in Q$, represented as $t: Q \times X_{in} \rightarrow Q$.
        \item An output function $o$ that maps a state $q \in Q$ and an input symbol $a \in X_{in}$ to a string of output symbols $s \in X_{out}$, represented as $o: Q \times X_{in} \rightarrow X^*_{out}$.
    \end{enumerate}
The automaton is classified as \textit{synchronous} if the output mapping $o$ generates a single symbol from the output alphabet $X_{out}$ for each state $q \in Q$ and input symbol $a \in X_{in}$. If this condition is not met, the automaton is called \textit{asynchronous}.

\end{definition}

An automaton can be graphically represented as a finite directed graph. Each state is depicted as a node in the graph. The transitions and outputs are illustrated by directed edges. Specifically, for each state $q \in Q$ and input symbol $a \in X_{in}$, there is a directed edge from the node representing $q$ to the node representing $t(q, a)$, labeled with $a / o(q, a)$.

Given an automaton $T = (X_{in}, X_{out}, Q, q_0, t, o)$, an \textit{input word} for $T$ is an infinite sequence $a_1a_2a_3... \in X_{in}^\omega$. The corresponding \textit{output word} is the concatenation of the outputs produced by the output mapping $o$ for each state transition, expressed as:
\[o(q_0, a_1) o(q_1, a_2) o(q_2, a_3)...,\]
where $\{q_n\}$ is the sequence of states starting from the initial state $q_0$, defined recursively as $q_n = t(q_{n-1}, a_n)$.

Note that the output word may be finite if the output function $o(q, a)$ is the empty word for all but finitely many input symbols $a$ in a given state $q$. However, our focus is on automata whose output words are always infinite. Such automata are called \textit{nondegenerate}, otherwise we call them \textit{degenerate}. A nondegenerate automaton defines a mapping from infinite input words to infinite output words over the respective alphabets.

Intuitively, an automaton can be seen as a graphical representation of a recurrent function, which reads the input word one symbol at a time, returns a symbol from $X$, and then moves to the next stage, which is uniquely determined by the current reading.   

For a detailed treatment on automata theory, we refer the reader to \cite[Section 1.1]{Belk Rational embeddings} and \cite[Section 2.1]{Grigorchuk Automata}.

\begin{definition}
\cite[Definition 1.2]{Belk Rational embeddings}A function $f: X_{\text{in}}^\omega \rightarrow X_{\text{out}}^\omega$ is called \textit{rational} if there exists a non-degenerate automaton $\mathcal{X}$ on the alphabets $X_{\text{in}}$ and $X_{\text{out}}$ such that $f(\psi) = \mathcal{X}(\psi)$ for all $\psi \in X_{\text{in}}^\omega$.
\end{definition}

Rational functions possess several key properties \cite[Proposition 1.3]{Belk Rational embeddings}:
\begin{enumerate}
    \item The rational functions are continuous with respect to the product topologies in $X_{\text{in}}^\omega$ and $X_{\text{out}}^\omega$.
    \item The composition of two rational functions is again a rational function.
    \item If $f: X_{\text{in}}^\omega \rightarrow X_{\text{out}}^\omega$ is a rational bijection, then the inverse function $f^{-1}: X_{\text{out}}^\omega \rightarrow X_{\text{in}}^\omega$ is also rational.
\end{enumerate}

\begin{definition}
\cite[Definition 1.4]{Belk Rational embeddings} Let $X$ be a finite alphabet, consisting of at least two symbols, and we define the \textit{rational group} $\mathcal{R}_X$. This group consists of all rational homeomorphisms of $X^\omega$. All rational groups~$\mathcal{R}_X$ are isomorphic to each other, we refer to a single rational group $\mathcal{R}$ without specifying the alphabet.
\end{definition}

When $G$ is a group, any injective homomorphism $G \rightarrow \RR_X$ is a rational representation of this group. The rational group $\mathcal{R}$ can be studied through an infinite rooted tree $X^*$ of finite strings on $X$. The boundary $\partial X^*$ of this tree is homeomorphic to $X^\omega$, and the action of $\mathcal{R}$ on~$X^\omega$ can be understood by analyzing the restrictions of rational functions on the subtrees of $X^*$. See section \ref{tree back} for the definitions of trees and boundaries.

It has been shown in \cite{Belk Rational embeddings} that every finitely generated group, whose Cayley graph is delta-hyperbolic, is embedded in the rational group $\RR$. 

\begin{definition}\cite[Definition 2.3]{Grigorchuk Automata}
    Let $f : X_{in}^\omega \rightarrow X_{out}^\omega$ be a continuous non-constant function, and let~$w \in X_{in}^*$ be a finite word. The \textit{local action} of $f$ in the word $w$ is denoted by~$f|_w$, is the mapping~$f|_w : X_{in}^\omega \rightarrow X_{out}^\omega$ defined as:
\[\forall u \in X_{in}^\omega\text{, } f(wu) = v f|_w(u),\]
where $v$ is the longest common prefix of the sets of words $\{f(wu) : u \in X_{in}^\omega\}$. If this set has only a single word, then the largest common prefix is infinite, in which case $f|_w$ is undefined.

In other words, the local action $f|_w$ describes how $f$ acts on infinite words that begin with the finite prefix $w$. If the output of $f$ on words with the prefix $w$ is independent of the infinite suffix after $w$, then the local action at $w$ is not defined. We do not allow $f$ to be a constant function, because a constant function may map two distinct input words to the same output word, which is not allowed.
\end{definition}

\begin{thm}
\cite[Theorem 2.5]{Grigorchuk Automata} A continuous mapping $f: X_{in}^\omega \to X_{out}^\omega$ is rational if and only if it has a finite number of local actions.
\end{thm}

\begin{definition}\cite[Definition 2.1]{Belk Rational embeddings}
    Let $\Gamma$ be a finite directed graph, then the \textit{subshift of finite type} is assigned to $\Gamma$ is the set $\Sigma_\Gamma$ of all infinite directed paths in $\Gamma$.
\end{definition}

Let $\alpha$ be a finite path in a directed graph $\Gamma$, then a \textit{cone} is the set $\CC_\alpha \subseteq \Sigma_\Gamma$ containing all infinite paths that start with the prefix $\alpha$. When $v$ is a vertex in $\Gamma$, then cone $\CC_v$ is the of all infinite directed paths that start from $v$. We denote $\CC_\emptyset$ as the set of all infinite paths in $\Gamma$. Every cone $\CC$ is a clopen subset of $\Sigma_\Gamma$ with the product topology, in addition, they form a basis for a topology. \cite[Section 2.1]{Belk Rational embeddings}

\begin{definition}\cite[Definition 2.19]{Belk Boone-Higman}
    A subshift $\Sigma_\Gamma$ has an \textit{irreducible core} if there exists an induced subgraph $\Gamma_0$ of $\Gamma$ such that the following conditions hold:
    \begin{enumerate}
        \item the graph $\Gamma_0$ is irreducible;
        \item for every vertex $v \in \Gamma_0$, there is a directed path in $\Gamma$ from $v$ to a vertex in $\Gamma_0$; 
        \item there exists $n \geq 0$ such that every directed path in $\Gamma$ of length $n$ terminates in $\Gamma_0$.
    \end{enumerate}
\end{definition}

\begin{prop}\cite[Proposition 2.12]{Belk Rational embeddings}
    Let $\Sigma_\Gamma$ be a subshift of finite type with no isolated points or empty cones, and $E\subseteq \Sigma_\Gamma$ be a nonempty clopen set. Then the set $\RR_{\Gamma,E}$  of rational homeomorphisms~$E \rightarrow E$ forms a group under composition.
\end{prop}
\begin{definition}\cite[Section 2.1]{Belk Rational embeddings} A \textit{canonical similarity} 
\[
L_\alpha: \CC_{t(\alpha)} \rightarrow \CC_\alpha
\]
is a homeomorphism defined by the formula $L_\alpha(\omega) = \alpha \cdot \omega$, where $\alpha$ is a finite path. Generally, when $\alpha$ and~$\beta$ are finite paths with $t(\alpha) = t(\beta)$, the composition $L_\beta \circ L_\alpha^{-1}$ defines the canonical similarity~$\CC_\alpha \rightarrow \CC_\beta$, which maps $\alpha \cdot \omega$ to $\beta \cdot \omega$, for all $\omega \in \CC_\alpha$.
\end{definition}

\begin{definition}\cite[Definition 2.32]{Belk Rational embeddings}
    Let $\Sigma_\Gamma$ be a subshift of finite type with no isolated points or empty cones. Let $E\subseteq \Sigma_\Gamma$ be a nonempty clopen set, and let $\RR_{\Gamma,E}$ be the associated rational group. A subgroup $G \leq \RR_{\Gamma,E}$ is called a \textit{rational similarity group} if, for every pair of cones $\CC_\alpha, \CC_\beta$ neither of which is contained within $E$ with $t(\alpha) = t(\beta)$, there exists an element $g \in G$ that maps $\CC_\alpha$ to $\CC_\beta$ through the canonical similarity.
\end{definition}

\begin{definition}\cite[Section 2.6]{Belk Boone-Higman}
    Let $f \in \RR_{X,E}$ be a non-degenerate map and let $\alpha, \beta$ be finite paths in~$X$, then $f$ has only a finite number of local actions $f{|}_\alpha$. The set of all local actions that occur infinitely many times is called the \textit{nucleus} and is denoted by $\NN_f$. Formally $\NN_f = \{{f|}_\alpha :  {f|}_\alpha = {f|}_\beta $  for infinitely many $\beta\}$. As a consequence $f{|}_\alpha \notin \NN_f$ only for a finite  number of different $\alpha$.
\end{definition}

\begin{definition}\cite[Definition 2.41]{Belk Boone-Higman}
    Let $G \leq \RR_X$, then the \textit{nucleus} $\NN_G$ of $G$ is defined as the union of all $\NN_g$ for $g\in G$. Hence, $\NN_G$ is the smallest set of functions such that for any $g \in G$, $g|_\alpha \in \NN_G$ holds for all but for a finite number of $\alpha \in X^*$. We call $G \leq \RR_X$ \textit{contracting} whenever $\NN_G$ is finite and $\Sigma_X$ has an irreducible core.
    \end{definition}

\begin{definition}\label{nuc def}\cite[Definition 2.43]{Belk Boone-Higman}
A set of maps $\NN$ is the \textit{nucleus of injections} if it satisfies the following conditions:
\begin{enumerate}
    \item The identity map belongs to $\NN$;
    \item For every map $x \in \NN$, $x|_\alpha \in \NN$.
    \item For every map $x \in \NN$ there exists $f \in \NN$ such that $x \in \NN_f$;
    \item For every map $x \in \NN$, $\NN_{x^{-1}} \subseteq \NN$;
    \item For every pair of maps $x$ and $y \in \NN$, $\NN_{xy}\subseteq \NN$.
\end{enumerate}
\end{definition}

\begin{thm}\label{thm: 2.46}\cite[Theorem 2.46]{Belk Boone-Higman} Let $\NN$ be a nucleus of injections over $\Sigma_X$ and $E \subseteq \Sigma_X$ be n nonempty clopen set, then
\[
    G=\{f \in \RR_{X,E}: \NN_f \subseteq \NN\}
\]
is a rational similarity group with nucleus $\NN$.
\end{thm}

\subsection{The group G} Let's recall the generators of the Thompson group $V_\tau$. In \cite[Section 4]{Burillo T} it has been shown that $V_\tau$ has an infinite generating set $\{x_i, y_i, c_{i+1}, \pi_i\}$, where $i \in \mathbb{Z}$, and a finite generating set $\{x_0, x_1, y_0, y_1$, $c_1, c_2, \pi_0,\pi_1\}$, see Figures \ref{fig:x_n gen1}, \ref{fig:y_n gen1}, \ref{fig:c_1 c2 gen} and \ref{fig:pi_0 pi_1 gen} for the generating elements. Our goal is to demonstrate that there exists an asynchronous automaton with initial states corresponding to the generating set $\{x_0, x_1, y_0, y_1$, $c_1, c_2, \pi_0,\pi_1\}$.

To accurately interpret the domain and ranges in trees using binary sequences, we need to establish a standardized notation for caret addresses. Let a left caret be denoted $L$ and a right caret as~$R$ to provide clarity in representation. In this context, two letters of the same kind represent a long caret, while one letter symbolizes a short caret. Consistently, we assign values 0 and 1 to $LL$ and $R$, respectively, to ensure a precise and coherent representation of caret addresses in binary sequences. 

\begin{example}
    Figure \ref{fig: LR example} demonstrates the intuition behind the notation that we use to address carets in a tree representation.
    \begin{figure} 
    \centering
        \includegraphics[width=0.37\linewidth]{"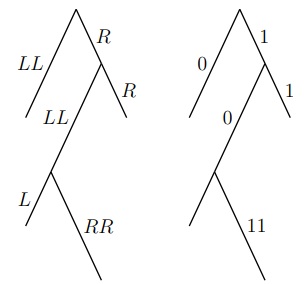"}
    \caption{Carets of the $y_1$ generator addressed by pairs $L$'s,$R$'s and $0$'s, $1$'s }
    \label{fig: LR example}
\end{figure}
    Note that a small left caret denoted by $L$ is not defined in the binary $\{0,1\}$ representation. However, this is not a problem, as we are allowed to add additional carets and perform basic moves.
    
    Here, a basic move is an interchange of an $x$-type caret followed by an extra $x$-type on the right with a $y$-type caret followed by another $y$-type caret from the left, and vise versa. See \cite{Ttau} for details. These operations can be performed, because we can split the unit interval into subintervals of length $\tau^2$, $\tau^3$ and $\tau^2$ in two different ways. First, by splitting the interval into subintervals of length $\tau$ and $\tau^2$ and then further dividing the subinterval of length $\tau$ into sequential subintervals of length $\tau^2$ and $\tau^3$. Or we can start by dividing the interval into subintervals  of length $\tau^2$ and $\tau$.
\end{example}

\begin{remark}
     Recall that the tree diagrams for the generators $x_n$ and $y_n$ are read following a sequence of carets that start from the root and, at each step, slide down the tree until a caret that does not have children. Note that a finite tree has at most $m$ such sequences, where $m$ is the total number of carets in the tree. 
\end{remark}

Our proof that the generators of $G$ are rational homeomorphisms will use the golden ratio base for numbers, where a base-$\tau$ expansion of $x \in \mathbb{R}$ is $0.\alpha_1\alpha_2\cdots$, where $\alpha_i \in \{0,1\}$ and $\sum_{i=1}^\infty \alpha_i \tau^i=x$. See \cite{golden ratio base} for a general introduction to irrational bases. Every $x \in \mathbb{R}$ has many different base $\tau$ expansions. For example, since $\tau = \tau^2 + \tau^3$, it follows that $\tau$ has expansions $0.1000\cdots$ and $0.011000\cdots$, additionally~$\tau = \tau^2 + \tau^4+\tau^5$, so $\tau$ also has expansion $0.01011000\cdots$. However, we can make the base $\tau$ expansion unique for most numbers in $[0,1]$ if we impose the rules that $\alpha_1=0$ and all finite strings of $0$'s in~$\alpha_2, \alpha_3,\ldots$ have even length. If we use these rules, then any number in $\mathbb{Z}[\tau] \cap (0,1)$ has 2 expansions and any other number in ~$[0,1]$ has a unique expansion. Applying these rules, we find that two base $\tau$ expansions represent the same $x$ if they have forms $0.\beta00111\cdots$ and $0.\beta1000\cdots$ for some finite~$\beta \in \{0,1\}^*$. For example, imposing these rules limits the possible expansions of $\tau$ to $0.1000\cdots$ and~$0.00111\cdots$.

The following propositions follow immediately from the golden-ratio base.
\begin{prop} \label{prop: q map}
    Let $q: \{0,1\}^\omega \rightarrow [0,1]$ be defined as $q = q_2q_1(\omega)$, where:
\[
\begin{aligned}
    q_1(\omega_1) &= 
    \begin{cases}
    00q_1(\omega_2), & \text{when } \omega_1 = 0\omega_2  \\ 
    1q_1(\omega_2), & \text{when } \omega_1 = 1\omega_2
    \end{cases}; \\
    q_2(\omega) &= 1 - \sum_{n=1}^\infty \tau^{n+1} (1-\omega_n) = \sum_{n=1}^\infty \tau^{n+1} \omega_n.\\
\end{aligned}
\]
\end{prop}
\noindent Note that $q$ is an order preserving surjection, where each point in $\mathbb{Z}[\tau] \cap (0,1)$ has 2 preimages and every other point in $[0,1]$ has 1 preimage; that is, $q$ is an almost one-to-one surjection. Moreover, $q$ commutes with the rational map that generates $V_\tau$. We check this for the generators $X_0$ and $x_0$ in Proposition \ref{prop q commutes for x_0}

The function $q_1$ performs a simple operation, replacing all occurrences of 0 with 00. Whereas $q_2$ mimics the behavior of the left $L$ or right $R$ operations. Observe that $L$ only changes the upper bound of the interval by a power of $\tau$, whereas $R$ only changes the lower bound. Take a finite sequence $RRR\dots$. In each sequential $R$ the interval is reduced from the left by $\tau^{n+1}$. This representation forces the lower bound to converge with the upper bound when the sequence is infinite. As a result, there are exactly two formulas that can convert the infinite sequence to a point in the interval. Either by summing up powers of $\tau$ at the positions of 1's or by subtracting from 1 the powers of $\tau$ at the 0's positions.

To illustrate how the formula works, let's consider an example. Let $\omega_{ex} = 010100\dots$
\[\begin{aligned}q(\omega_{ex})&= q(010100\dots) = q_2q_1(010100\dots) = q_2(0010010000\dots) \\&= \sum_{n=1}^\infty \tau^{n+1}(\omega_n) = \tau^4 +\tau^7.
\end{aligned}\]
The following corollary follows from Proposition \ref{prop: q map}.`
\begin{cor}\label{h commutes}
    There exists a homeomorphism $h$ such that the following diagram commutes.
\end{cor}
\begin{figure}[ht]
    \centering
    \includegraphics[width=0.25\linewidth]{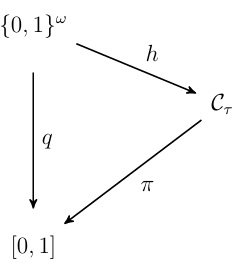}
\end{figure}
\noindent Since $q$ is order-preserving, the fiber over each $x$ in $[0,1]$ consists of 2 points if the $x$ is in $\mathbb{Z}[\tau]$ and one point otherwise.

\begin{prop}\label{prop q commutes for x_0}
    The map $q$ makes the following diagram commute, where $X_0$ is a generator of~$G$ and $x_0$ is a generator of $V_\tau$.
\begin{figure}[ht]
    \centering
    \includegraphics[width=0.4\linewidth]{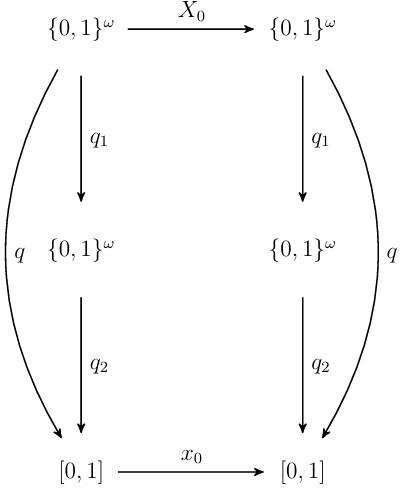}
\end{figure}
\end{prop}
\begin{proof}
We will verify $x_0q = qX_0$. Let's recall the $x_0$ generator and the $X_0$ generator, see Figures \ref{fig:x_n gen1} and \ref{fig:X_0 gen}. From reading the diagram that describes the generator $x_0$, we introduce two linear functions~$f_L$, $f_R : [0,1] \rightarrow [0,1]$, defined by: 
\[
\begin{aligned}
    f_L(t) &= \tau t;\\
    f_R(t) &= (1-\tau)+\tau t.\\
\end{aligned}
\]
\begin{figure}
    \centering
    \includegraphics[width=0.4\linewidth]{"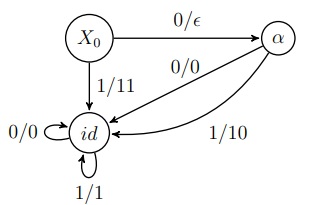"}
    \caption{The $X_0$ generator}
    \label{fig:X_0 gen}
\end{figure}
\noindent For simplicity, we denote $f_x \circ f_y = f_{xy}$. Due to being linear, these functions satisfy the following equations:
\[
\begin{aligned}
    &x_0f_{L^4}(\omega_1) = f_{L^2}(\omega_1);\\
    &x_0f_{L^2R}(\omega_1) = f_{RL^2}(\omega_1);\\
    &x_0f_R = f_{R^2}(\omega_1);\\
    &f_Lq_2(\omega_1) = q_2(0\omega_1) ;\\
    &f_Rq_2(\omega_1) = q_2(1\omega_1) .\\
\end{aligned}
\]

We distinguish 3 cases where the input word starts with 00, 01, and 1, corresponding to all possible inputs. 
\[
\begin{aligned}
    x_0q_2q_1(00\omega) &=  x_0q_2(0000\omega_2) =x_0f_{L^4}q_2(\omega_2) = f_{L^2}q_2(\omega_2)\\
    &=q_2(00\omega_2) =  q_2q_1(0\omega) = q_2q_1X_0(00\omega);\\
    x_0q_2q_1(01\omega) &=  x_0q_2(001\omega_2) =x_0f_{L^2R}q_2(\omega_2) = f_{RL^2}q_2(\omega_2)\\
    &=q_2(100\omega_2) =  q_2q_1(10\omega) = q_2q_1X_0(01\omega);\\
    x_0q_2q_1(1\omega) &=  x_0q_2(1\omega_2) =x_0f_Rq_2(\omega_2) = f_{R^2}q_2(\omega_2)\\
    &=q_2(11\omega_2) =  q_2q_1(11\omega) = q_2q_1X_0(1\omega).\\
    \end{aligned}
\]
Hence, $q$ in fact maps $X_0$ to $x_0$.
\end{proof}
\noindent It is similarly straightforward to check that this rational map acts the same on other pairs of generators.

\begin{proof}[Proof of Theorem \ref{v_tau automaton}]
    For each $g \in G$, let $\phi(g)$ be the unique element of $V_\tau$ that makes the following diagram commute:
\begin{figure}[H]
    \centering
    \includegraphics[width=0.3\linewidth]{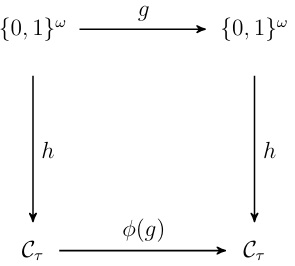}
\end{figure}

\noindent Let $v \in \{0,1\}^\omega$, then by construction it will be mapped to $[0,1]$ in two different ways $q\circ g (v)$ and $\phi(g)\circ q (v)$, thus $\phi$ is well defined. We will now show that $\phi$ is a homomorphism. Let $g_1, g_2 \in G$, then~$g_2 \circ g_1$ is also an element of $G$. Hence, by the definition of $\phi$ the following holds:
\[\begin{aligned}
    q \circ g_1 &= \phi(g_1) \circ q;\\
    q \circ g_2 &= \phi(g_2) \circ q;\\
    q \circ g_2 \circ g_1 &= \phi(g_2 \circ g_1) \circ q.\\
    \end{aligned}
\]
By manipulating the first two equations we obtain:
\[
    q\circ g_2 \circ g_1 = \phi(g_2) \circ \phi(g_1) \circ q.
\]
And thus:
\[
    \phi(g_2 \circ g_1) \circ q=\phi(g_2) \circ \phi(g_1) \circ q.
\]
Therefore, $\phi$ is a homomorphism. Figure \ref{fig: phi homomorphism} represents a diagram of this homomorphism for an arbitrary pair $g_1,g_2 \in G$.
\begin{figure}[htbp]
     \centering
     \includegraphics[width=0.5\linewidth]{"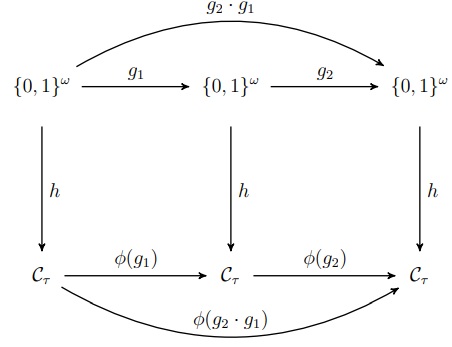"}
    \caption{A diagram representing the homomorphism of $\phi$}
    \label{fig: phi homomorphism}
\end{figure}
    Combining this with Lemma \ref{h commutes} makes $\phi$ an isomorphism from $G$ to $V_\tau$.
\end{proof}

\noindent Since by Theorems \ref{thm t gen} and \ref{thm v gen} the generators of both $T_\tau$ and $F_\tau$ are contained in $V_\tau$, we obtained the following two corollaries.
\begin{cor}
    The group of homeomorphisms of $\{0,1\}^\omega$ generated by $X_0, X_1, Y_0$ and $Y_1$ is isomorphic to $F_\tau$. See Figure \ref{fig:v_tau automaton} for the automata. 
\end{cor}
\begin{cor}
    The group of homeomorphisms of $\{0,1\}^\omega$ generated by $X_0, X_1, Y_0, Y_1, C_1$ and $C_2$ is isomorphic to $T_\tau$. See Figure \ref{fig:v_tau automaton} for the automata. 
\end{cor}
\noindent Having established the automaton representation of $V_\tau$, we can study its nucleus.

 \begin{thm}\label{nuc}
    The nucleus of the automaton that generates the Thompson group $V_\tau$ is given by the states $\{\beta, \gamma, id\}$.
\end{thm}

\begin{proof}
To identify the nucleus of the automaton that generates the Thompson group $V_\tau$, we make an educated guess based on the properties of the nucleus. It is known that every map in the nucleus must occur for infinitely many words. The only maps that potentially satisfy this condition are $\beta$, $\gamma$, $id$ since the nucleus elements must follow a cycle in the automaton diagram, see Figure \ref{fig:v_tau automaton}. Let us denote the set of these maps as $\NN = \{\beta, \gamma, id\}$.

We now proceed to verify that the set $\NN$ satisfies the conditions required  to be the nucleus of injections.

The first condition follows since the identity map $id$ belongs to $\NN$.

The second condition is satisfied, as for every map $x \in \NN$, $x{|}_\alpha \in \NN$ for all $\alpha$, by observation of Figure \ref{fig:v_tau automaton}.

The third condition is immediate since $\{\beta,\gamma,id\} \subseteq \NN_\beta$ by observation of Figure \ref{fig:v_tau automaton}.

The fourth condition is verified by directly computing the inverse maps of $\beta, \gamma,$ and $ id$, as shown in Figure \ref{fig:beta inverse automaton}. Observe that~$\beta^{-1}$ is defined for words that start with~$10$ or $0$, and~$\beta^{-1}|_{10} = id$, and~$\beta^{-1}|_{0} = \beta$. Whereas $\gamma^{-1}$ is defined only for words that start with $011$ or $10$, and~$\gamma^{-1}|_{011} = \gamma^{-1}|_{10} = id$. Therefore,~$\NN_{\beta^{-1}} = \{\beta,\gamma,id\}$,~$\NN_{\gamma^{-1}} = \{id\}$ and~$\NN_{id^{-1}} = \NN_{id}$. Thus, the fourth condition is satisfied.

\begin{figure}
    \centering
    \includegraphics[width=0.65\linewidth]{"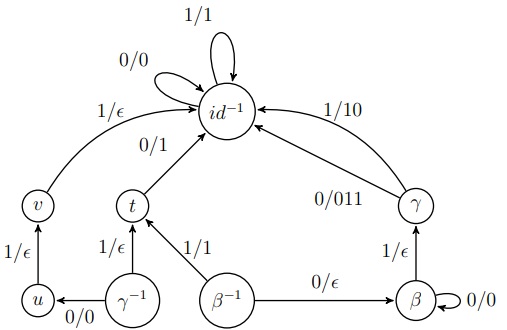"}
\caption{Automata of the the inverse maps of $\beta$, $\gamma$ and $id$ of the Thompson group $V_\tau$}
    \label{fig:beta inverse automaton}
\end{figure}

The fifth condition is the most complex, as it requires us to consider all possible compositions of maps in the nucleus. We will analyze these compositions case by case, showing that they always result in maps within our proposed nucleus. Verifying that the nucleus of a product with the identity state $id$ belongs to the set $\NN$ is trivial. Therefore, we only need to confirm that the nuclei of $\gamma^2$, $\gamma\beta$, $\beta\gamma$, and $\beta^2$ are contained in the set $\NN$. We can verify this by considering the following eight possible scenarios:
\[
\begin{aligned}
    &\gamma^2(0\omega) = \gamma(1\omega) = 0\omega;\\
    &\gamma^2(1\omega) = \gamma(0\omega) = 1\omega;\\
    &\gamma\beta(0\omega) = \gamma(0)\beta(\omega) = 1\beta(\omega);\\
    &\gamma\beta(1\omega) = \gamma^2(\omega);\\
    &\beta\gamma(0\omega) = \beta(1\omega) = 1\gamma(\omega);\\
    &\beta\gamma(1\omega) = \beta(0\omega) = 0\beta(\omega);\\
    &\beta^2(0\omega) = \beta(0)\beta(\omega) = 0\beta^2(\omega);\\
    &\beta^2(1\omega) = \beta(1)\gamma(\omega) = 1\gamma^2(\omega).\\
\end{aligned}
\]
See Figure \ref{fig:gamma square} for the automaton diagram of these maps.
Observe that the maps~$\gamma^2$ and~$\beta^2$ act as the identity. Whereas the maps~$\gamma \beta$ and~$\beta \gamma$ map all possible sequences to the maps~$\gamma$,~$\beta$ and~$\gamma^2 = id$.  Therefore, since every possible sequence is mapped to a state in the set $\NN$, the fifth condition is also satisfied.

\begin{figure}
    \centering
    \includegraphics[width=0.6\linewidth]{"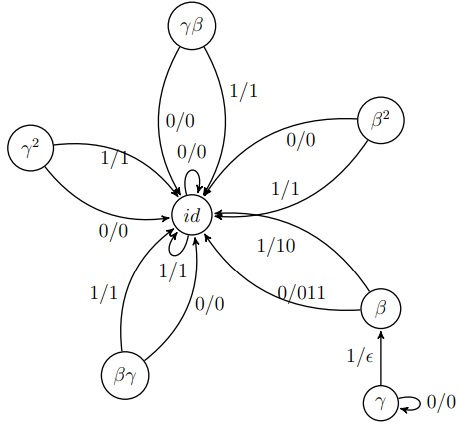"}
    \caption{The composition of states of $\NN$ }
    \label{fig:gamma square}
\end{figure}

Therefore, the set $\NN=\{\beta, \gamma, id\}$ satisfies all five conditions required to be the nucleus of the automaton that generates the Thompson group $V_\tau$.

Furthermore, given that the binary alphabet is irreducible and the nucleus $\NN$ is finite, we conclude that the nucleus is contracting.
\end{proof}

\begin{remark}
   It is not hard to see that $V_\tau$ is precisely this group $G$, in particular $V_\tau$ is a rational similarity group.
\end{remark}

We end this section by constructing the Cantor set on which $V_\tau$ acts. For a detailed introduction to the Cantor set, we refer the reader to \cite[Section 30]{Willard}.   
\begin{definition}
    Let $E$ be a countable dense subset of $(0,1)$. \textit{The blowup} of $[0,1]$ along~$E$ is the set
    \[
    ([0,1] \setminus E) \cup \{x^-: x \in E\} \cup \{x^+ :x \in E\}.
    \]
We endow this blowup with the order topology, where $a<x^-<x^+<b$ for any $x\in E$ and $a,b\in [0,1]\setminus E$ such that $a<x<b$.
\end{definition}

\begin{thm} \cite[Cantor's isomorphism theorem. Theorem 4.3]{Jech} 
Any two countable, dense, unbounded linear orders are order-isomorphic
\end{thm}

\noindent We can now prove that a $blowup$ along $[0,1]$ is an efficient tool in construction of a Cantor set. Note that we consider the Cantor set in a topological setting.
\begin{thm}
Any blowup of $[0,1]$ along a countable dense subset $E$ of $(0,1)$ is homeomorphic to the Cantor set.    
\end{thm}
\begin{proof}
  By Cantor's theorem we have an order isomorphism $f\colon E \to D$, where $D=\mathbb{Z}\bigl[\tfrac{1}{2}\bigr]\cap(0,1)$. This extends continuously to a homeomorphism $f\colon [0,1]\to [0,1]$ that maps $E$ to $D$, and it follows that the blowup of $[0,1]$ along $E$ is homeomorphic to the blowup of $[0,1]$ along $D$.  But the latter is obviously homeomorphic to the usual Cantor set. 
\end{proof}

\begin{definition}\label{def Ctau}
Let $I_\tau = \mathbb{Z}[\tau] \cap (0,1)$, where $\mathbb{Z}[\tau] = \{a + b\tau : a,b \in \mathbb{Z}\}$. Then $\CC_\tau$ is a Cantor set defined by a blowup along $I_\tau$.
\end{definition}

\noindent Let $x_i \in I_\tau$ then the Cantor set $\CC_\tau$ is depicted as:
\[
\begin{tikzpicture}[scale=0.85]
\draw (1.5,0) node[anchor=north]{}
-- (3,0) node[anchor=south]{}
    (4.5,0) node[anchor=north]{}
-- (6,0) node[anchor=north]{}
    (7.5,0) node[anchor=north]{}
-- (9,0) node[anchor=north]{}
    (10.5,0) node[anchor=north]{}
-- (12,0) node[anchor=north]{}
    (13.5,0) node[anchor=north]{}
-- (15,0) node[anchor=north]{};

\node[circle,fill=black,inner sep=0pt,minimum size=0pt,label=above:{$\cdots$}] (a) at (0.5,-0.27) {};
\node[circle,fill=black,inner sep=0pt,minimum size=0pt,label=above:{$\cdots$}] (a) at (16.5,-0.27) {};

\node[circle,fill=black,inner sep=0pt,minimum size=4pt,label=above:{$x_1^+$}] (a) at (1.5,0) {};
\node[circle,fill=black,inner sep=0pt,minimum size=4pt,label=above:{$x_2^-$}] (a) at (3,0) {};
\node[circle,fill=black,inner sep=0pt,minimum size=4pt,label=above:{$x_2^+$}] (a) at (4.5,0) {};
\node[circle,fill=black,inner sep=0pt,minimum size=4pt,label=above:{$x_3^-$}] (a) at (6,0) {};
\node[circle,fill=black,inner sep=0pt,minimum size=4pt,label=above:{$x_3^+$}] (a) at (7.5,0) {};
\node[circle,fill=black,inner sep=0pt,minimum size=4pt,label=above:{$x_4^-$}] (a) at (9,0) {};
\node[circle,fill=black,inner sep=0pt,minimum size=4pt,label=above:{$x_4^+$}] (a) at (10.5,0) {};
\node[circle,fill=black,inner sep=0pt,minimum size=4pt,label=above:{$x_5^-$}] (a) at (12,0) {};
\node[circle,fill=black,inner sep=0pt,minimum size=4pt,label=above:{$x_5^+$}] (a) at (13.5,0) {};
\node[circle,fill=black,inner sep=0pt,minimum size=4pt,label=above:{$x_6^-$}] (a) at (15,0) {};
\end{tikzpicture}
\]
Here, each interval above represents a similar Cantor set. The Thompson's group $V_\tau$ acts on this Cantor set. Specifically, if $f\in V_\tau$ maps an interval $(a,b)$ to $(c,d)$ for some  $a,b,c,d\in \Z[\tau]$, then $f$ maps $[a^+,b^-]$ homeomorphically to $[c^+,d^-]$. It can be seen as the group of all possible permutations of non-dividable intervals contained in $\CC_\tau$,

\section{The monoid \textit{M}}\label{sec: Monoid}
This section introduces the reader to the main object of study in this paper: the monoid~$M = \langle L,R : LR^2 =RL^2\rangle$. We start by building up the language necessary to talk about graphs, after which we introduce the monoid $M$ and its associated Cayley graph. We end the section by presenting the distance formula for the Cayley graph of $M$ and introducing an intriguing representation of $M$ in terms of the real line.

\subsection{Preliminaries on graphs}\label{graph back}We first state some graph background definitions as stated in \cite[Section 2]{Kong}. We will define what we mean by graph is, impose a partial order relation on it and induce a metric and a topology on it.

An \textit{undirected graph} $\Gamma =(V, E)$ consists of a countable set $V$, elements of which are called \textit{vertices}, and a symmetric subset~$E \subseteq \big\{ \{x,y\} \text{ : } x,y \in V  \text{ and } x \neq y \big\}$, whose elements are unordered pairs and are called \textit{edges}. We represent edges as pairs of vertices, we do not allow loops, and we allow at most one edge between any given pair of vertices. The graph is said to be \textit{locally finite} if for every vertex~$x \in V$,~$\deg(x) := \#\{y : \{x, y\} \in E\} < \infty$, in other words, if every vertex emits a finite number of edges. A path of length $n$ is a sequence of edges $\{x_0,x_1\},\{x_1,x_2\}\dots\{x_{n-1},x_{n}\}$, where paths of length~$0$ are vertices.

A \textit{geodesic} is a shortest path between two vertices. For $x, y \in V$, a geodesic from $x$ to~$y$ is denoted by $\pi(x, y)$. The distance between $x$ and $y$ is the length of a geodesic and is denoted by $d(x, y)$. In case no geodesic exists, then $d(x, y) = \infty$. If for all $x, y \in V$, $d(x, y)$ is finite, then the graph $\Gamma = (V, E)$ is \textit{connected}; in this case $d$ is a word metric on~$V$.

A \textit{root} of a connected graph is a fixed vertex from which there exists at least one geodesic for every other vertex in the graph. When a graph $(V, E)$ is a locally finite connected graph and there exists a vertex $r \in V$ that acts as a fixed root, then the triple $(V, E, r)$ is called a \textit{rooted graph}. From now on, we reserve the variable $r$ for the root of a given graph. Note that for a given connected graph, any vertex can be chosen as its root. For $x, r \in V$ we denote $|x| := d(x, r)$  and let $V_n := \{x \in V : |x| = n\}$. Then~$V = \bigcup_{n=0}^\infty V_n$. We introduce a partial order relation $\preceq$ on $V$ with $y \preceq x$ if and only if $y$ belongs to some~$\pi(r, x)$.

Let $m \geq 0$ and $x \in V$, then
\[J_m(x) := \{y \in V : x \preceq y, |y| = |x| + m\}\]
is called the \textit{$m$-th descendant set}, and
\[\quad J_{-m}(x) := \{z \in V : x \in J_m(z)\}\]
is called the \textit{$m$-th predecessor set} of $x$, respectively. In principle, we allow $J_m(x)$ to be empty.

\subsection{The monoid \textit{M} and its Cayley graph}
Let $M = \langle L,R : LR^2= RL^2 \rangle$ be the monoid generated by 2 elements $L$ and $R$ with the action being multiplication by the right with the relation $LR^2=RL^2$. The Cayley graph of $M$ is an undirected locally finite rooted graph denoted by $Cay(M) = (M,E,1)$, where:
\begin{itemize}
    \item the vertex set $M$ is in bijection with the set of elements of the monoid $M$, i.e., every vertex in $Cay(M)$ represents a unique element in $M$; 
    
    \item the edge set $E$ consists of all pairs of vertices $x,y \in M$ such that $x =yL$ or $x=yR$, i.e. there exists an edge between two vertices if one of their corresponding elements in $M$ can be obtained from the second one by right multiplication by $L$ or $R$;

    \item the root of the graph is the identity element 1 of the monoid $M$.
\end{itemize}
\noindent In the case of the monoid  $M$, the defining relation $LR^2 = RL^2$ imposes specific self-similarity constraints on the structure of $Cay(M)$. For a visual representation of $Cay(M)$ refer to Figure \ref{fig: MM}. Note that the picture does not accurately represent the bottom layer; we expect the reader to imagine a fractal structure at the boundary of the graph. 

\begin{figure}
    \centering
    \includegraphics[width=0.89\linewidth]{"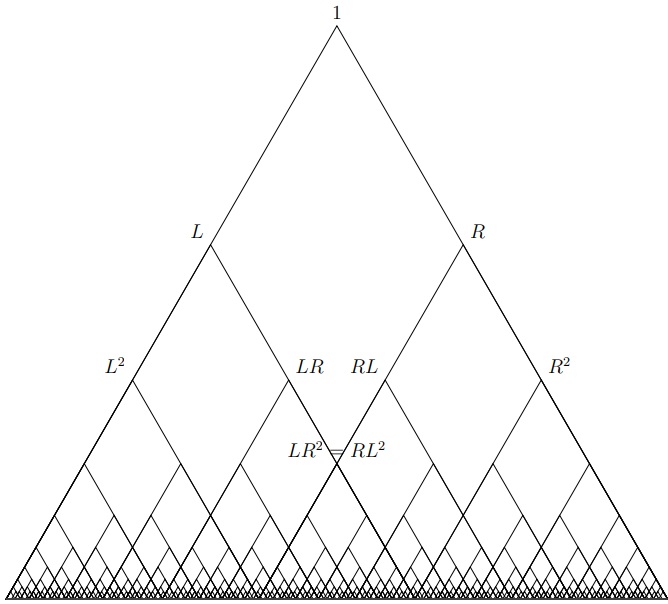"}
    \caption{The Cayley graph $Cay(M)$ of the monoid $M$ with some labeled elements}
    \label{fig: MM}
\end{figure}

\begin{definition}
    For any pair of vertices $x, y \in M$ in the graph $Cay(M) = (M, E, 1)$, a \textit{code} for a path between $x$ and $y$ is a word $w = w_1w_2\cdots w_n$, where each $w_i \in {L, R, L^{-1}, R^{-1}}$, with $L$ and~$R$ being the generators of the monoid $M$, and $L^{-1}$ and $R^{-1}$ being their formal inverses. The code for a path $w$ from~$x$ to~$y$ satisfies the equation $xw = xw_1w_2\cdots w_n = y$ in the free group generated by $L$ and $R$. The length of $w$ is defined as $n$ corresponding to the number of elements in the word $w$. We denote a geodesic code between $x$ and $y$ by $c(x,y)$.
\end{definition}   
\noindent Although $L^{-1}$ and $R^{-1}$ are used in the representation, they do not exist in the monoid $M$ itself, but rather represent backtracking in the Cayley graph. In the actual graph $Cay(M)$, each step of the path corresponds to an edge, which is always a right multiplication by $L$, $R$, $L^{-1}$ or $R^{-1}$ with inverses indicating moving ``backwards'' along an edge $L$ or $R$, respectively. 

The \textit{cone} of an element $x \in M$, denoted by $Cone(x)$, is the set of all elements in $M$ that can be obtained by multiplying $x$ on the right by any element of the monoid $M$, formally, 
\[
Cone(x) = \{x \cdot m : m \in M\},
\]
in other words, $m \in M$ is in the cone of $x$ when $x$ is a prefix of $m$.
\begin{prop}
    If $x,y \in M$ then $Cone(x) \subseteq Cone(y)$ if and only if $x \in Cone(y)$.
\end{prop}

\begin{proof}
    We will start from the forward condition; let $x, y \in M$ be such that~$Cone(x) \subseteq Cone(y)$. Then~$x \in Cone(x) \subseteq Cone(y)$.

    For the opposite direction, let $x \in Cone(y)$. Then there exists $m \in M$ such that $x = ym$, then~$Cone(x) = Cone(ym) \subseteq Cone(y)$.
\end{proof}
\noindent We end this section with a small lemma, which is clear from Figure \ref{fig: MM}.
\begin{lemma}
    $Cone(L) \cap Cone(R) = Cone(LR^2) =Cone(RL^2).$
\end{lemma}

\subsection{The distance formula for the graph \textit{Cay(M)}}
In this section, we introduce the distance formula for the graph $Cay(M)$. We start by proving that every cone in $Cay(M)$ is geodesically convex, and then through a number of lemmas we introduce a geodesic formula for any given pair of elements in $Cay(M)$. The main purpose of this section is to prove the following theorem, which will be used later in multiple sections.

\begin{thm}\label{M distance}
    Let $x,y $ be any given pair of vertices in the graph $Cay(M)$. Then the distance between $x$ and $y$ is given  by:
\[
d(x,y) =\begin{cases}
        d(x',y'), & \begin{tabular}[t]{@{}l@{}} 
        $\textit{when } x=mx', y=my' \textit{ for some nontrivial } m\in M$\\
\end{tabular}\\
        |x|+|y| -2, & \begin{tabular}[t]{@{}l@{}} 
        $\text{when } x \in Cone(LR), y \in Cone(RL) \textit{ and } x,y \notin Cone(LR^2);$\\
\end{tabular}\\
        |x|+|y|,  & \begin{tabular}[t]{@{}l@{}} 
        else.\\
\end{tabular} 
    \end{cases}
\]
\end{thm}

\noindent This formula offers a straightforward approach to calculating the distance between any given pair of vertices $x,y$ in $M$. There are a total of three steps to follow. The first step is to remove the common prefix, resulting in the obtaining $x'$ and $y'$. The next step is to identify the first pair of prefixes of $x'$ and~$y'$. Finally, if $x' \in Cone(LR)$ and $y' \in Cone(RL)$, then the distance is $|x'| + |y'| - 2$; if not, it is simply~$|x'| + |y'|$.

We will prove this theorem by a number of lemmas throughout this section.
\begin{definition}
Let $S \subseteq X$ be a subset of a metric space $X$. The set $S$ is said to be \textit{geodesically convex} if for any two elements $x, y \in S$, there exists at least one geodesic joining $x$ and $y$ that is contained entirely in $S$.

In addition, when every geodesic that joins any two elements $x, y \in S$ is completely contained in $S$, the set is called \textit{strongly geodesically convex}.
\end{definition}

\begin{prop} \label{monoid convex}
    Each cone of $Cay(M)$ is strongly geodesically convex.
\end{prop}

\begin{proof}
\begin{figure}
    \centering
    \includegraphics[width=0.89\linewidth]{"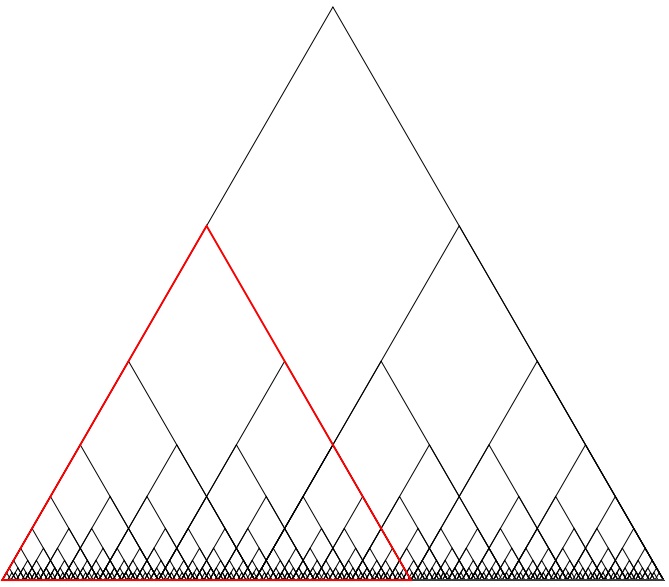"}
    \caption{$Cone(L)$ of the graph $Cay(M)$}
    \label{fig:LM}
\end{figure}

Let $x,y \in Cone(L)$ (represented by the red triangle in Figure \ref{fig:LM}). Let's assume, for the sake of contradiction, that a geodesic $\pi(x,y)$ at some point leaves $Cone(L)$. Then $\pi(x,y)$ can be decomposed into 3 subgeodesics $\pi(x,y) = \pi(x,LR^m)\cdot \pi(LR^m,LR^n)\cdot \pi(LR^n,y)$. We will focus on the second subgeodesic $\pi(LR^m,LR^n)$, as it represents the portion of the path that is supposed to leave $Cone(L)$. 

We will now show that the unique code $c(LR^m, LR^n)$ is $R^{m-n}$ (without loss of generality let $m>n $). Any deviation from the direct ascending path, whose code is $R^{m-n}$, would introduce additional steps, making the path longer. Therefore, $R^{m-n}$ is the unique geodesic code from~$LR^m$ to~$LR^n$. Note that it lies entirely within $Cone(L)$. This contradicts the assumption that a geodesic $\pi(LR^m, LR^n)$ leaves $Cone(L)$. Hence, $Cone(L)$ is stronlgy geodesically convex.

Due to symmetry, we apply the same argument to $Cone(R)$ by interchanging the roles of $L$ and $R$.

Observe that $Cay(M)$ is isometric to $Cone(L)$ and $Cone(R)$. Now, since~$Cone(LL)$ and $Cone(LR)$ are strongly geodesically convex in $Cone(L)$ they are also strongly geodesically convex in~$Cay(M)$. By induction, every cone in $Cone(L)$ and $Cone(R)$ is strongly geodesically convex. 

Suppose $Cone(m)$ is strongly geodesically convex, then $Cone(m)$ is isometric to $Cay(M)$ itself. Since $Cone(L)$ and $Cone(R)$ are strongly geodesically convex in $Cay(M)$, then $Cone(mL)$ and~$Cone(mR)$ are strongly geodesically convex in $Cone(m)$. Thus, $Cone(mL)$ and $Cone(mR)$ are strongly geodesically convex in $Cay(M)$. Hence, every~$Cone(m) \in Cay(M)$ is strongly geodesically convex.
\end{proof}

\begin{prop} \label{boundary path}
    Let $n,m \geq 0$, then the unique geodesic code of $\pi(L^n,R^m)$ is~$L^{-n}R^m$.
\end{prop}
    See Figure \ref{fig:LRjM} for an illustration of the unique geodesic $\pi(L^n,R^m)$ in blue and $Cone(LR^2)$ in red. This visualization will aid in understanding the subsequent proof by contradiction.

\begin{figure}
    \centering
    \includegraphics[width=0.89\linewidth]{"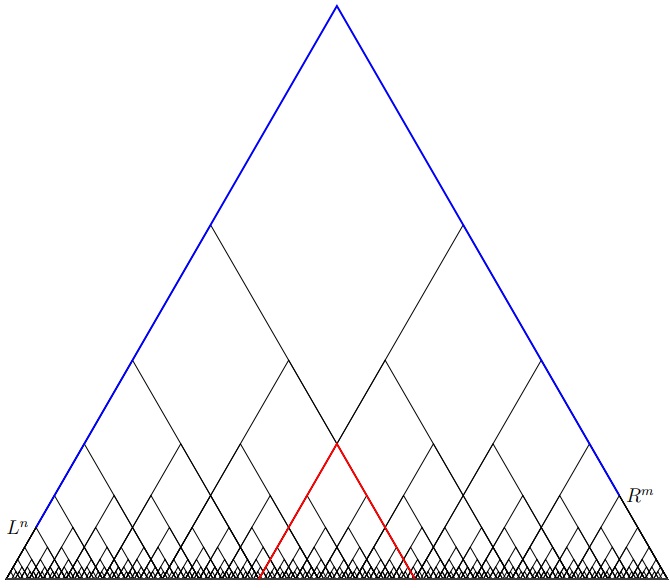"}
    \caption{The graph $Cay(M)$ with the $L^{-n}R^m$ code and $Cone(LR^2)$ highlighted}
    \label{fig:LRjM}
\end{figure}
\begin{proof}

We will prove this by induction on $k$. For $k\geq0$, let $n,m \geq 0$, then $n+m\leq k$ implies~$d(L^n,R^m) = n+m$. Let $k \leq 1$, then the statement holds trivially. Now let's assume that the statement holds for $k-1$. We will now prove it for $k$. 

Suppose, for contradiction, that $L^{-n}R^m$ is not the geodesic code $\pi(L^n,R^m)$, and does not visit the root $1$. Then this geodesic must pass through $Cone(LR^2)$. In this case, a code $c(L^n,R^m)$ can be decomposed into 2 subcodes:
    \[
    c(L^n,R^m) = c(L^n,RL^{2+a}) \cdot c(RL^{2+a}, R^m),
    \]
where $a \leq n-2$. If $a \geq n-2$, we will end-up with a code longer than $L^{-n}R^m$, contradicting the assumption that $c(L^n,R^m)$ is not the geodesic code.

Let's look at the second subcode $c(RL^{2+a}, R^m)$. Due to the convex property of the monoid $M$, a code $c(RL^{2+a}, R^m)$ cannot leave $Cone(R)$, therefore the code $c(RL^{2+a}, R^m)$ is the same as $c(L^{2+a}, R^{m-1})$. 

Note that $(a +2) +(m-1) \leq n+m - 1 = k-1$. Then, according to the induction hypothesis, the length of a code $c(RL^{2+a}, R^m)$ is $a+m+1$. The code $c(L^{-2-a}R^{m-1})$ is of such length, this implies that a code~$c(L^n,R^m)$ visits $LR^2$.

Given that a code $c(L^n,R^m)$ visits $LR^2$, we observe that it can be decomposed as:
\[
c(L^n,R^m) = c(L^n, LR^2) \cdot c(LR^2, R^m).
\]

Consider the subcode $c(L^n, LR^2)$. From the convex property of $Cay(M)$, we observe that $c(L^n, LR^2) = c(L^{n-1}, R^2)$. Since~$n-1+2 < k$, according to the inductive hypothesis,~$d(L^{n-1}, R^2)  = n+1$.

This gives $d(L^n, LR^2) = n+1$ and similarly, $d(LR^2, R^m) = m+1$. Hence, 
\[
d(L^n,R^m) = (n+1)+(m+1) = n+m+2.
\]
But this is greater than $n+m$, which is the length of the code $c(L^{-n}R^m)$. This contradicts the assumption that $\pi(L^n,R^m)$ does not pass through the root 1. 

It is left to prove uniqueness. Observe that $c(L^n,R^m) = c(L^n,1) \cdot c(1,R^m)$, where both subgcodes have unique ascending and descending respective codes. Therefore, the unique code for the geodesic $\pi(L^n,R^m)$ is~$c(L^{-n}R^m)$.
\end{proof}

\begin{prop}\label{x R path}
    Let $x \in \big(Cone(L)\setminus Cone(LR)\big)$, then $d(x,R^n) = |x|+n$.
\end{prop}

\begin{proof}
The distance $d(x,R^n)$ depends on a geodesic $\pi(x,R^n)$, which in its half has two possibilities: it visits the root 1 or it does not. If it does, then we are trivially done. If $n=0$ we are also trivially done, let $n>0$.

For the purpose of contradiction, let's assume that there exists a geodesic $\pi(x,R)$ that does not visit the root and has length at most $|x|+n-1$. If such a geodesic exists, then at some point it must enter and leave $Cone(LR)$. Then Proposition \ref{boundary path} implies that it must visit~$LR$.

We can split the assumed geodesic into two subgeodesics in the following way:
\[
\pi(x,R) = \pi(x,LR) \cdot \pi(LR,R^n).
\]
Following Proposition \ref{boundary path}, we observe that $d(LR,R^n) = n+2$ has two possible geodesics, one that travels through the root and the other that goes through $LR^2$. Then $d(x,R^n) = d(x,LR) +d(LR,R^n) = d(x,LR) +n+2 \leq |x|+n-1$ implies that $d(x,LR) \leq |x|-3$. 

But this means that
\[
    d(x,1) =d(x,LR)+d(LR,1) \leq |x|-3+2 = |x|-1.
\]
This is a contradiction, as there is no path from $x$ to the root of length $|x|-1$. Therefore, if $x \in \big(Cone(L) \setminus Cone(LR)\big)$, then every geodesic $\pi(x,R^n)$ must pass through the root, consequently,~$d(x,R^n) = |x| + n$.
\end{proof}

We now have enough tools to introduce a geodesic code formula $c(x,y)$ for any pair of vertices $x,y \in Cay(M)$. Before we formally state it, we will explain how it works. This is done in a number of steps:
\begin{enumerate}
    \item identify the smallest $Cone(m)$ such that $x,y \in Cone(m)$, note that $m$ can be the empty word;
    \item remove the common prefix $m$ from $x$ and $y$, obtaining elements $x'$ and $y'$, respectively. Without loss of generality, let $x'$ start with $L$ and $y'$ start with $R$;
    \item if $x' \in Cone(LR)$ and~$y' \in Cone(RL)$, then one of the geodesics from~$x$ to~$y$ starts from~$x$, takes the ascending path to~$mLR$, then through~$mLR^2$ travels to~$mRL$, and finally takes the descending path to~$y$. We denote the code of such geodesic as~$c(x,y) = (x'^{-1}RL)(R^{-2}L^{-1}y)$;
    \item if $x' \notin Cone(LR)$ or $y' \notin Cone(RL)$, then there may exist multiple geodesics from~$x$ to~$y$. One of these takes the ascending path from~$x$ to the root of~$Cone(m)$ (the vertex~$m$) and then takes the descending path to~$y$. We denote the code of such geodesic as~$c(x,y) = x'^{-1} y'$.
\end{enumerate}

\begin{prop} \label{M geo}
   Let $x,y $ be any given pair of vertices in the graph $Cay(M)$, and let $x'$ and $y'$ be elements obtained from $x$ and $y$ by removing their common prefix $m$. The geodesic code between~$x$ and~$y$ is given by the following equation:
   \[
    c(x,y) = \begin{cases}
    c(x',y'), & \begin{tabular}[t]{@{}l@{}}
$\text{if } x,y \in Cone(m);$\\
\end{tabular} \\
    (x^{-1}RL^2) (R^{-2}L^{-1}y), & \begin{tabular}[t]{@{}l@{}}
$\text{if }x \in Cone(A_2) \text{ and } y \in Cone(A_4);$\\
\end{tabular} \\
    x^{-1}y, & \begin{tabular}[t]{@{}l@{}}
$\text{else.}$\\
\end{tabular}
   \end{cases}
   \]
   where $Cone(m) \neq M$, $Cone(A_2) = Cone(LR) \setminus Cone(LR^2)$ and $Cone(A_4) = $ $Cone(RL) \setminus Cone(LR^2)$.
\end{prop}
    
\noindent Before proving Proposition \ref{M geo} we invite the reader to view Figure \ref{fig:a1-5}. This figure demonstrates 5 subsets of vertices of the graph $Cay(M)$.
    \[\begin{aligned}
    &A_1=Cone(L) \setminus Cone(LR);\\
    &A_2 =Cone(LR) \setminus Cone(LR^2);\\
    &A_3 =Cone(LR^2);\\
    &A_4 =Cone(RL) \setminus Cone(LR^2);\\ 
    &A_5 =Cone(R) \setminus Cone(RL).\\
    \end{aligned}\]
    Essentially, if $x'$ and $y'$ are in $Cone(m)$, then a geodesic will not pass through the root of $Cone(m)$ only if both $x'$ and $y'$ belong to the areas $A_2$ and $A_4$, respectively. If only one of them belongs to its respective area, then a geodesic has the choice of passing through the root or not. In case none of them are, then every geodesic passes through the root $m$. 

\begin{figure}
    \centering
    \includegraphics[width=0.89\linewidth]{"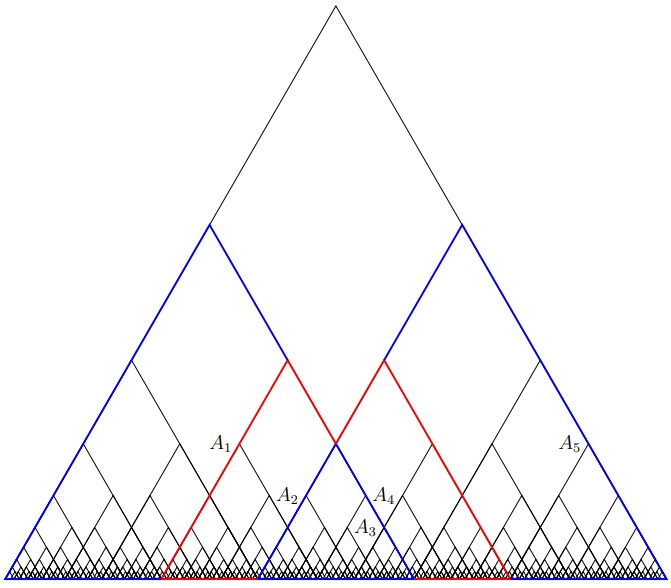"}
    \caption{The areas $A_1$ - $A_5$ on the Cayley graph $M$}
    \label{fig:a1-5}
\end{figure}

\begin{proof}
Suppose $x,y$ are in $Cone(m)$, where $m$ is not the root, then the first case follows immediately from Proposition~\ref{monoid convex}. 

Now suppose $x$ and $y$ are not in any nontrivial $Cone(m)$. Clearly~$x$ and~$y$ cannot belong to $Cone(LR^2)$. Without loss of generality, let~$x$ start with~$L$ and~$y$ start with~$R$. Let's assume that a geodesic $\pi(x,y)$ passes through $Cone(LR^2)$. From Proposition~\ref{boundary path} it is known that such a geodesic must travel along the boundary of~$Cone(LR^2)$. Hence, if~$\pi(x,y)$ passes though~$Cone(LR^2)$ it must visit~$LR^2$. 

Therefore, to pass from $Cone(L)$ to $Cone(R)$ a geodesic $\pi(x,y)$ must visit $1$ or $LR^2$. This implies that
\[
\pi(x,y) = min \{\pi (x ,1) \cdot \pi(1,y), \text{  } \pi(x, LR^2) \cdot \pi (LR^2, y)\}.
\]
We only need to find the codes for $\pi(x ,1)$ and $\pi (x, LR^2)$, since the codes for $y$ will be analogous cases.

Let's start with the geodesic that goes to the root, then
\[
c(x ,1) = x^{-1}\text{ and }d(x ,1)= |x|.
\] 
We now analyze the second possible geodesic $\pi(x, LR^2)$. We will split this into 2 cases, when $x \in \big(Cone(LR) \setminus Cone(LR^2)\big)$ (area $A_2$ of Figure \ref{fig:a1-5}) and when $x \in \big(Cone(L) \setminus Cone(LR)\big)$ (area $A_1$). Then by Proposition \ref{x R path}:
\[
\pi(x, LR^2) = \begin{cases}
    |x|-1, \text{ when } x \in Cone(LR);\\
    |x|+1, \text{ else.}
\end{cases}
\]
Due to symmetry, we obtain the following:
\[
\pi(y, LR^2) = \begin{cases}
    |y|-1, \text{ when } y \in Cone(RL);\\
    |y|+1, \text{ else.}
\end{cases}
\]
This is enough to conclude the following three statements.
\begin{enumerate}
    \item If $x \in \big(Cone(LR) \setminus Cone(LR^2)\big)$ and $y \in \big(Cone(RL) \setminus Cone(LR^2)\big)$, then a geodesic $\pi(x,y)$ passes through $LR^2$ and has length $|x| + |y| - 2$. One of the possible codes for a geodesic is $(x^{-1}RL^2)(R^{-2}L^{-1}y)$;
    \item If only one of the elements is in $Cone(LR) \setminus Cone(LR^2)$ or $Cone(RL) \setminus Cone(LR^2)$ respectively, then a geodesic has a choice of visiting either the root or $LR^2$, both resulting in the same length. One of the possible codes for this geodesic is $(x^{-1}y)$
    \item When neither of the elements is in these regions, a geodesic $\pi(x,y)$ will always pass through the root.
\end{enumerate}

Recall from Proposition \ref{monoid convex} that every cone in $Cay(M)$ is strongly geodesically convex. Hence, combining this with the initial reduction to the smallest common $Cone(m)$ concludes the proof.
\end{proof}

\begin{proof}[Proof of Theorem \ref{M distance}]
    Following Proposition \ref{M geo} we  take the length of a code between $x$ and $y$.
\end{proof}

Let $x,y$ be a pair of vertices in $Cay(M)$ and let $w \in \{L,R\}^*$. Observe that multiplying~$x$ and~$y$ from the left by $w$ does not alter their distance. This means that the left multiplication is an isometry. We end this section with a small but important consequence of Theorem~\ref{M distance}.

\begin{cor}
    The monoid $M$ is left cancellative.
\end{cor}

\subsection{A representation of the monoid \textit{M} in terms of the real line}This section introduces a set of subintervals $I_M$ that will allow us to study the properties of $M$ in terms of the real line. In particular, for every element of $M$ we associate a closed subinterval from the unit interval of the real line.

\begin{definition}
Let $I_c$ be the set of all closed intervals $[x,y] \subseteq [0,1]$, such that $x,y \in \mathbb{R}$ and $x < y$.
\end{definition}

If $[x,y] \subseteq I_C$ then we define the subinterval as follows:
\begin{equation}\label{eq: intervals}
\begin{aligned}
[x,y]_1 & \coloneq [x,y];\\
[x,y]_L & \coloneq [x, y - (y-x)\tau^2];\\
[x,y]_R & \coloneq [x + (y-x)\tau^2, y].\\
\end{aligned}
\end{equation}
Recall that $\tau$ is the positive root of the equation $x+x^2=1$, approximately 0.618034.

If $w = w_1w_2\cdots w_n$ is a word in $\{L,R\}^*$, we define $[x,y]_w$ recursively by $[x,y]_w= ([x,y]_{w_1w_2\cdots w_{n-1}})_{w_n}$.

\begin{definition}\label{def:I_M}
Let $w = w_1w_2\cdots w_n$ be a word in $\{L,R\}^*$, then we define $I_M \subseteq I_c$ as the collection of all subintervals $[0,1]_w$.
\end{definition}

\begin{lemma} 
Let $[x,y] \in I_M$, then the following equality holds:
\[
[x,y]_{LR^2} = [x,y]_{RL^2}.
\]
\end{lemma}

\begin{proof}
We will verify this rule by direct computation:
\[\begin{aligned}
[x,y]_{LR^2} &= [x, x\tau^2 + y - y\tau^2]_{R^2} \\
&= [x - x\tau^2 + x\tau^4 + y\tau^2 -  y\tau^4, x\tau^2 + y -  y\tau^2]_R \\
&= [x-2x\tau^2+3x\tau^4-x\tau^6+2y\tau^2-3y\tau^4+y\tau^6, x\tau^2 +y-y\tau^2]\\
&= [x - x\tau^2 +  y\tau^2, x\tau^2 +y-y\tau^2];
\end{aligned}\]

\[\begin{aligned}
[x,y]_{RL^2} &= [x - x\tau^2 + y\tau^2, y]_{L^2} \\
&= [x - x\tau^2 + y\tau^2,  x\tau^2 - x\tau^4 + y - y\tau^2 + y\tau^4]_L \\
&= [x - x\tau^2 +  y\tau^2, 2x\tau^2 - 3x\tau^4 + x\tau^6 + y - 2y\tau^2 + 3y\tau^4 - y\tau^6]\\
&= [x - x\tau^2 +  y\tau^2, x\tau^2 +y-y\tau^2].
\end{aligned}\]
\noindent The final simplification in both cases is based on the property $\tau = \tau^2+\tau^3$. Both computations give the same result, thus proving the lemma.
\end{proof}

\begin{lemma}\label{lem:I_M_homeomorphic_M}
If $x,y \in M$ then $Cone(x) \subseteq Cone (y)$ if and only if $[0,1]_x \subseteq [0,1]_y$. See Figure \ref{fig: I_m} for an illustration.
\end{lemma}

\begin{proof}    
We will start by proving the forward direction. Let $x, y \in M$ be such that $Cone(x) \subseteq Cone(y)$. Then there exists $m \in M$ such that $x = ym$. Therefore,
\[
[0,1]_x = [0,1]_{ym} = ([0,1]_y)_m \subseteq [0,1]_y.
\]

We will prove the opposite direction by induction on the length of $x$. Let $x,y \in M$ be such that~$|y| \leq |x| \leq n$, then $[0,1]_x \subseteq [0,1]_y$ implies $Cone(x) \in Cone(y)$.

Let's verify the base case for $|x| = 1$. Then $Cone(y)$ is equal to $Cone(x)$ or the whole monoid $Cone(1)$. Hence, for $x,y \in M$ such that $|y|\leq|x| = 1$, $Cone(x) \subseteq Cone(y)$ implies that $y$ being a prefix of $x$ or the empty word. This forces $[0,1]_x \subseteq [0,1]_y$. 

Assume that the statement holds for all $x$ such that $|x| \leq n$. Now, let $|x| = n+1$. We point out the key property of $I_M$: 
\[
[0,1]_L \cap [0,1]_R = [0,1]_{LR^2} = [0,1]_{RL^2}.
\]
Without loss of generality, let $x$ and $y$ have no common prefix, let $x$ start with~$R$, and~$y$ start with~$L$. Then $[0,1]_{Rx'} \subseteq [0,1]_{Ly'}$ for some $x', y'$. This implies $[0,1]_{Rx'} \subseteq [0,1]_L$, which is only possible if $x'$ starts with $L^2$. Therefore,
\[
[0,1]_{Rx'} \subseteq [0,1]_{RL^2} = [0,1]_{LR^2} \subseteq [0,1]_L.
\]
Hence, $[0,1]_{x'} \subseteq [0,1]_{R^2} \subseteq [0,1]_{R}$. By the induction hypothesis,~$Cone(x') \subseteq Cone(R)$, this forces $Cone(x) \subseteq Cone(RL^2) = Cone(LR^2) \subseteq Cone (L)$. But this means that $x$ and $y$ have a common prefix that can be canceled. Then by the induction hypothesis~$Cone(x_2x_3\cdots x_n) \subseteq Cone(y_2y_3\cdots y_n)$, which forces~$Cone(x) \subseteq Cone(y)$.
\end{proof}

\begin{figure}
    \centering
    \includegraphics[width=0.89\linewidth]{"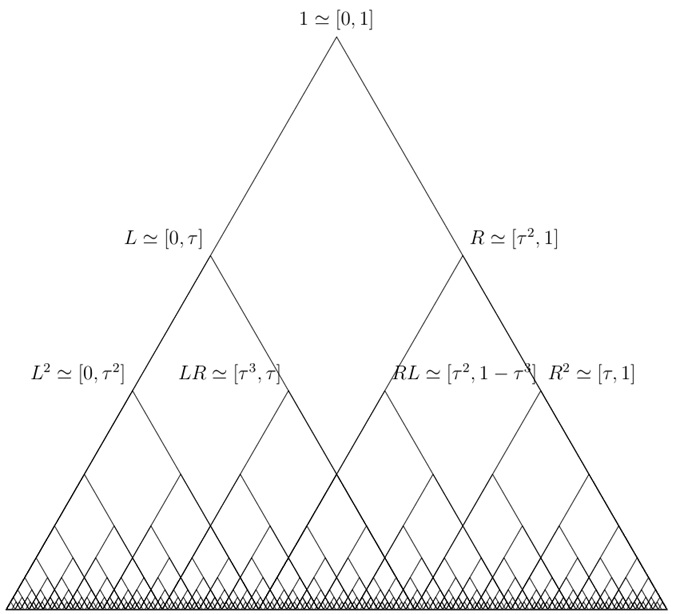"}
    \caption{The Cayley graph $Cay(M)$ with labeled  elements of $I_M$}
    \label{fig: I_m}
\end{figure}
\begin{lemma}
    Let $x, y, z \in M$. Then $[0,1]_x \cap [0,1]_y = [0,1]_z$ if and only if $Cone(x) \cap Cone(y) = Cone(z)$.
\end{lemma}

\begin{proof}
    We will first prove the forward condition. Let $[0,1]_x \cap [0,1]_y = [0,1]_z$, then there exists $m, n \in M$ such that $[0,1]_{xm} = [0,1]_z = [0,1]_{yn}$. This implies:
\[
        [0,1]_z = [0,1]_{xm} \subseteq [0,1]_x \quad \text{and} \quad [0,1]_z = [0,1]_{yn} \subseteq [0,1]_y
\]
    By Lemma \ref{lem:I_M_homeomorphic_M}, this is equivalent to:
\[
        Cone(z) \subseteq Cone(x) \quad \text{and} \quad Cone(z) \subseteq Cone(y)
\]
    Therefore, $Cone(z) \subseteq Cone(x) \cap Cone(y)$.

    To show the reverse inclusion, let $w \in Cone(x) \cap Cone(y)$. Then $[0,1]_w \subseteq [0,1]_x \cap [0,1]_y = [0,1]_z$, which implies $w \in Cone(z)$. Thus, $Cone(x) \cap Cone(y) \subseteq Cone(z)$. Now, combining both inclusions, we conclude $Cone(x) \cap Cone(y) = Cone(z)$.

    For the opposite direction, let $Cone(x) \cap Cone(y) = Cone(z)$. Then there exist $m, n \in M$ such that~$Cone(xm) = Cone(z) = Cone(yn)$. This implies:
\[
        Cone(z) \subseteq Cone(x) \quad \text{and} \quad Cone(z) \subseteq Cone(y)
\]
    By Lemma \ref{lem:I_M_homeomorphic_M}, this is equivalent to:
    \[
        [0,1]_z \subseteq [0,1]_x \quad \text{and} \quad [0,1]_z \subseteq [0,1]_y
    \]
    Therefore, $[0,1]_z \subseteq [0,1]_x \cap [0,1]_y$.

    It is left to check the reverse inclusion, let $t \in [0,1]_x \cap [0,1]_y$. Then by Lemma \ref{lem:I_M_homeomorphic_M} the element of $M$ corresponding to $t$ is in $Cone(x) \cap Cone(y) = Cone(z)$, which implies~$t \in [0,1]_z$. Thus, $[0,1]_x \cap [0,1]_y \subseteq [0,1]_z$. By combining both inclusions, we conclude~$[0,1]_x \cap [0,1]_y = [0,1]_z$.
\end{proof}

\section{The horofunction boundary of \textit{Cay(M)}}\label{horofunction boundary section}
This section presents one of the main results of this paper: the horofunction boundary of $\MM$. We will start by briefly introducing the reader to the geometric concept of horofunction boundaries, introduced by Gromov in \cite{Gromov}. This will create a solid foundation for research of new properties of the monoid $M$.

\subsection{Preliminaries on horofunction boundaries}\label{tree back}
We briefly restate some modified definitions stated in \cite[Section 1]{Belk Rational embeddings} and introduce an alternative approach to finding the horofunction boundary using equivalence classes based on the paper of Belk, Bleak, and Matucci \cite[Section 1.3]{Belk Rational embeddings}.

    Let $\Gamma = (V,E)$ be a locally finite connected graph. We impose the path metric on~$V$. We define~$F(V,\mathbb{Z})$ as the abelian group of all integer-valued functions on~$V$. Let~$\overline{F}(V,\mathbb{Z})$ be the quotient of~$F(V,\mathbb{Z})$ by the subgroup of constant functions. This means that within~$\overline{F}(V,\mathbb{Z})$ two functions~$f$ and~$g$ from~$F(V,\mathbb{Z})$ are equivalent if their difference~$f-g$ is a constant function. This construction allows us to focus on the variations of functions over~$V$, where~$\overline{F}$ denotes the equivalence classes of~$F$. If~$f \in F (V, \mathbb{Z})$ we let~$\overline{f}$ denote its image on~$\overline{F}(V, \mathbb{Z})$.

Observe that $F(V,\mathbb{Z})=\mathbb{Z}^V$ is a topological space under the product
topology. Consequently, $\overline{F}(V,\mathbb{Z})$, being a quotient of $F(V,\mathbb{Z})$, naturally inherits a quotient topology.

\begin{definition}
    \cite[Definition 1.22]{Belk Rational embeddings} Let $\Gamma=(V,E)$ be a locally finite connected graph. Let~$x \in V$ then for all $y \in V$ the corresponding \textit{distance function} $d_x: V \rightarrow \mathbb{Z}$ is defined by:
    \[
        d_x(y) = d(x, y).
    \]
The function $i: V \rightarrow \overline{F}(V,\mathbb{Z})$ defined by 
    \[ i(x)=\overline{d}_x\]
    for all $x\in V$, is called the \textit{canonical embedding}.

\end{definition}

\begin{definition}
    \cite[Definition 1.23]{Belk Rational embeddings} Let ${\partial}_hV$ be the set of accumulation points of $i(V)$ in~$\overline{F}(V,\mathbb{Z})$, this set is called the \textit{horofunction boundary} of $V$.

    We call a function $f:V \rightarrow \mathbb{Z}$ a \textit{horofunction} when $\overline{f} \in {\partial}_hV$.
\end{definition}
The  horofunction boundary $\partial_h V$ is equipped with the topology of pointwise convergence. Specifically, for a sequence of functions $\overline{f}_n \in \overline{F}(V,\mathbb{Z})$ and a function $f \in \overline{F}(V,\mathbb{Z})$, we say $f_n \to f$ if and only if for every $v \in V$, $f_n(v) \to f(v)$ as $n \to \infty$. This topology can be alternatively described as the subspace topology inherited from the product topology on $\mathbb{Z}^V$, where $\mathbb{Z}$ is given the discrete topology, additionally this gives us the quotient topology on $\overline{F}(V,\mathbb{Z})$.
   
The definition of the horofunction boundary given above, while mathematically precise, may not provide intuitive insight into its structure. Following \cite{Belk Rational embeddings}, we will adopt an alternative characterization of the horofunctions introduced by Belk, Bleak, and Matucci. This approach uses combinatorial methods that are more aligned with the graph-theoretic context presented in this paper and provides a less complex realization of horofunctions.

A \textit{vector field} on a locally connected graph $\Gamma = (V,E)$ with respect to a chosen distance function $d$ is a function $F: E\rightarrow E\cup E_{\pm}$, where $E_{\pm} = \big\{(x,y) \text{ : } \{x,y\} \in E \big\}$, such that for all pairs $x,y \in E$:
\[
F(\{x,y\}) =\begin{cases}
\{x,y\} & \text{if } d(x) = d(y);\\
(x,y) & \text{if } d(x) = d(y);\\
(y,x) & \text{if } d(y) > d(x);
\end{cases}
\]
where $\{x,y\}$ denotes an unordered pair of vertices, and $(x,y)$ denotes an ordered pair. 

Essentially, the graph $\Gamma_F=\big(V,F(E)\big)$ has the same set of vertices as $\Gamma$, however, now we allow some or all edges to be directed.

For any vertex $x \in V$, we can define a \textit{principal vector field} $F_x$ such that for all pairs or vertices~$y,z \in E$ the following hold:
\[F_x(\{y, z\}) = \begin{cases}
\{y,z\} & \text{if } d_x(y) = d_x(z);\\(y,z) & \text{if } d_x(z) < d_x(y); \\
(z,y) & \text{if } d_x(z) > d_x(y). \\
\end{cases}\]
In a principal vector field $F_x$ all oriented edges point along geodesics towards the given vertex $x$. Essentially $F_x$ depicts the distance function $d_x$ from all vertices to $x$. For simplicity, we refer to $\Gamma_F=\big(V,F_x(E)\big)$ as the principal vector field $F_x$, since it is rather simple to depict vectors on a graph than to formally express all relations.

Let's consider a simple example to illustrate how the principal vector fields are constructed.
\begin{example}
    Let $\Gamma=(V,E)$ be the following graph with a fixed vertex $x$:
\begin{figure}[H]
    \centering
    \includegraphics[width=0.25\linewidth]{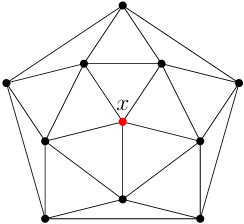}
\end{figure}
Then the principal vector field $F_x$ corresponding to the vertex $x$, is the following:
\begin{figure}[H]
    \centering
    \includegraphics[width=0.25\linewidth]{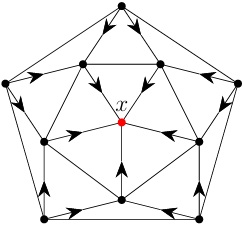}
\end{figure}
\end{example}

\begin{definition}
Fix $B \subseteq E$ a subset of edges. We say $x \underset{B}{\sim} y$ if $F_x |_B = F_y |_B$.
Note that $\underset{B}{\sim}$ is an equivalence relation. We denote this equivalence class by $[x]_B$. 
\end{definition}
\noindent A set $B_n = \{x \in V : d_r(x) \leq n\}$ is called a \textit{ball}, where $r$ denotes the root of the graph. The corresponding equivalence class is denoted by $\AA(B_n)$ and is called an \textit{atom} \cite[Definition 3.1]{Belk Rational embeddings}.  

\begin{example}
Let $\Gamma = (V,E,r)$ be a locally finite connected rooted graph and $B_0 \subseteq E$. Then its 0-level atom $\AA(B_0)$ is the whole graph $V$ itself.
\end{example}

We call a graph $\Gamma=(V,E)$ a \textit{tree} whenever there exists a unique geodesic for every pair of vertices in $V$, consecutively $\Gamma =(V,E,r)$ is called a \textit{rooted tree} when this condition is met. 

A \textit{geodesic ray} is an infinite sequence of edges in which every finite subsequence of edges acts as a geodesic for the corresponding pair of vertices; in other words, it is an infinite sequence of geodesics~$\pi(x_0, x_1), \pi(x_1,x_2), \dots$ such that for any $n, m \geq 0$ the sequence 
\[
\pi(x_n, x_{n+1}), \pi(x_{n+1},  x_{n+2}), \dots, \pi(x_{m-2}, x_{m-1}), \pi(x_{m-1} x_m)\] 
acts as a geodesic from $x_n$ to $x_m$.

Let $\Gamma =(V,E,r)$ be a rooted tree, then the set of equivalence classes of infinite geodesic rays is called the \textit{boundary} of a tree. 

\begin{definition}\cite[Definition 3.4]{Belk Rational embeddings} 
    For every $n\geq 0$, we denote $\AA_n(V)$ as the set of \textit{infinite atoms} in~$\AA(B_n)$, i.e, an \textit{atom} is in $\AA_n(V)$ if and only if its cardinality is infinite. The disjoint union
    \[\AA(V) = \bigsqcup^\infty_{n=0} \AA_n(V)\]
    is called the \textit{tree of atoms} of $V$.
\end{definition}
The tree of atoms $\mathcal{A}(V)$ is equipped with the following topologies:
\begin{itemize}
    \item for each $n \geq 0$, $\mathcal{A}_n(V)$ has the discrete topology;
    \item the topology on $\mathcal{A}(V) = \bigsqcup^\infty_{n=0} \mathcal{A}_n(V)$ is the disjoint union topology;
    \item the boundary of the tree of atoms $\partial \mathcal{A}(V)$ inherits the subspace topology of the product topology on $\mathcal{A}(V)^{\mathbb{N}}$, where the elements of the boundary are identified with infinite descending paths in the tree.
\end{itemize}

Note that since $\AA(V)$ is a disjoint union, each element in $\AA(V)$ can be represented as an ordered pair $(n, A)$, where $n \geq 0$ and $A \in \AA_n(V)$. This distinction is important, because it is possible for the same atom $A$ to be in $\AA_n(V)$ for multiple values of $n$. Observe that in the tree of atoms there exists an edge from $(n,A)$ to $(n+1,A')$ if and only if $A' \subseteq A$. 

\begin{definition}
    \cite[Definition 3.7]{Belk Rational embeddings} A \textit{morphism} from $[x]_{B_m}$ to $[y]_{B_n}$ is a bijection~$\phi: [x]_{B_m} \to [y]_{B_n}$ such that:
    \begin{enumerate}
        \item if $u,v \in [x]_{B_m}$ and $p \in \{L,R\}^*$, then $u =pv$ if and only if $\phi(u) = p\phi(v)$;
        \item for all $z\in [x]_{B_m}$ and all $k\geq 0$,~$\phi([z]_{B_{m+k}}) = [\phi(z)]_{B_{n+k}}$.
    \end{enumerate} 
\end{definition}
\noindent This definition of morphism is a weaker modification than the one presented in \cite{Belk Rational embeddings}, as we have to use arbitrary graph isomorphisms since there are no group elements.

The corresponding morphism is a pair of isomorphisms $\phi_1: [x]_{B_m} \to [y]_{B_n}$ and $\phi_2: [y]_{B_n} \to [x]_{B_n}$. We say that $[x]_{B_m} \underset{T}{\sim} [y]_{B_n}$ have the \textit{same equivalence type} if such a morphism exists. We will briefly check that $\underset{T}{\sim}$ is an equivalence relation. 
\begin{proof}
We will first check reflexivity. Let $[x]_{B_m} \underset{T}{\sim} [x]_{B_m}$, we choose $\phi: [x]_{B_m} \to [x]_{B_m}$ to be the identity map:
\begin{itemize}
    \item  let $u,v \in [x]_{B_m}$ and $p \in \{L,R\}^*$, then $u = pv \iff\phi(u) = u = pv = p\phi(v)$;
    \item for all $z \in  [x]_{B_m}$ and $k \geq 0$, $\phi([z]_{B_{m+k}}) = [z]_{B_{m+k}} = [\phi(z)]_{B_{m+k}}$. 
\end{itemize}
Hence, $ \underset{T}{\sim}$ is reflexive.

We will now check symmetry. Let $[x]_{B_m} \underset{T}{\sim} [y]_{B_n}$, then $\phi: [x]_{B_m} \to [y]_{B_n}$ is a morphism, we will show that $\phi^{-1}: [y]_{B_n} \to [x]_{B_m}$ is also a morphism:
\begin{itemize}
    \item let $u,v \in [y]_{B_n}$ and $p \in E$, then $u = pv \iff \phi\big(\phi^{-1}(u)\big) = p\phi\big(\phi^{-1}(v)\big) \iff \phi^{-1}(u) = p\phi^{-1}(v)$;
    \item for all $z \in [y]_{B_n}$ and $k \geq 0$, $\phi^{-1}([z]{B_{n+k}}) = \phi^{-1}\Big([\phi\big(\phi^{-1}(z)\big)]_{B_{n+k}}\Big) = [\phi^{-1}(z)]_{B_{m+k}}$.
\end{itemize}
Hence, $ \underset{T}{\sim}$ is symmetric.

It is left to check transitivity. Let $[x]_{B_m} \underset{T}{\sim} [y]_{B_n}$ and $[y]_{B_n} \underset{T}{\sim} [z]_{B_r}$, then~$\phi: [x]_{B_m} \to [y]_{B_n}$ and $\psi: [y]_{B_n} \to [z]_{B_r}$ are morphisms. We will now show that $\psi \cdot \phi: [x]_{B_m} \to [z]_{B_r}$ is also a morphism:
\begin{itemize}
    \item let $u,v \in [x]_{B_m}$ and $p \in \{L,R\}^*$, then $u = pv \iff \phi(u) = p\phi(v) \iff \psi\big(\phi(u)\big) = p\psi\big(\phi(v)\big) \iff (\psi \cdot \phi)(u) = p(\psi \cdot \phi)(v)$;
    \item for all $z \in [x]_{B_m}$ and $k \geq 0$, $(\psi \cdot \phi)([z]_{B_{m+k}}) = \psi\big(\phi([z]_{B_{m+k}})\big) = \psi\big([\phi(z)]_{B_{n+k}}\big) = [\psi\big(\phi(z)\big)]_{B_{r+k}} = [(\psi \cdot \phi)(z)]_{B_{r+k}}$.
\end{itemize}
Hence, $\underset{T}{\sim}$ is transitive. Since all conditions are met, we conclude that $\underset{T}{\sim}$ is indeed an equivalence relation.
\end{proof}

\begin{thm}\label{equivalence boundary = horofunction boundary}
    \cite[Theorem 3.6]{Belk Rational embeddings} The boundary of the tree of atoms $\AA(V)$ is homeomorphic to the horofunction boundary ${\partial}_h V$ of $V$.
\end{thm}
\noindent This theorem plays a key role in Section \ref{horofunction boundary section}. Not only does it provide the necessary tools to find the horofunction boundary using a combinatorial approach, but it also allows us to visualize the horofunction boundary.

A finite directed \textit{multigraph} is a graph $\Gamma = (V,E)$, satisfying the conditions of a finite directed graph with the only difference that it allows multiple edges between the same pair of vertices. The corresponding \textit{path language} $\LL(\gamma,r) \subseteq E^*$ is the set of all finite paths $e_1e_2\cdots e_n \in \Gamma$ that begin with a fixed vertex $r$, since the graph is directed, we do not allow a path a vertex $v$ to move along edges that enter $v$. The set $\LL(\gamma,r)$ has a natural structure of a locally finite tree, where the root acts as the empty path $\epsilon$, and the boundary $\partial \LL(\gamma,r)$ is the set of all infinite paths in $\Gamma$ starting from the root $r$  \cite[Section 2.1]{Belk Rational embeddings}.

Consider a short example, which demonstrates the intuition behind the path language.
\begin{example}
    Figure \ref{fig: langauge} demonstrates such a tree. 
\begin{figure}
    \centering
    \includegraphics[width=0.8\linewidth]{"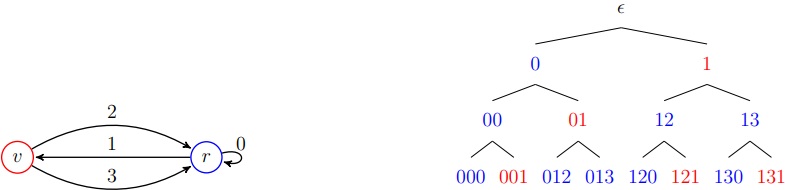"}
    \caption{A directed multigraph $\Gamma$ and its corresponding path language tree $\LL (\Gamma,r)$}
    \label{fig: langauge}
\end{figure}

A \textit{self-similar} structure \cite[Section 2.1]{Belk Rational embeddings} on $\LL(\Gamma,r)$ can be defined by the following statements:
\begin{itemize}
    \item if two paths $p,q \in \LL(\Gamma,r)$ end up in the same vertex of $\Gamma$, then we say that they have the same type;
    \item let two paths $p,q \in \LL(\Gamma,r)$ end up in the same vertex $v$, then for every finite directed path $r$ staring from $v$, we define a single morphism $\gamma_{pq} : \LL(\Gamma,r)_p \rightarrow \LL(\Gamma,r)_q$ by 
    \[
    \gamma_{pq} (pr) =qr.
    \]
\noindent We call $\gamma_{pq}$ a \textit{prefix replacement morphism}.
\end{itemize}
\end{example}

According to \cite[Proposition 2.21]{Belk Rational embeddings}, every self-similar tree $T$ is isomorphic to a path language. We define the
\textit{type graph} $\Gamma$ of a self-similar tree $T$ as follows:
\begin{itemize}
    \item for every vertex type in $T$ there is exactly one vertex in $\Gamma$;
    \item for every pair of vertices $x,y \in T$, the number of edges in $\Gamma$ between $x$ and $y$ corresponds to the number of children of type $y$ that every vertex of type $x$ has in~$T$.
\end{itemize}

By \cite[Proposition 2.21]{Belk Rational embeddings}, the tree of atoms $\AA (V)$ is isomorphic to the set of all finite directed paths starting from the root. Hence, the horofunction boundary $\partial_h V$ is naturally homeomorphic to the space of all infinite paths that start from the root in its corresponding type graph.  

\subsection{The relationship between cones and distances in \textit{Cay(M)}} This section introduces the tools required to study the atoms of $\MM$. In particular, we will introduce the connection between the cones of $\MM$ and the distance functions. Although intuitively the structure of the cones of $\MM$ seems complicated, by the end of the section we will demonstrate that the structure is understood through 4 statements.

\begin{lemma}\label{field distance eq}

    Let $x \in M$, then the following hold:
    \begin{enumerate}
        \item $x \in Cone (L) \Leftrightarrow d(x,L) < d(x,1)$;
        \item $x \in Cone (R) \Leftrightarrow d(x,R) < d(x,1)$;
        \item $x \notin Cone (L) \Leftrightarrow d(x,L) > d(x,1)$;
        \item $x \notin Cone (R) \Leftrightarrow d(x,R) > d(x,1)$;
        \item $x \in Cone (LR^2) \Leftrightarrow d(x,R), d(x,L) < d(x,1)$;
        \item $x \in \big(Cone (LR) \cup Cone(RL)\big) \Leftrightarrow d(x,LR) < d(x,L) \Leftrightarrow d(x,RL) < d(x,R)$.
    \end{enumerate}
\end{lemma}

\begin{proof}
    We begin by proving the first equivalent statement.
    \begin{enumerate}
    \item We will start from the forward direction. Let $x \in Cone(L)$. By the definition of the graph $\mathcal{M}$, $d(x,1) = |x|$. Since $x \in Cone(L)$, there exists $x' \in M$ such that~$x = Lx'$. For $d(x,L)$, we apply Proposition \ref{monoid convex} that states that every cone of $\MM$ is convex, and Theorem \ref{M distance} which tells us that~$d(x,L) = d(x',1) = |x'|$. Note that~$|x'| < |x|$ because $x = Lx'$. Hence, $x \in Cone (L) \Rightarrow d(x,L) < d(x,1)$.

    For the opposite direction, let $x \in M$ be such that~$d(x,L) < d(x,1)$. Then $d(x,1) = |x|$, which implies that~$d(x,L) \leq |x| -1$. By Theorem~\ref{M distance} this is possible only when~$x \in Cone(L)$.

    \item The second equivalent statement is proven analogously to the first by interchanging the roles of $L$ and $R$.
    
    \item We will show that for all $x \in M$ $d(x,1) \neq d(x,L)$. Observe that $d(x,1) = |x|$. Now by Theorem~\ref{M distance} $d(x,L)$ is $|x|-1$ or $|x|+1$ depending on $x$. Hence, the statement $x \notin Cone (L) \Leftrightarrow d(x,L) > d(x,1)$ is the contrapositive of the first equivalence and therefore follows directly from it.

    \item The fourth equivalent statement is a symmetric case of the third.

    \item The fifth equivalent statement is a direct consequence of the first two. 

    For the forward direction, let $x\in M$ be such that $x\in Cone(LR^2)$. Then $x \in Cone(LR^2) \subseteq Cone(L)$ and $x \in Cone(LR^2) = Cone(RL^2) \subseteq Cone(R)$. This implies that $x \in Cone(L)$ and $x \in Cone (R)$. Now we apply equivalences 1 and 2 to conclude that $d(x,L) < d(x,1)$ and $d(x,R) < d(x,1)$.
    
    For the opposite direction, let $x\in M$ be such that $d(x,L) < d(x,1)$ and $d(x,R) < d(x,1)$. Then by equivalences 1 and 2, $x \in Cone (L)$ and $x \in Cone(R)$. Hence,~$x \in Cone(L) \cap Cone(R) = Cone(LR^2)$, thus proving the equivalence.
    
    \item The last equivalent statement follows from Theorem \ref{M distance}. We will prove $x \in \big(Cone (LR) \cup Cone(RL)\big) \Leftrightarrow d(x,LR) < d(x,L).$ As the last equivalence is an analogous case of the second one.

    We will start with the forward direction. Let $x \in \big(Cone(LR) \cup Cone(RL)\big)$. We have two cases that correspond to the first symbol of $x$. If $x$ begins with $L$, we are trivially done. Let $x$ start with~$R$, then by Theorem \ref{M distance} $d(x,LR) = |x|-2$, while $d(x,L) = |x|-1$.

    We will now prove the opposite direction. Let $d(x,LR) < d(x,L)$, then following the proof of Proposition \ref{M geo}, we observe that the set of points for which this inequality holds is precisely~$Cone(LR) \cup Cone(RL)$.
    \end{enumerate}
\end{proof}

\begin{definition}
    A \textit{hyperedge} is a line that splits a space into 2 disjoint \textit{half-planes}.
\end{definition}
In the context of the graph $\MM$, where the space is the set of vertices $M$, a hyperedge selects a specific subset of vertices. To visualize this concept, see Figure \ref{fig: hyper}, where the hyperedge, depicted in red, divides the set of vertices into two parts:
\begin{itemize}
    \item the left half satisfies the first equivalent statement of Lemma \ref{field distance eq};
    \item the right half satisfies the negation of this statement, i.e. the third statement of Lemma \ref{field distance eq}.
\end{itemize}
\noindent The orientation of a vector field from $L$ to $1$ indicates the direction towards a half-plane where the vertices are closer to $1$ than to vertex $L$. Note that every hyperedge crosses at least 2 edges. Whenever a vector field is imposed on one edge that is crossed by a hyperedge, it automatically predetermines the orientations of all other edges that this hyperedge crosses. See Figure \ref{fig: hyper vect} for a visualization of a vector-field imposed on a hyperedge that satisfies statement 2 of Lemma \ref{field distance eq}.

\begin{figure}
    \centering
    \includegraphics[width=0.89\linewidth]{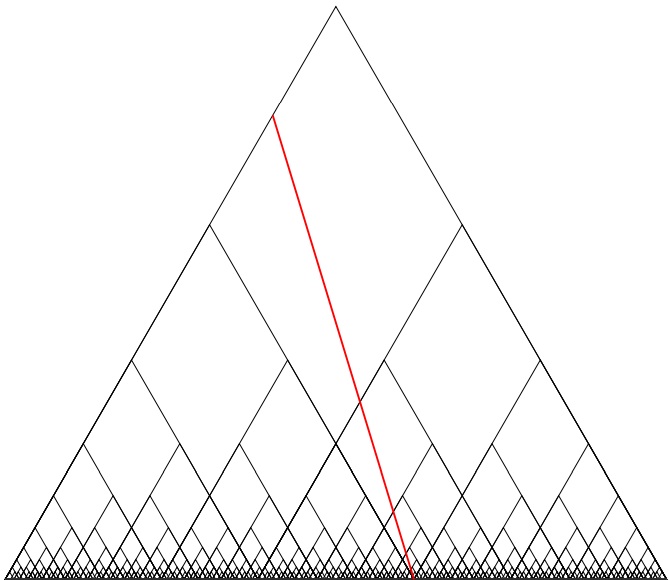}
    \caption{A hyperedge in red splitting $Cone(L)$ and $M\backslash Cone(L)$}
    \label{fig: hyper}
\end{figure}

\begin{figure}
    \centering
    \includegraphics[width=0.89\linewidth]{"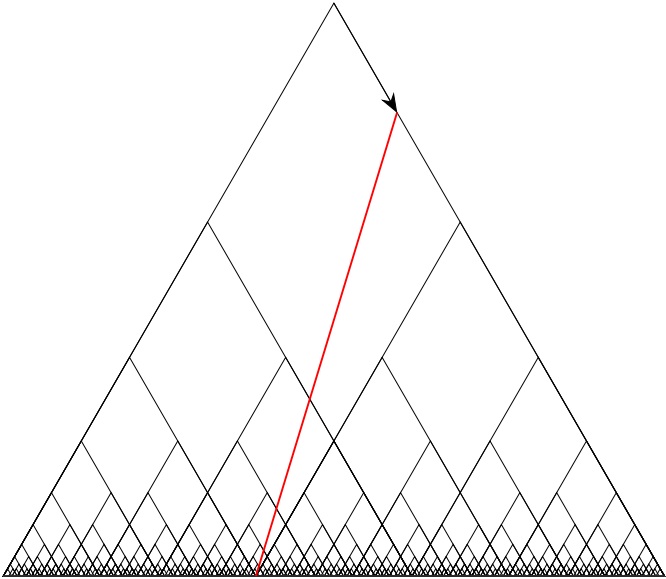"}
    \caption{A vector field pointing towards the half-plane, whose elements are closer to $R$ than to $1$}
    \label{fig: hyper vect}
\end{figure}

\begin{lemma}\label{cone mlr = d mlr}
    Let $x,m \in M$. Then $ x \in \big(Cone(mLR) \cup Cone(mRL)\big) \Leftrightarrow d(x,mLR) < d(x,mL)\Leftrightarrow d(x,mRL) < d(x,mR)$. See Figure \ref{fig:mlr mrl} for visualization.
\end{lemma}

\begin{figure}
    \centering
    \includegraphics[width=0.89\linewidth]{"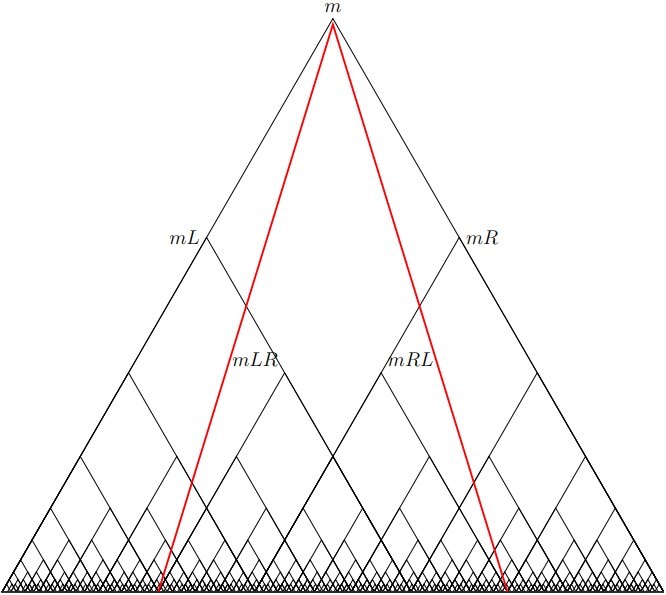"}
    \caption{A hyperedge showing $Cone(mLR) \cup Cone(mRL)$}
    \label{fig:mlr mrl}
\end{figure}

\begin{proof}
    We will start from the forward direction. Let $x \in \big(Cone(mLR) \cup Cone(mRL)\big)$, then by Proposition \ref{monoid convex} we know that $d(x,mLR) = d(x',LR)$ and $d(x,mL) = d(x',L)$, where $x = mx'$. Now by Theorem \ref{M distance} $d(x',L) < d(x',LR)$. An analogous argument is applied to prove $d(x',R) < d(x',RL)$.

    Now let $d(x,mLR) < d(x,mL)$, now assume that there exists $x \notin Cone(m)$ that satisfies this condition. Then there exists a geodesic from $mL$ to $x$ that passes through $mR$ but by Proposition \ref{M geo} this is not possible as a geodesic first takes the ascending path until it reaches the root $n$, $nLR$ or $nRL$ of the smallest $Cone(n)$ containing both $mL$ and $x$.

    Hence, the condition $d(x,mLR) < d(x,mL)$ only makes sense when $x \in Cone(m)$. Now we remove the common prefix $m$ and apply part 6 of Lemma \ref{field distance eq}. Hence, $d(x,mLR) < d(x,mL)$ implies $x \in \big(Cone(mLR) \cup Cone(mRL)\big)$.

    To prove that $d(x,mRL) < d(x,mR)$ implies $x \in \big(Cone(mLR) \cup Cone(mRL)\big)$, we apply an analogous argument by interchanging the roles of $L$ and $R$.

    Hence, the statement $ x \in \big(Cone(mLR) \cup Cone(mRL)\big) \Leftrightarrow d(x,mLR) < d(x,mL)\Leftrightarrow d(x,mRL) < d(x,mR)$ is indeed valid.
\end{proof}

\begin{cor}
    Let $x,m \in M$, such that there exists $m' \in M$ that satisfies $m = m'L$  then $d(x,mR) < d(x,m)\Leftrightarrow x \in \big(Cone(mR) \cup Cone(m'RL)\big)$.
\end{cor}

\begin{figure}
    \centering
    \includegraphics[width=0.89\linewidth]{"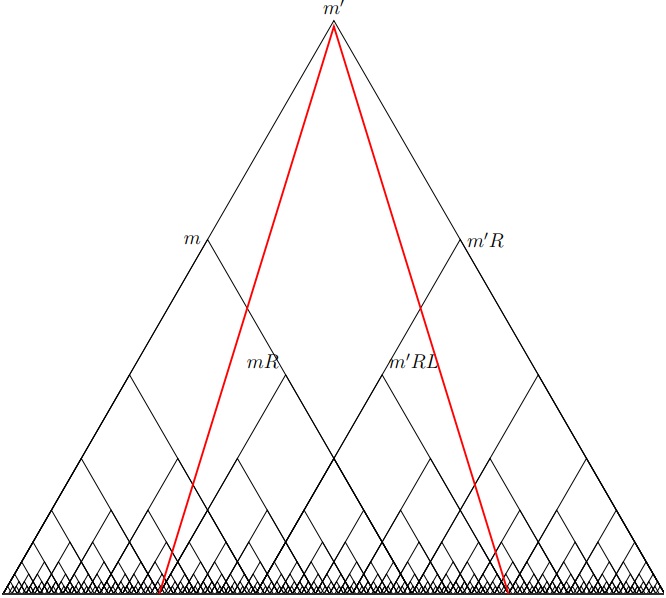"}
    \caption{A hyperedge picking out elements that satisfy the inequality $d(x, m) < d(x,mR)$}
    \label{fig: mLR shift}
\end{figure}

\begin{proof}
    The proof requires studying two cases where $m'$ exists and not. When $m' \notin M$, the conclusion follows trivially. 
    
    Let's look at the first case where $m' \in M$. We apply Lemma \ref{cone mlr = d mlr} to $Cone(m')$, see Figure \ref{fig: mLR shift}. Hence, $d(x,mR) < d(x,m)\Leftrightarrow x \in \big(Cone(mR) \cup Cone(m'RL)\big)$ follows immediately.
\end{proof}

\begin{cor}
    Let $x, m \in M$, then $d(x,mLR^2) < d(x,mLR) \Leftrightarrow d(x,mR)< d(x,m)$.
\end{cor}
\begin{proof}
    This is another consequence of Theorem \ref{M distance}. Observe that the element $mR$ satisfies the relation $d(x,mLR^2) <$ $ d(x,mLR)$. This implies that every element of $Cone(mR)$ satisfies the inequality, which implies $d(x,mR)<$ $d(x,m)$.
\end{proof}
\noindent These corollaries have non-trivial results that play a key role in understanding the atoms of $\MM$. Observe that whenever $m = m'L$, $d(x,mR^2) < d(x,mR)$ triggers a recursive effect as $d(x,m'LR^2) < d(x,m'LR)$ implies $d(x,m'R) < d(x,m')$. This will continue to recur until it hits the boundary of the graph or a state $d(x,nLR) < d(x,nL)$. This means that the set of points corresponding to the inequality~$d(x,mR) < d(x,m)$ will be the largest cone containing $mR$ but not $m$, and the cone $m'RL$.

Bringing it together gives us the following equivalent statements.
\begin{prop}\label{distance cone statements}
        Let $x, m \in M$, then there exists $n \in M$ such that $mL = nR^i$, where~$i$ is as large as possible. There exists $n' \in M$ such that $nL = n'R^j$, where $j = max\{0,1\}$. Then the following equivalent statements are valid:
    \begin{enumerate}
        \item $d(x,mL) < d(x,m) \Leftrightarrow x \in \big(Cone (n) \cup Cone(n')\big)$;
        \item $d(x,mL) > d(x,m) \Leftrightarrow x \notin \big(Cone (n) \cup Cone(n')\big)$.
    \end{enumerate}
    
    By interchanging the roles of $L$ and $R$, we get the following. Let $x, m \in M$, then there exists $n \in M$ such that $mR = nL^i$, where $i$ is as large as possible. There exists $n' \in M$ such that $nR = n'L^j$, where $j = max\{0,1\}$. Then the following equivalent statements are valid:
    \begin{enumerate}\addtocounter{enumi}{2}
        \item $d(x,mR) < d(x,m)\Leftrightarrow x \in \big(Cone (n) \cup Cone(n')\big)$;
        \item $d(x,mR) > d(x,m)\Leftrightarrow x \notin \big(Cone (n) \cup Cone(n')\big)$.
    \end{enumerate}
\end{prop}
\noindent Note that it is possible for $Cone(n') \subset Cone(n)$. See Figure \ref{fig: equiv statements} for a visualization of this process.
\begin{figure}
    \centering
    \includegraphics[width=0.89\linewidth]{"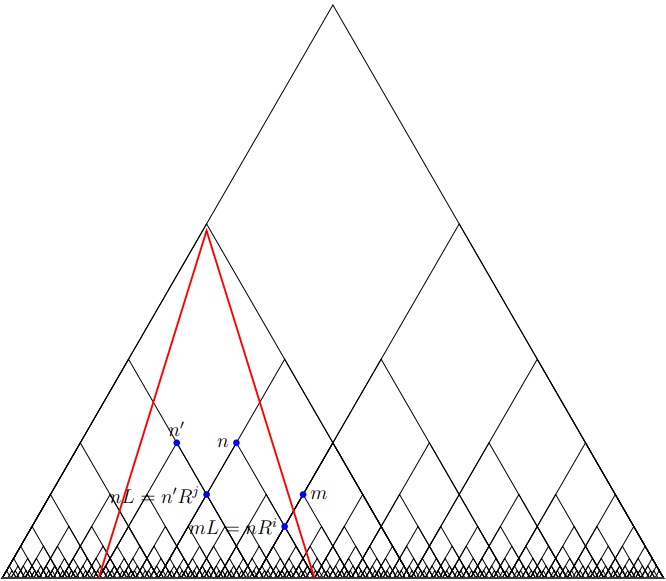"}
    \caption{Visualization of statements 1 and 2 of Proposition \ref{distance cone statements}}
    \label{fig: equiv statements}
\end{figure}
Although this proposition alone may not appear immediately insightful, it has promising consequences. If $x$ is an element in $Cone(m)$, then knowing the signs of the inequalities~$d(x,m) \neq d(x,mL)$ and $d(x,m) \neq d(x,mR)$, tells all the distances between every element~$n \notin Cone(m)$ and its offsprings. In other words, for every $x \in Cone(m)$, the distance vectors $\{m,mL\}$ and $\{m,mR\}$ determine the distance vectors outside $Cone(m)$. 

We end this section with a theorem that is a final consequence of these propositions. For any $m \in M$ we call the set $Cone(pLR) \cup Cone(pRL)$ a \textit{triangle} in $\MM$. We call the sets $Cone(L^n)$ and $Cone(R^n)$ \textit{side cones} for any $n>0$.

\begin{thm}
    Let $S$ be a triangle or a side cone in $\MM$. Let $p,q \in M$ be adjacent vertices with $p \in S$ and $q \in S^c$. Then the following statements hold:
    \begin{enumerate}
        \item if $x \in S$ then $d(x,p) < d(x, q)$; 
        \item if $x \in S^c$ then $d(x,q) < d(x, p)$.
    \end{enumerate}
\end{thm}
This theorem implies that for every edge in the Cayley graph $\MM$, there is exactly one hyperedge that crosses it. As a consequence, we say that two atoms $\AA$ and $\BB$ are of the same type if there exist $Cone(m) \supseteq \AA$ and $Cone(n) \supseteq \BB$ such that the following hold:

    \begin{itemize}
        \item for every vertex $a \in \AA$ there exists a vertex $b \in \BB$ such that $m^{-1}a = n^{-1}b$, i.e., the atoms have the same shapes;
        \item $F_{\AA} |_{Cone(m)} = F_{\BB}|_{Cone(n)}$, i.e., the cones that contain $\AA$ and $\BB$ have the same vector fields.
    \end{itemize}

\subsection{The horofunction boundary of \textit{M}} This section presents one of the main results of this paper, in particular demonstrating the connection between the horofunction boundary $\partial_h M$ of the Cayley graph $\MM$ of the monoid $M$ and the small golden ratio $\tau = \frac{\sqrt{5}-1}{2}$. We start by introducing the space $\DD_\tau$, followed by a series of lemmas that establish a connection between the horofunction boundary $\partial_h M$ and $\DD_\tau$. 

Let $I_\tau = \mathbb{Z}[\tau] \cap (0,1)$, where $\mathbb{Z}[\tau] = \{a + b\tau : a,b \in \mathbb{Z}\}$ and $\tau = \frac{1+\sqrt{5}}{2} \approx 0.618034$ is the small golden ratio. Recall that by Definition \ref{def Ctau} a blowup of $[0,1]$ along $I_\tau$ is a Cantor set denoted by $\CC_\tau$.

\begin{definition}\label{def:D_tau}
Let $\DD_\tau$ be the set
\[
\DD_\tau = ([0,1] \setminus I_\tau) \cup \{x^- : x \in I_\tau\} \cup \{x : x \in I_\tau\} \cup \{x^+ : x \in I_\tau\}
\]
equipped with a linear order $\prec$ satisfying the following conditions:
\begin{itemize}
    \item $\prec$ agrees with the standard order $<$ on $[0,1] \setminus I_\tau$;
    \item for all $x \in I_\tau$, $x^- \prec x \prec x^+$;
    \item if $x \in I_\tau$ and $y \in [0,1] \setminus I_\tau$, then $x^+ \prec y \iff x \prec y \iff x^- \prec y$;
    \item if $x, y \in I_\tau$ and $x < y$, then $x^- \prec x \prec x^+ \prec y^- \prec y \prec y^+$.
\end{itemize}
\end{definition}
\noindent The topology on $\DD_\tau$ is the order topology induced by $\prec$. The set $\DD_\tau$ can be viewed as a "double blowup" of $[0,1]$ along $I_\tau$, that is, every point in $I_\tau$ is replaced by three points instead of two points.

Although $\DD_\tau$ is not itself a Cantor set, it contains $\CC_\tau$ as a subset, which is homeomorphic to the Cantor set. The key difference is that $\DD_\tau$ contains additional isolated points (the elements $x \in I_\tau$) between each pair of blowup points in $\CC_\tau$.

Let $x_i \in I_\tau$ then the Cantor set $\CC_\tau$ is depicted as:
\[
\begin{tikzpicture}[scale=0.85]
\draw (1.5,0) node[anchor=north]{}
-- (3,0) node[anchor=south]{}
    (4.5,0) node[anchor=north]{}
-- (6,0) node[anchor=north]{}
    (7.5,0) node[anchor=north]{}
-- (9,0) node[anchor=north]{}
    (10.5,0) node[anchor=north]{}
-- (12,0) node[anchor=north]{}
    (13.5,0) node[anchor=north]{}
-- (15,0) node[anchor=north]{};

\node[circle,fill=black,inner sep=0pt,minimum size=0pt,label=above:{$\cdots$}] (a) at (0.5,-0.27) {};
\node[circle,fill=black,inner sep=0pt,minimum size=0pt,label=above:{$\cdots$}] (a) at (16.5,-0.27) {};

\node[circle,fill=black,inner sep=0pt,minimum size=4pt,label=above:{$x_1^+$}] (a) at (1.5,0) {};
\node[circle,fill=black,inner sep=0pt,minimum size=4pt,label=above:{$x_2^-$}] (a) at (3,0) {};
\node[circle,fill=black,inner sep=0pt,minimum size=4pt,label=above:{$x_2^+$}] (a) at (4.5,0) {};
\node[circle,fill=black,inner sep=0pt,minimum size=4pt,label=above:{$x_3^-$}] (a) at (6,0) {};
\node[circle,fill=black,inner sep=0pt,minimum size=4pt,label=above:{$x_3^+$}] (a) at (7.5,0) {};
\node[circle,fill=black,inner sep=0pt,minimum size=4pt,label=above:{$x_4^-$}] (a) at (9,0) {};
\node[circle,fill=black,inner sep=0pt,minimum size=4pt,label=above:{$x_4^+$}] (a) at (10.5,0) {};
\node[circle,fill=black,inner sep=0pt,minimum size=4pt,label=above:{$x_5^-$}] (a) at (12,0) {};
\node[circle,fill=black,inner sep=0pt,minimum size=4pt,label=above:{$x_5^+$}] (a) at (13.5,0) {};
\node[circle,fill=black,inner sep=0pt,minimum size=4pt,label=above:{$x_6^-$}] (a) at (15,0) {};
\end{tikzpicture}
\]
However, $\DD_\tau$ has additional isolated points in between every pair of non-dividable intervals. The following depicts the Cantor-like set $\DD_\tau$.
\[
\begin{tikzpicture}[scale=0.85]
\draw (1.5,0) node[anchor=north]{}
-- (3,0) node[anchor=south]{}
    (4.5,0) node[anchor=north]{}
-- (6,0) node[anchor=north]{}
    (7.5,0) node[anchor=north]{}
-- (9,0) node[anchor=north]{}
    (10.5,0) node[anchor=north]{}
-- (12,0) node[anchor=north]{}
    (13.5,0) node[anchor=north]{}
-- (15,0) node[anchor=north]{};

\node[circle,fill=black,inner sep=0pt,minimum size=0pt,label=above:{$\cdots$}] (a) at (0.5,-0.27) {};
\node[circle,fill=black,inner sep=0pt,minimum size=0pt,label=above:{$\cdots$}] (a) at (16.5,-0.27) {};

\node[circle,fill=black,inner sep=0pt,minimum size=4pt,label=above:{$x_1^+$}] (a) at (1.5,0) {};
\node[circle,fill=black,inner sep=0pt,minimum size=4pt,label=above:{$x_2^-$}] (a) at (3,0) {};
\node[circle,fill=black,inner sep=0pt,minimum size=4pt,label=above:{$x_2$}] (a) at (3.75,0) {};
\node[circle,fill=black,inner sep=0pt,minimum size=4pt,label=above:{$x_2^+$}] (a) at (4.5,0) {};
\node[circle,fill=black,inner sep=0pt,minimum size=4pt,label=above:{$x_3^-$}] (a) at (6,0) {};
\node[circle,fill=black,inner sep=0pt,minimum size=4pt,label=above:{$x_3$}] (a) at (6.75,0) {};
\node[circle,fill=black,inner sep=0pt,minimum size=4pt,label=above:{$x_3^+$}] (a) at (7.5,0) {};
\node[circle,fill=black,inner sep=0pt,minimum size=4pt,label=above:{$x_4^-$}] (a) at (9,0) {};
\node[circle,fill=black,inner sep=0pt,minimum size=4pt,label=above:{$x_4$}] (a) at (9.75,0) {};
\node[circle,fill=black,inner sep=0pt,minimum size=4pt,label=above:{$x_4^+$}] (a) at (10.5,0) {};
\node[circle,fill=black,inner sep=0pt,minimum size=4pt,label=above:{$x_5^-$}] (a) at (12,0) {};
\node[circle,fill=black,inner sep=0pt,minimum size=4pt,label=above:{$x_5$}] (a) at (12.75,0) {};
\node[circle,fill=black,inner sep=0pt,minimum size=4pt,label=above:{$x_5^+$}] (a) at (13.5,0) {};
\node[circle,fill=black,inner sep=0pt,minimum size=4pt,label=above:{$x_6^-$}] (a) at (15,0) {};
\end{tikzpicture}
\]
Recall that according to Lemma \ref{lem:I_M_homeomorphic_M} there is a  one-to-one correspondence between the elements of the monoid $M$ and the subintervals of the unit interval $I_M$, where $x \in M$ can be represented by $[0,1]_x \in I_M$. The system of equations \ref{eq: intervals} represents the actions of $L$ and $R$ on the interval.

We claim that these equations must be slightly modified to be able to use them in terms of infinite paths of $\MM$:
\begin{equation}\label{equation atoms}
\begin{aligned}
[x,y]_1 & \coloneq [x,y];\\
[x,y]_L & \coloneq [x, (y - (y-x)\tau^2)^-];\\
[x,y]_R & \coloneq [(x + (y-x)\tau^2)^+, y].\\
\end{aligned}
\end{equation}
The logic behind this step is motivated by the first 2-level decompositions of atoms of $\MM$ presented in Propositions \ref{prop: level-1 decomposition} and \ref{prop: level-2 decomposition}. In these propositions, we will see that the same element $RLLL\cdots = L^2RRR\cdots = LRLRLR\cdots$ of the monoid, depending on the chosen path, belongs to distinct atoms. This is not allowed by the definition of atoms. To solve this issue, we introduce a double blowup along elements of $\mathbb{Z}[\tau]\cap(0,1)$, resulting in the introduced modified equations above.

\begin{thm}\label{horofunction boundary of $M$}
    The horofunction boundary $\partial_h \MM$ of $\MM$ is naturally homeomorphic to~$\DD_\tau$.
\end{thm}
\noindent The remainder of the section is dedicated to proving this theorem. To understand the structure of the horofunction boundary $\partial_h \MM$, we will decompose infinite atoms into sequential infinite atom levels until no new infinite atom types appear.

\begin{prop}\label{prop: level-1 decomposition}
Let $\AA_1(\MM)$ be the set of the infinite atom at the 1-level of $\MM$. Then~$\AA_1(\MM)  = a_1 \cup b_1 \cup c_1$, where 
\[\begin{aligned}    
&a_1 = Cone(L) \setminus Cone(LR^2);\\
&b_1 = Cone(LR^2);\\
&c_1 = Cone(R) \setminus Cone(LR^2).\\
\end{aligned}\]
In other words, the graph $\MM$ is decomposed into 3 infinite sets by different vector fields of the unit ball centered at the root. See Figure \ref{fig: abc_1} for a visualization.  

\begin{figure}
    \centering
    \includegraphics[width=0.89\linewidth]{"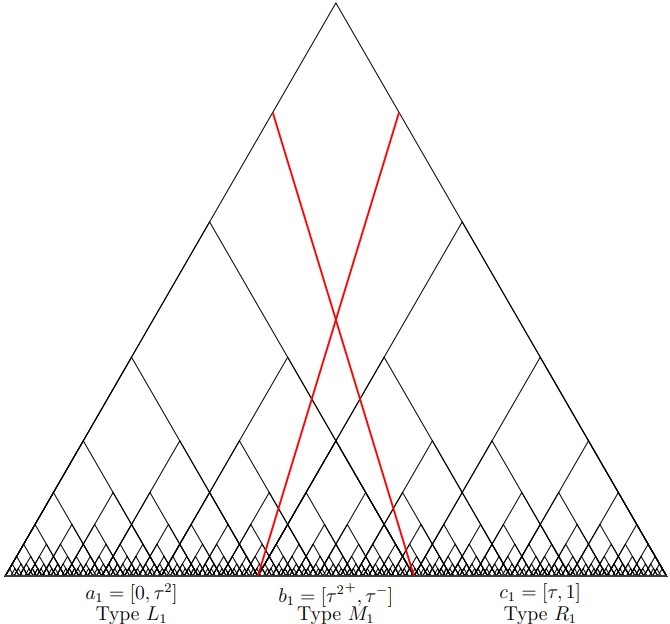"}
    \caption{Partition of 1-level infinite atoms $\AA_1(\MM)$}
    \label{fig: abc_1}
\end{figure}
\end{prop}

\begin{proof}
     Let $B_1$ be a unit ball centered at the root of $\MM = (M,E,1)$. Observe that there are a total of 4 possible principal vectors in $B_1$. We will denote them as:
     \begin{enumerate}
         \item $F_1(\{1,L\}) = (L,1)$ and $F_1(\{1,R\})=(R,1)$;
         \item $F_L(\{1,L\}) = (1,L)$ and $F_L(\{1,R\})=(R,1)$;
         \item $F_R(\{1,L\}) = (L,1)$ and $F_R(\{1,R\})=(1,R)$;
         \item $F_{LR^2}(\{1,L\}) = (1,L)$ and $F_{LR^2}(\{1,R\})=(1,R)$.
     \end{enumerate}
    To ease the understanding of the complicated representations of the possible vector fields in $B_1$ we invite the reader to take a look at their scaled graphic representation:
\[
\begin{tikzpicture}[scale=0.3]
\draw(4.32624,7.49329) node[anchor=north]{}
-- (7,12.1244) node[anchor=south]{}
-- (9.67376,7.49329) node[anchor=north]{}
;

\draw (2.67376, 4.63111) node[anchor=north]{}
-- (4.32624,7.49329) node[anchor=south]{}
-- (5.97871, 4.63111) node[anchor=north]{}
;

\draw (8.02129, 4.63111) node[anchor=north]{}
-- (9.67376,7.49329) node[anchor=south]{}
-- (11.3262, 4.63111) node[anchor=north]{}
;

\draw (1.6524745577599187,2.862183385639525) node[anchor=north]{}
-- (2.67376, 4.63111) node[anchor=south]{}
-- (3.6950454422400814, 2.862183385639525) node[anchor=north]{};

\draw (8.652480-3.6950454422400814, 2.862183385639525) node[anchor=north]{}
-- (5.97871, 4.63111) node[anchor=south]{}
-- (7,2.862183385639525) node[anchor=north]{};

\draw (7,2.862183385639525) node[anchor=north]{}
-- (8.02129, 4.63111) node[anchor=south]{}
-- (8.652480-3.6950454422400814+2.042570884480163+2.042570884480163, 2.862183385639525) node[anchor=north]{};

\draw (14-1.6524745577599187,2.862183385639525) node[anchor=north]{}
-- (11.3262, 4.63111) node[anchor=south]{}
-- (14-3.6950454422400814,2.862183385639525) node[anchor=north]{}
;

\draw[-{Stealth[length=3mm]}]
(4.32624,7.49329) node[anchor=south]{}
--(5.66312,9.808845) node[anchor=north]{}
;

\draw[-{Stealth[length=3mm]}]
(9.67376,7.49329) node[anchor=north]{}
--(14-5.66312,9.808845) node[anchor=south]{}
;

\end{tikzpicture}
\quad
\begin{tikzpicture}[scale=0.3]
\draw(4.32624,7.49329) node[anchor=north]{}
-- (7,12.1244) node[anchor=south]{}
-- (9.67376,7.49329) node[anchor=north]{}
;

\draw (2.67376, 4.63111) node[anchor=north]{}
-- (4.32624,7.49329) node[anchor=south]{}
-- (5.97871, 4.63111) node[anchor=north]{}
;

\draw (8.02129, 4.63111) node[anchor=north]{}
-- (9.67376,7.49329) node[anchor=south]{}
-- (11.3262, 4.63111) node[anchor=north]{}
;

\draw (1.6524745577599187,2.862183385639525) node[anchor=north]{}
-- (2.67376, 4.63111) node[anchor=south]{}
-- (3.6950454422400814, 2.862183385639525) node[anchor=north]{};

\draw (8.652480-3.6950454422400814, 2.862183385639525) node[anchor=north]{}
-- (5.97871, 4.63111) node[anchor=south]{}
-- (7,2.862183385639525) node[anchor=north]{};

\draw (7,2.862183385639525) node[anchor=north]{}
-- (8.02129, 4.63111) node[anchor=south]{}
-- (8.652480-3.6950454422400814+2.042570884480163+2.042570884480163, 2.862183385639525) node[anchor=north]{};

\draw (14-1.6524745577599187,2.862183385639525) node[anchor=north]{}
-- (11.3262, 4.63111) node[anchor=south]{}
-- (14-3.6950454422400814,2.862183385639525) node[anchor=north]{}
;

\draw[-{Stealth[length=3mm]}]
(7,12.1244) node[anchor=south]{}
--(5.66312,9.808845) node[anchor=north]{}
;

\draw[-{Stealth[length=3mm]}]
(9.67376,7.49329) node[anchor=north]{}
--(14-5.66312,9.808845) node[anchor=south]{}
;

\end{tikzpicture}
\quad
\begin{tikzpicture}[scale=0.3]
\draw(4.32624,7.49329) node[anchor=north]{}
-- (7,12.1244) node[anchor=south]{}
-- (9.67376,7.49329) node[anchor=north]{}
;

\draw (2.67376, 4.63111) node[anchor=north]{}
-- (4.32624,7.49329) node[anchor=south]{}
-- (5.97871, 4.63111) node[anchor=north]{}
;

\draw (8.02129, 4.63111) node[anchor=north]{}
-- (9.67376,7.49329) node[anchor=south]{}
-- (11.3262, 4.63111) node[anchor=north]{}
;

\draw (1.6524745577599187,2.862183385639525) node[anchor=north]{}
-- (2.67376, 4.63111) node[anchor=south]{}
-- (3.6950454422400814, 2.862183385639525) node[anchor=north]{};

\draw (8.652480-3.6950454422400814, 2.862183385639525) node[anchor=north]{}
-- (5.97871, 4.63111) node[anchor=south]{}
-- (7,2.862183385639525) node[anchor=north]{};

\draw (7,2.862183385639525) node[anchor=north]{}
-- (8.02129, 4.63111) node[anchor=south]{}
-- (8.652480-3.6950454422400814+2.042570884480163+2.042570884480163, 2.862183385639525) node[anchor=north]{};

\draw (14-1.6524745577599187,2.862183385639525) node[anchor=north]{}
-- (11.3262, 4.63111) node[anchor=south]{}
-- (14-3.6950454422400814,2.862183385639525) node[anchor=north]{}
;

\draw[-{Stealth[length=3mm]}]
(4.32624,7.49329) node[anchor=south]{}
--(5.66312,9.808845) node[anchor=north]{}
;

\draw[-{Stealth[length=3mm]}]
(7,12.1244) node[anchor=north]{}
--(14-5.66312,9.808845) node[anchor=south]{}
;

\end{tikzpicture}
\quad
\begin{tikzpicture}[scale=0.3]
\draw(4.32624,7.49329) node[anchor=north]{}
-- (7,12.1244) node[anchor=south]{}
-- (9.67376,7.49329) node[anchor=north]{}
;

\draw (2.67376, 4.63111) node[anchor=north]{}
-- (4.32624,7.49329) node[anchor=south]{}
-- (5.97871, 4.63111) node[anchor=north]{}
;

\draw (8.02129, 4.63111) node[anchor=north]{}
-- (9.67376,7.49329) node[anchor=south]{}
-- (11.3262, 4.63111) node[anchor=north]{}
;

\draw (1.6524745577599187,2.862183385639525) node[anchor=north]{}
-- (2.67376, 4.63111) node[anchor=south]{}
-- (3.6950454422400814, 2.862183385639525) node[anchor=north]{};

\draw (8.652480-3.6950454422400814, 2.862183385639525) node[anchor=north]{}
-- (5.97871, 4.63111) node[anchor=south]{}
-- (7,2.862183385639525) node[anchor=north]{};

\draw (7,2.862183385639525) node[anchor=north]{}
-- (8.02129, 4.63111) node[anchor=south]{}
-- (8.652480-3.6950454422400814+2.042570884480163+2.042570884480163, 2.862183385639525) node[anchor=north]{};

\draw (14-1.6524745577599187,2.862183385639525) node[anchor=north]{}
-- (11.3262, 4.63111) node[anchor=south]{}
-- (14-3.6950454422400814,2.862183385639525) node[anchor=north]{}
;

\draw[-{Stealth[length=3mm]}]
(7,12.1244) node[anchor=south]{}
--(5.66312,9.808845) node[anchor=north]{}
;

\draw[-{Stealth[length=3mm]}]
(7,12.1244) node[anchor=north]{}
--(14-5.66312,9.808845) node[anchor=south]{}
;

\end{tikzpicture}
\]   
Recall that a vector $F(\{n,m\}) = (n,m)$ is equivalent to the statement $d(x,n) > d(x,m)$.

We will now analyze the four cases that correspond to the possible vector fields in $B_1$. We start from the first case.

Let $x \in M$ be such that $d(x,L) > d(x,1)$ and $d(x,R) > d(x,1)$, following Lemma \ref{field distance eq} we observe that~$x \notin Cone(L)$ and~$x \notin Cone(R)$. Hence,~$x \in \{1\}$. This is a finite set; hence, by definition, it does not contribute to the horofunction boundary.

We move on to the second case. Let $x \in M$ be such that $d(x,L) < d(x,1)$ and $d(x,R) > d(x,1)$. Following Lemma \ref{field distance eq} we observe that $x \in Cone(L)$ and $x \notin Cone(R)$. Hence, $x \in Cone(L) \setminus Cone(LR^2)$. Let this atom be denoted as $a_1$. This is an infinite set; hence it contributes to the horofunction boundary.

The third case is similar to the second one. Let $x \in M$ be such that $d(x,L) > d(x,1)$ and~$d(x,R) < d(x,1)$, following Lemma \ref{field distance eq} we observe that $x \notin Cone(L)$ and $x \in Cone(R)$. Hence,~$x \in Cone(R) \setminus Cone(LR^2)$. Let this atom be denoted as $c_1$

We are left with the last case. Let $x \in M$ be such that $d(x,L) < d(x,1)$ and $d(x,R) < d(x,1)$, following Lemma \ref{field distance eq} we observe that $x \in Cone(L)$ and $x \in Cone(R)$. Hence, $x \in Cone(LR^2)$. Let this atom be denoted as $b_1$

The following is a scaled visualization of these 4 sets, for a full size picture we refer the reader to Figure \ref{fig: abc_1}.
\[
\begin{tikzpicture}[scale=0.3]
\draw(4.32624,7.49329) node[anchor=north]{}
-- (7,12.1244) node[anchor=south]{}
-- (9.67376,7.49329) node[anchor=north]{}
;

\draw (2.67376, 4.63111) node[anchor=north]{}
-- (4.32624,7.49329) node[anchor=south]{}
-- (5.97871, 4.63111) node[anchor=north]{}
;

\draw (8.02129, 4.63111) node[anchor=north]{}
-- (9.67376,7.49329) node[anchor=south]{}
-- (11.3262, 4.63111) node[anchor=north]{}
;

\draw (1.6524745577599187,2.862183385639525) node[anchor=north]{}
-- (2.67376, 4.63111) node[anchor=south]{}
-- (3.6950454422400814, 2.862183385639525) node[anchor=north]{};

\draw (8.652480-3.6950454422400814, 2.862183385639525) node[anchor=north]{}
-- (5.97871, 4.63111) node[anchor=south]{}
-- (7,2.862183385639525) node[anchor=north]{};

\draw (7,2.862183385639525) node[anchor=north]{}
-- (8.02129, 4.63111) node[anchor=south]{}
-- (8.652480-3.6950454422400814+2.042570884480163+2.042570884480163, 2.862183385639525) node[anchor=north]{};

\draw (14-1.6524745577599187,2.862183385639525) node[anchor=north]{}
-- (11.3262, 4.63111) node[anchor=south]{}
-- (14-3.6950454422400814,2.862183385639525) node[anchor=north]{}
;

\draw[-{Stealth[length=3mm]}]
(4.32624,7.49329) node[anchor=south]{}
--(5.66312,9.808845) node[anchor=north]{}
;

\draw[-{Stealth[length=3mm]}]
(9.67376,7.49329) node[anchor=north]{}
--(14-5.66312,9.808845) node[anchor=south]{}
;

\draw [red,line width=0.4mm,](9.67376544224,2.862183385639525) node[anchor=north]{}
--(5.66312,9.808845);
\draw [red,line width=0.4mm,](14-9.67376544224,2.862183385639525) node[anchor=north]{}
--(14-5.66312,9.808845);

\end{tikzpicture}
\quad
\begin{tikzpicture}[scale=0.3]
\draw(4.32624,7.49329) node[anchor=north]{}
-- (7,12.1244) node[anchor=south]{}
-- (9.67376,7.49329) node[anchor=north]{}
;

\draw (2.67376, 4.63111) node[anchor=north]{}
-- (4.32624,7.49329) node[anchor=south]{}
-- (5.97871, 4.63111) node[anchor=north]{}
;

\draw (8.02129, 4.63111) node[anchor=north]{}
-- (9.67376,7.49329) node[anchor=south]{}
-- (11.3262, 4.63111) node[anchor=north]{}
;

\draw (1.6524745577599187,2.862183385639525) node[anchor=north]{}
-- (2.67376, 4.63111) node[anchor=south]{}
-- (3.6950454422400814, 2.862183385639525) node[anchor=north]{};

\draw (8.652480-3.6950454422400814, 2.862183385639525) node[anchor=north]{}
-- (5.97871, 4.63111) node[anchor=south]{}
-- (7,2.862183385639525) node[anchor=north]{};

\draw (7,2.862183385639525) node[anchor=north]{}
-- (8.02129, 4.63111) node[anchor=south]{}
-- (8.652480-3.6950454422400814+2.042570884480163+2.042570884480163, 2.862183385639525) node[anchor=north]{};

\draw (14-1.6524745577599187,2.862183385639525) node[anchor=north]{}
-- (11.3262, 4.63111) node[anchor=south]{}
-- (14-3.6950454422400814,2.862183385639525) node[anchor=north]{}
;

\draw[-{Stealth[length=3mm]}]
(7,12.1244) node[anchor=south]{}
--(5.66312,9.808845) node[anchor=north]{}
;

\draw[-{Stealth[length=3mm]}]
(9.67376,7.49329) node[anchor=north]{}
--(14-5.66312,9.808845) node[anchor=south]{}
;

\draw [red,line width=0.4mm,](9.67376544224,2.862183385639525) node[anchor=north]{}
--(5.66312,9.808845);
\draw [red,line width=0.4mm,](14-9.67376544224,2.862183385639525) node[anchor=north]{}
--(14-5.66312,9.808845);

\end{tikzpicture}
\quad
\begin{tikzpicture}[scale=0.3]
\draw(4.32624,7.49329) node[anchor=north]{}
-- (7,12.1244) node[anchor=south]{}
-- (9.67376,7.49329) node[anchor=north]{}
;

\draw (2.67376, 4.63111) node[anchor=north]{}
-- (4.32624,7.49329) node[anchor=south]{}
-- (5.97871, 4.63111) node[anchor=north]{}
;

\draw (8.02129, 4.63111) node[anchor=north]{}
-- (9.67376,7.49329) node[anchor=south]{}
-- (11.3262, 4.63111) node[anchor=north]{}
;

\draw (1.6524745577599187,2.862183385639525) node[anchor=north]{}
-- (2.67376, 4.63111) node[anchor=south]{}
-- (3.6950454422400814, 2.862183385639525) node[anchor=north]{};

\draw (8.652480-3.6950454422400814, 2.862183385639525) node[anchor=north]{}
-- (5.97871, 4.63111) node[anchor=south]{}
-- (7,2.862183385639525) node[anchor=north]{};

\draw (7,2.862183385639525) node[anchor=north]{}
-- (8.02129, 4.63111) node[anchor=south]{}
-- (8.652480-3.6950454422400814+2.042570884480163+2.042570884480163, 2.862183385639525) node[anchor=north]{};

\draw (14-1.6524745577599187,2.862183385639525) node[anchor=north]{}
-- (11.3262, 4.63111) node[anchor=south]{}
-- (14-3.6950454422400814,2.862183385639525) node[anchor=north]{}
;

\draw[-{Stealth[length=3mm]}]
(4.32624,7.49329) node[anchor=south]{}
--(5.66312,9.808845) node[anchor=north]{}
;

\draw[-{Stealth[length=3mm]}]
(7,12.1244) node[anchor=north]{}
--(14-5.66312,9.808845) node[anchor=south]{}
;

\draw [red,line width=0.4mm,](9.67376544224,2.862183385639525) node[anchor=north]{}
--(5.66312,9.808845);
\draw [red,line width=0.4mm,](14-9.67376544224,2.862183385639525) node[anchor=north]{}
--(14-5.66312,9.808845);
\end{tikzpicture}
\quad
\begin{tikzpicture}[scale=0.3]
\draw(4.32624,7.49329) node[anchor=north]{}
-- (7,12.1244) node[anchor=south]{}
-- (9.67376,7.49329) node[anchor=north]{}
;

\draw (2.67376, 4.63111) node[anchor=north]{}
-- (4.32624,7.49329) node[anchor=south]{}
-- (5.97871, 4.63111) node[anchor=north]{}
;

\draw (8.02129, 4.63111) node[anchor=north]{}
-- (9.67376,7.49329) node[anchor=south]{}
-- (11.3262, 4.63111) node[anchor=north]{}
;

\draw (1.6524745577599187,2.862183385639525) node[anchor=north]{}
-- (2.67376, 4.63111) node[anchor=south]{}
-- (3.6950454422400814, 2.862183385639525) node[anchor=north]{};

\draw (8.652480-3.6950454422400814, 2.862183385639525) node[anchor=north]{}
-- (5.97871, 4.63111) node[anchor=south]{}
-- (7,2.862183385639525) node[anchor=north]{};

\draw (7,2.862183385639525) node[anchor=north]{}
-- (8.02129, 4.63111) node[anchor=south]{}
-- (8.652480-3.6950454422400814+2.042570884480163+2.042570884480163, 2.862183385639525) node[anchor=north]{};

\draw (14-1.6524745577599187,2.862183385639525) node[anchor=north]{}
-- (11.3262, 4.63111) node[anchor=south]{}
-- (14-3.6950454422400814,2.862183385639525) node[anchor=north]{}
;

\draw[-{Stealth[length=3mm]}]
(7,12.1244) node[anchor=south]{}
--(5.66312,9.808845) node[anchor=north]{}
;

\draw[-{Stealth[length=3mm]}]
(7,12.1244) node[anchor=north]{}
--(14-5.66312,9.808845) node[anchor=south]{}
;

\draw [red,line width=0.4mm,](9.67376544224,2.862183385639525) node[anchor=north]{}
--(5.66312,9.808845);
\draw [red,line width=0.4mm,](14-9.67376544224,2.862183385639525) node[anchor=north]{}
--(14-5.66312,9.808845);
\end{tikzpicture}
\]   
Let the atoms $a_1$, $b_1$, and $c_1$ be of types $L_1$, $M_1$, and $R_1$, respectively.
\end{proof}

Following Lemma \ref{lem:I_M_homeomorphic_M} we should be able to represent the atoms $a_1$, $b_1$, and $c_1$ in terms of the real line. However, we immediately face a problem. The atom $a_1 = Cone(L) \setminus Cone(LR^2)$ corresponds to the interval $[0,\tau^2]$, and the atom $b_1 = Cone(LR^2)$ corresponds to the interval $[\tau^2, \tau]$. Observe that the point $\tau^2$ belongs both to $a_1$ and~$b_1$. This is due to the fact that $\tau^2$ corresponds to multiple elements of~$M$ in particular $L^2RRR\cdots$ and~$RLLL\cdots$. When studying atoms, we do not allow such constructions, as infinite paths of atoms must lead to distinct elements. Hence, there is a blowup of $\tau^2$ in such a way that~$\tau^2 \neq {\tau^2}'$,~$\tau^2 < {\tau^2}'$, and $\tau^2 \in a_1$, ${\tau^2}' \in b_1$.

However, the statement above is not true when studying the atoms of $\MM$. To avoid repetition of the argument, we will just state that the atom $a_1$ will decompose in such a way that the point $\tau^2$ will blowup again. This will be seen in the next proposition.

Hence, equations \ref{equation atoms} can be used to represent atoms. We will now state the atoms $a_1$, $b_1$, and $c_1$ in terms of intervals.

We start from $a_1$. The condition $a_1 \subseteq Cone(L)$ is equivalent to $a_1 \subseteq [0,1]_L = [0,{\tau}^-]$ and $a_1 \not\subseteq [0,1]_R = [\tau^+,1]$, so $a_1$ is the set $[0,\tau^2]$.   

For the atom $b_1$ the condition $b_1 \subseteq Cone(LR^2)$ is equivalent to $b_1 = [0,1]_{LR^2} = [{\tau^2}^+, \tau^-]$. 

We invite the reader to verify that $c_1 = [\tau,1]$.

\begin{prop} \label{prop: level-2 decomposition}
    The atom $a_1$ decomposes into three disjoint infinite atoms in $\AA_2(\MM)$, denoted as~$a_2$, $b_2$ and $c_2$. These infinite atoms correspond to the following subsets of $\MM$: 
    \[\begin{aligned}
    &a_2 = Cone(L^2) \setminus Cone(L^2R^2);\\
    &b_2 = Cone(L^2R^2);\\
    &c_2 = Cone(LR) \setminus \big( Cone(LRL^2) \cup Cone(LR^2)\big).\\
    \end{aligned}\]
Moreover, $a_1 = \{L\} \cup a_2 \cup b_2 \cup c_2$. See Figure \ref{fig: a_2} for a visualization.  
\end{prop}

\begin{figure}
    \centering
    \includegraphics[width=0.89\linewidth]{"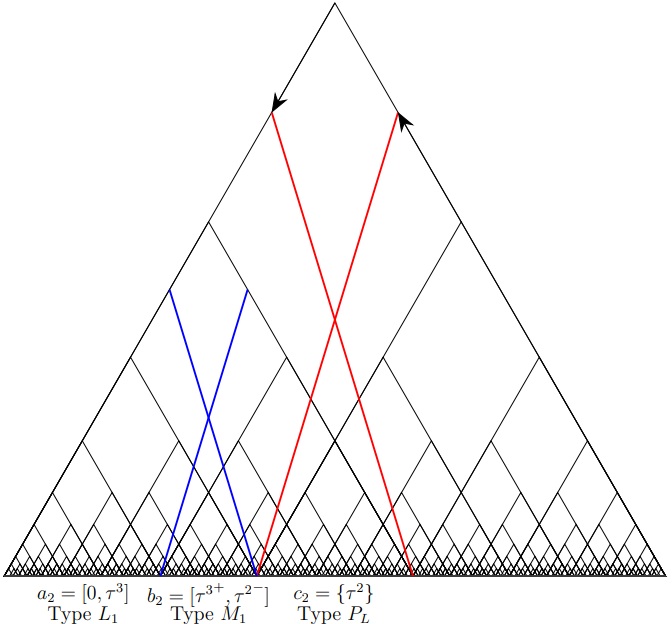"}
    \caption{Decomposition of the atom $a_1$ into $a_2$, $b_2$, and $c_2$ in the second level $\AA_2(\MM)$}
    \label{fig: a_2}
\end{figure}

\begin{proof}
    Let's consider the $a_1$ atom, whose vector field in $B_1$ points towards $Cone(L)$, see case 2 of Proposition \ref{prop: level-1 decomposition} for an illustration. Following Proposition \ref{distance cone statements} we know that the vectors in $Cone(L)$ determine the decomposition of the atoms in $Cone(L)$; therefore, we ignore the vectors of $Cone(L) \setminus Cone(R)$ as they will be predetermined by those in~$Cone(L)$. Hence, there will be a total of 4 cases:
     \begin{enumerate}
         \item $F_L(\{L,L^2\}) = (L^2,L)$ and $F_L(\{L,LR\})=(LR,L)$;
         \item $F_{L^2}(\{L,L^2\}) = (L,L^2)$ and $F_{L^2}(\{L,LR\})=(LR,L)$;
         \item $F_{LR}(\{L,L^2\}) = (L^2,L)$ and $F_{LR}(\{L,LR\})=(L,LR)$;
         \item $F_{L^2R^2}(\{L,L^2\}) = (L,L^2)$ and $F_{L^2R^2}(\{L,LR\})=(L,LR)$.
     \end{enumerate}
\[
\begin{tikzpicture}[scale=0.3]
\draw(4.32624,7.49329) node[anchor=north]{}
-- (7,12.1244) node[anchor=south]{}
-- (9.67376,7.49329) node[anchor=north]{}
;

\draw (2.67376, 4.63111) node[anchor=north]{}
-- (4.32624,7.49329) node[anchor=south]{}
-- (5.97871, 4.63111) node[anchor=north]{}
;

\draw (8.02129, 4.63111) node[anchor=north]{}
-- (9.67376,7.49329) node[anchor=south]{}
-- (11.3262, 4.63111) node[anchor=north]{}
;

\draw (1.6524745577599187,2.862183385639525) node[anchor=north]{}
-- (2.67376, 4.63111) node[anchor=south]{}
-- (3.6950454422400814, 2.862183385639525) node[anchor=north]{};

\draw (8.652480-3.6950454422400814, 2.862183385639525) node[anchor=north]{}
-- (5.97871, 4.63111) node[anchor=south]{}
-- (7,2.862183385639525) node[anchor=north]{};

\draw (7,2.862183385639525) node[anchor=north]{}
-- (8.02129, 4.63111) node[anchor=south]{}
-- (8.652480-3.6950454422400814+2.042570884480163+2.042570884480163, 2.862183385639525) node[anchor=north]{};

\draw (14-1.6524745577599187,2.862183385639525) node[anchor=north]{}
-- (11.3262, 4.63111) node[anchor=south]{}
-- (14-3.6950454422400814,2.862183385639525) node[anchor=north]{}
;

\draw[-{Stealth[length=3mm]}]
(7,12.1244) node[anchor=south]{}
--(5.66312,9.808845) node[anchor=north]{}
;

\draw[-{Stealth[length=3mm]}]
(9.67376,7.49329) node[anchor=north]{}
--(14-5.66312,9.808845) node[anchor=south]{}
;

\draw[-{Stealth[length=3mm]}]
(2.67361632,4.63085322) node[anchor=south]{}
--(3.49980816,6.06186621) node[anchor=north]{}
;

\draw[-{Stealth[length=3mm]}]
(5.97838368,4.63085322) node[anchor=north]{}
--(5.15219184,6.06186621) node[anchor=south]{}
;

\end{tikzpicture}
\quad
\begin{tikzpicture}[scale=0.3]
\draw(4.32624,7.49329) node[anchor=north]{}
-- (7,12.1244) node[anchor=south]{}
-- (9.67376,7.49329) node[anchor=north]{}
;

\draw (2.67376, 4.63111) node[anchor=north]{}
-- (4.32624,7.49329) node[anchor=south]{}
-- (5.97871, 4.63111) node[anchor=north]{}
;

\draw (8.02129, 4.63111) node[anchor=north]{}
-- (9.67376,7.49329) node[anchor=south]{}
-- (11.3262, 4.63111) node[anchor=north]{}
;

\draw (1.6524745577599187,2.862183385639525) node[anchor=north]{}
-- (2.67376, 4.63111) node[anchor=south]{}
-- (3.6950454422400814, 2.862183385639525) node[anchor=north]{};

\draw (8.652480-3.6950454422400814, 2.862183385639525) node[anchor=north]{}
-- (5.97871, 4.63111) node[anchor=south]{}
-- (7,2.862183385639525) node[anchor=north]{};

\draw (7,2.862183385639525) node[anchor=north]{}
-- (8.02129, 4.63111) node[anchor=south]{}
-- (8.652480-3.6950454422400814+2.042570884480163+2.042570884480163, 2.862183385639525) node[anchor=north]{};

\draw (14-1.6524745577599187,2.862183385639525) node[anchor=north]{}
-- (11.3262, 4.63111) node[anchor=south]{}
-- (14-3.6950454422400814,2.862183385639525) node[anchor=north]{}
;

\draw[-{Stealth[length=3mm]}]
(7,12.1244) node[anchor=south]{}
--(5.66312,9.808845) node[anchor=north]{}
;

\draw[-{Stealth[length=3mm]}]
(9.67376,7.49329) node[anchor=north]{}
--(14-5.66312,9.808845) node[anchor=south]{}
;

\draw[-{Stealth[length=3mm]}]
(3.49980816,6.06186621) node[anchor=south]{}
--(2.67361632,4.63085322) node[anchor=north]{}
;

\draw[-{Stealth[length=3mm]}]
(5.97838368,4.63085322) node[anchor=north]{}
--(5.15219184,6.06186621) node[anchor=south]{}
;

\end{tikzpicture}
\quad
\begin{tikzpicture}[scale=0.3]
\draw(4.32624,7.49329) node[anchor=north]{}
-- (7,12.1244) node[anchor=south]{}
-- (9.67376,7.49329) node[anchor=north]{}
;

\draw (2.67376, 4.63111) node[anchor=north]{}
-- (4.32624,7.49329) node[anchor=south]{}
-- (5.97871, 4.63111) node[anchor=north]{}
;

\draw (8.02129, 4.63111) node[anchor=north]{}
-- (9.67376,7.49329) node[anchor=south]{}
-- (11.3262, 4.63111) node[anchor=north]{}
;

\draw (1.6524745577599187,2.862183385639525) node[anchor=north]{}
-- (2.67376, 4.63111) node[anchor=south]{}
-- (3.6950454422400814, 2.862183385639525) node[anchor=north]{};

\draw (8.652480-3.6950454422400814, 2.862183385639525) node[anchor=north]{}
-- (5.97871, 4.63111) node[anchor=south]{}
-- (7,2.862183385639525) node[anchor=north]{};

\draw (7,2.862183385639525) node[anchor=north]{}
-- (8.02129, 4.63111) node[anchor=south]{}
-- (8.652480-3.6950454422400814+2.042570884480163+2.042570884480163, 2.862183385639525) node[anchor=north]{};

\draw (14-1.6524745577599187,2.862183385639525) node[anchor=north]{}
-- (11.3262, 4.63111) node[anchor=south]{}
-- (14-3.6950454422400814,2.862183385639525) node[anchor=north]{}
;

\draw[-{Stealth[length=3mm]}]
(7,12.1244) node[anchor=south]{}
--(5.66312,9.808845) node[anchor=north]{}
;

\draw[-{Stealth[length=3mm]}]
(9.67376,7.49329) node[anchor=north]{}
--(14-5.66312,9.808845) node[anchor=south]{}
;

\draw[-{Stealth[length=3mm]}]
(2.67361632,4.63085322) node[anchor=south]{}
--(3.49980816,6.06186621) node[anchor=north]{}
;

\draw[-{Stealth[length=3mm]}]
(5.15219184,6.06186621) node[anchor=north]{}
--(5.97838368,4.63085322) node[anchor=south]{}
;

\end{tikzpicture}
\quad
\begin{tikzpicture}[scale=0.3]
\draw(4.32624,7.49329) node[anchor=north]{}
-- (7,12.1244) node[anchor=south]{}
-- (9.67376,7.49329) node[anchor=north]{}
;

\draw (2.67376, 4.63111) node[anchor=north]{}
-- (4.32624,7.49329) node[anchor=south]{}
-- (5.97871, 4.63111) node[anchor=north]{}
;

\draw (8.02129, 4.63111) node[anchor=north]{}
-- (9.67376,7.49329) node[anchor=south]{}
-- (11.3262, 4.63111) node[anchor=north]{}
;

\draw (1.6524745577599187,2.862183385639525) node[anchor=north]{}
-- (2.67376, 4.63111) node[anchor=south]{}
-- (3.6950454422400814, 2.862183385639525) node[anchor=north]{};

\draw (8.652480-3.6950454422400814, 2.862183385639525) node[anchor=north]{}
-- (5.97871, 4.63111) node[anchor=south]{}
-- (7,2.862183385639525) node[anchor=north]{};

\draw (7,2.862183385639525) node[anchor=north]{}
-- (8.02129, 4.63111) node[anchor=south]{}
-- (8.652480-3.6950454422400814+2.042570884480163+2.042570884480163, 2.862183385639525) node[anchor=north]{};

\draw (14-1.6524745577599187,2.862183385639525) node[anchor=north]{}
-- (11.3262, 4.63111) node[anchor=south]{}
-- (14-3.6950454422400814,2.862183385639525) node[anchor=north]{}
;

\draw[-{Stealth[length=3mm]}]
(7,12.1244) node[anchor=south]{}
--(5.66312,9.808845) node[anchor=north]{}
;

\draw[-{Stealth[length=3mm]}]
(9.67376,7.49329) node[anchor=north]{}
--(14-5.66312,9.808845) node[anchor=south]{}
;

\draw[-{Stealth[length=3mm]}]
(3.49980816,6.06186621) node[anchor=south]{}
--(2.67361632,4.63085322) node[anchor=north]{}
;

\draw[-{Stealth[length=3mm]}]
(5.15219184,6.06186621) node[anchor=north]{}
--(5.97838368,4.63085322) node[anchor=south]{}
;
\end{tikzpicture}
\]   
We leave it to the reader to check that Cases 1, 2, and 4 are analogous to the corresponding cases analyzed in Propositions \ref{prop: level-1 decomposition} subject to scaling. 

However, we will check the decomposition of the third case, as it has a nontrivial result. Let~$x \in M$ be such that $d(x,LR)<d(x,L)<d(x,1)<d(x,R)$ and $d(x,LR)<d(x,L^2)$. By Lemma \ref{field distance eq} this is equivalent to~$x \notin Cone(R)$,~$x \in Cone(L)$,~$x \notin Cone(L^2)$ and~$x \in Cone(LR)$. Hence,~$x \in Cone(LR) \setminus \big(Cone(L^2R^2) \cup Cone(LR^2)\big)$. Let this atom be denoted as~$c_2$. 

We translate this into terms of intervals $c_2 \subseteq [0,1]_{LR}= [{\tau^3}^+ , \tau^-] $, $b_2 \not\subseteq [0,1]_{L^2R^2} = [0,{\tau^2}^-]$, and $b_2 \not\subseteq [0,1]_{LR^2} = [{\tau^2}^+, \tau^-]$. Hence, $c_2 = \{\tau^2\}$ is a point.
\end{proof}

\noindent Observe that the pairs of atoms $a_1$, $a_2$ and $b_1$ and $b_2$ have the same shape and type. Let~$c_2$ be an atom of type $P_L$.

\begin{prop}
    The infinite atom $b_1$ at the 1-level does not decompose in $\AA_2(\MM)$.
\end{prop}
\noindent To prove this proposition, one can apply the same method presented above and see that there is only one possibility for the 2-level atom vector field configuration and that it does not decompose $b_1$. We will refer to this infinite atom as $d_2$.
\[
\begin{tikzpicture}[scale=0.3]
\draw(4.32624,7.49329) node[anchor=north]{}
-- (7,12.1244) node[anchor=south]{}
-- (9.67376,7.49329) node[anchor=north]{}
;

\draw (2.67376, 4.63111) node[anchor=north]{}
-- (4.32624,7.49329) node[anchor=south]{}
-- (5.97871, 4.63111) node[anchor=north]{}
;

\draw (8.02129, 4.63111) node[anchor=north]{}
-- (9.67376,7.49329) node[anchor=south]{}
-- (11.3262, 4.63111) node[anchor=north]{}
;

\draw (1.6524745577599187,2.862183385639525) node[anchor=north]{}
-- (2.67376, 4.63111) node[anchor=south]{}
-- (3.6950454422400814, 2.862183385639525) node[anchor=north]{};

\draw (8.652480-3.6950454422400814, 2.862183385639525) node[anchor=north]{}
-- (5.97871, 4.63111) node[anchor=south]{}
-- (7,2.862183385639525) node[anchor=north]{};

\draw (7,2.862183385639525) node[anchor=north]{}
-- (8.02129, 4.63111) node[anchor=south]{}
-- (8.652480-3.6950454422400814+2.042570884480163+2.042570884480163, 2.862183385639525) node[anchor=north]{};

\draw (14-1.6524745577599187,2.862183385639525) node[anchor=north]{}
-- (11.3262, 4.63111) node[anchor=south]{}
-- (14-3.6950454422400814,2.862183385639525) node[anchor=north]{}
;

\draw[-{Stealth[length=3mm]}]
(5.66312,9.808845) node[anchor=south]{}
--(7,12.1244) node[anchor=north]{}
;

\draw[-{Stealth[length=3mm]}]
(9.67376,7.49329) node[anchor=north]{}
--(14-5.66312,9.808845) node[anchor=south]{}
;

\draw[-{Stealth[length=3mm]}]
(2.67361632,4.63085322) node[anchor=south]{}
--(3.49980816,6.06186621) node[anchor=north]{}
;

\draw[-{Stealth[length=3mm]}]
(5.15219184,6.06186621) node[anchor=north]{}
--(5.97838368,4.63085322) node[anchor=south]{}
;

\draw[-{Stealth[length=3mm]}]
(14-5.15219184,6.06186621) node[anchor=north]{}
--(14-5.97838368,4.63085322) node[anchor=south]{}
;

\draw[-{Stealth[length=3mm]}]
(14-2.67361632,4.63085322) node[anchor=south]{}
--(14-3.49980816,6.06186621) node[anchor=north]{}
;

\end{tikzpicture}
\]
Note that $b_1$ and $d_2$ are atoms of different types. Let $d_2$ be of type $M_2$.

From now on we will omit the formal proofs of atom decomposition, as essentially they follow the same format presented above. However, we invite the reader to convince themselves that the statements are valid.

\begin{cor} \label{tau2 and tau breakoff}
The graph $\MM$ contains 7 infinite atoms in $\AA_2(\MM)$. See Figure \ref{fig: aa_2} for visualization.
\begin{figure}
    \centering
    \includegraphics[width=0.89\linewidth]{"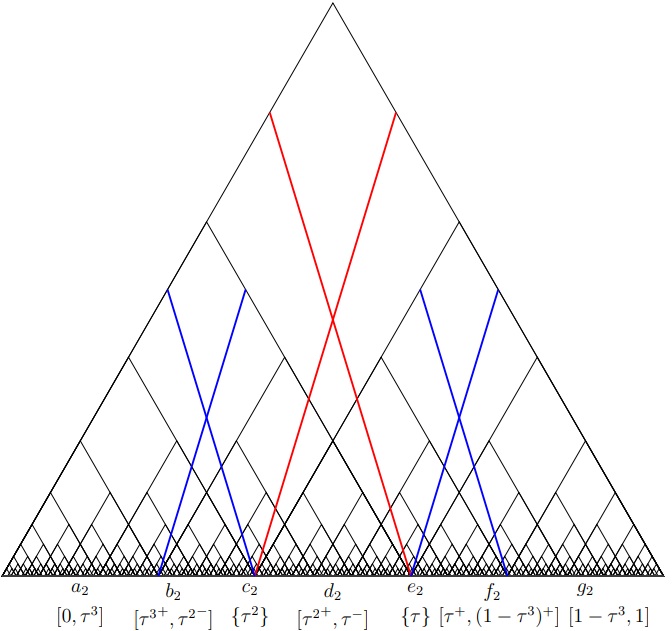"}
     \caption{Partition of 2-level infinite atoms $\AA_2(\MM)$}
     \label{fig: aa_2}
\end{figure}

\end{cor}

\begin{prop}
    The infinite atom $d_2$ at 2-level decomposes into three disjoint infinite atoms in $\AA_3(\MM)$, denoted as $f_3$, $g_3$, and $h_3$. These infinite atoms correspond to the following subsets of $\MM$: 
    \[\begin{aligned}
    &f_3 = Cone(LR^2L^2);\\
    &g_3 = Cone(L^2R^2) \setminus \{Cone(LR^2L^2) \cup Cone(LR^4)\};\\
    &h_3 = Cone(LR^4).\\
    \end{aligned}\]
Moreover, $d_2 = f_3 \cup g_3 \cup h_3$. See Figure \ref{fig: d_2} for visualization.

\begin{figure}
    \centering
    \includegraphics[width=0.89\linewidth]{"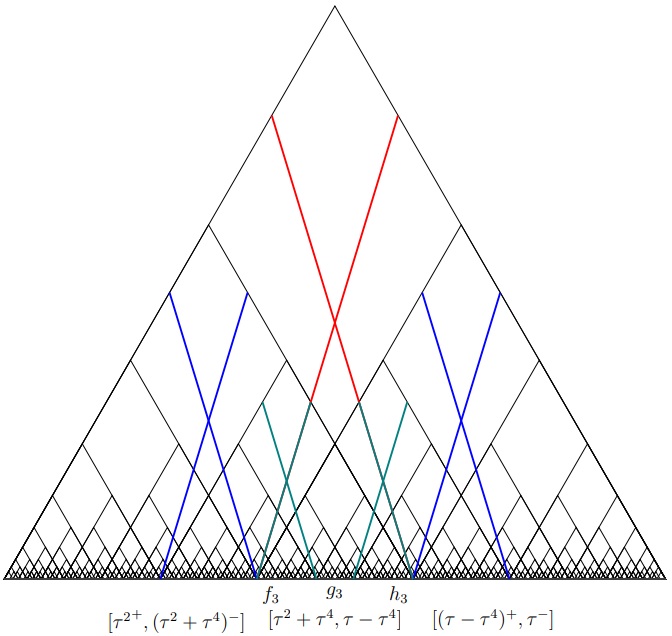"}
    \caption{Decomposition of the atom $d_2$ into $f_3$, $g_3$, and $h_3$ in the sequential level~$\AA_3(\MM)$}
    \label{fig: d_2}
\end{figure}

\end{prop}
The proof follows the same logic as for the previous decompositions. 

Note that $f_3$ and  $h_3$ are atoms of type $M_1$, this can be seen by looking at~$Cone(LR)$ and~$Cone(RL)$ and observing that the regions of atoms $f_3$, $h_3$ and $b_1$ within their respective cones are the same. Whereas~$g_3$ is an infinite atom of a new type, let it be denoted by~$M_3$.

\begin{prop}
    The infinite atom $g_3$ at the 3-level decomposes into three disjoint infinite atoms in $\AA_4(\MM)$, denoted as $j_4$, $k_4$, and $l_4$. These infinite atoms correspond to the following subsets of $\MM$: 
    \[\begin{aligned}
    &j_4 = Cone(LR^2L) \setminus \big(Cone(LR^2L^2) \cup Cone(LR^2LR^2)\big);\\
    &k_4 = Cone(LR^2LR^2);\\
    &l_4 = Cone(RL^2R) \setminus \big(Cone(RL^2R^2) \cup Cone(RL^2RL^2)\big).\\
    \end{aligned}\]
Moreover, $f_3 = j_4 \cup k_4 \cup l_4 \cup \{LR^2\}$. See Figure \ref{fig: f_3} for visualization.

\begin{figure}
    \centering
    \includegraphics[width=0.89\linewidth]{"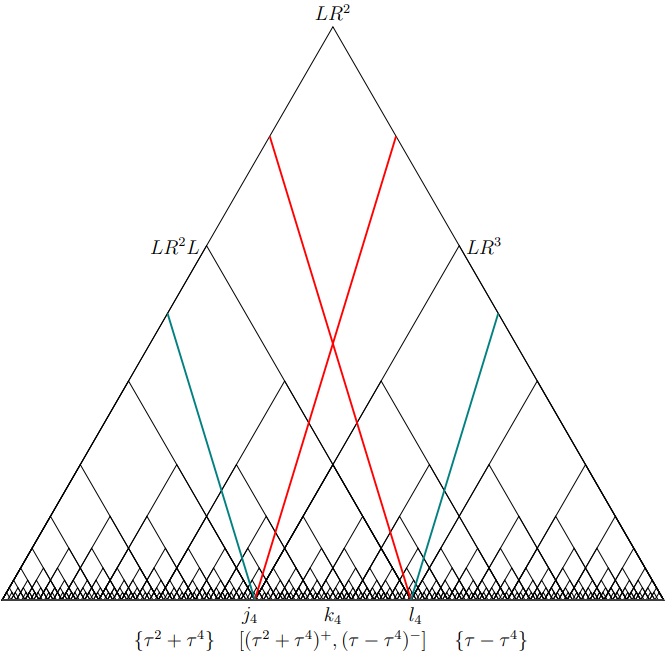"}
    \caption{Representation of the decomposition of the atom $f_3$ into $j_4$, $k_4$, and $l_4$}
    \label{fig: f_3}
\end{figure}

\end{prop}
The proof follows the same logic as before. 

Note that $j_4$ is an atom of type $P_R$, $k_4$ is of type $M_1$, and $l_4$ is of type $P_L$. Since no new infinite atom types have appeared, we obtain the following corollary.

\begin{cor}The tree of infinite atoms $\AA(\MM)$ has only a finite number of atom types. See Figure \ref{fig: type graph} for the type graph $\TT$ and Figure \ref{fig: tree of atoms} for the tree of infinite atoms $\AA(\MM)$.
\end{cor}

\begin{figure}
    \centering
    \includegraphics[width=0.59\linewidth]{"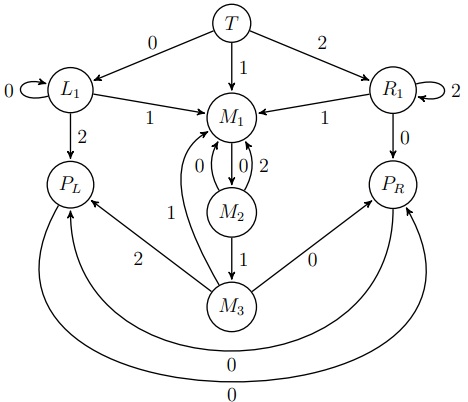"}
    \caption{The type graph $\TT$ of the infinite atoms in $\MM$ with labeled edges}
    \label{fig: type graph}
\end{figure}

\begin{figure}
    \centering
    \includegraphics[width=0.45\linewidth]{"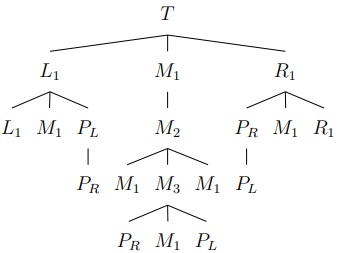"}
    \caption{The tree of atoms $\AA (\MM)$}
    \label{fig: tree of atoms}
\end{figure}
\noindent The path space $P(\TT)$ of the graph $\TT$ (Figure \ref{fig: type graph}) is the set of all infinite sequences of labels of infinite directional paths starting from the vertex $T$. The topology on $P(\TT)$ is generated by cylinder sets of the form $C_\alpha = \{w \in P(\TT) : w \text{ starts with } \alpha\}$ for finite paths~$\alpha$.

\begin{prop}\label{tree homeomorphism}
    The path space of the graph $\TT$ depicted by Figure \ref{fig: type graph} is naturally homeomorphic to $\DD_\tau$.  
\end{prop}

\begin{proof}
We will construct a homeomorphism $f: P(\TT) \rightarrow \DD_\tau$ recursively as follows:
\[
f(w) = [0,1]^T_w,
\]
where $[x,y]^A_w$ is defined for atoms $A \in \{T, L_1, M_1, R_1, M_2, M_3, P_L, P_R\}$ by:
\begin{equation}\label{eq: f function}
    \begin{aligned}
    [0,1]^T_{0\omega} &= [0,\tau^2]^{L_1}_{\omega};\\
    [0,1]^T_{1\omega} &= [{\tau^2}^+, \tau^-]^{M_1}_{\omega};\\
    [0,1]^T_{2\omega} &= [\tau,1]^{R_1}_{\omega};\\
    [x,y]^{L_1}_{0\omega} &= [0,\tau y]^{L_1}_{\omega};\\
    [x,y]^{L_1}_{1\omega} &= [(\tau y)^+,y^-]^{M_1}_{\omega};\\
    [x,y]^{L_1}_{2\omega} &= \{y\};\\
    [x^+,y^-]^{M_1}_{0\omega} &= [x^+,y^-]^{M_2}_{\omega};\\
    [x,y]^{R_1}_{0\omega} &= \{x\};\\  
    \end{aligned}
    \quad
    \begin{aligned}
    [x,y]^{R_1}_{1\omega} &= [x^+, \big(y-\tau(y-x)\big)^-]^{M_1}_{\omega};\\
    [x,y]^{R_1}_{2\omega} &=  [y-\tau(y-x),y]^{R_1}_{\omega};\\
    [x^+,y^-]^{M_2}_{0\omega} &=  [x^+, (x+\tau^2(y-x))^-]^{M_1}_{\omega};\\
    [x^+,y^-]^{M_2}_{1\omega} &=  [x+\tau^2(y-x),y-\tau^2(y-x)]^{M_3}_{\omega};\\
    [x^+,y^-]^{M_2}_{2\omega} &=  [\big(y-\tau^2(y-x)\big)^+,y^-]^{M_1}_{\omega};\\
    [x,y]^{M_3}_{0\omega} &= \{x\};\\
    [x,y]^{M_3}_{1\omega} &=  [x^+,y^-]^{M_1}_{\omega};\\
    [x,y]^{M_3}_{2\omega} &= \{y\}.\\
    \end{aligned}
\end{equation}
\noindent Although this function is written in a complicated form, it is not difficult to understand, as it only represents a recursive decomposition of atoms. This definition of $f$ means that it takes an infinite path~$ w = w_1w_2\cdots \in P(\TT)$  and produces a sequence of nested intervals in $[0,1]$, each corresponding to a subdivision of an atom according to the choices of $ w_1, w_2,\cdots$. In each step, the function $f$ applies a transformation to the interval $[0,1]$ based on the current value of $w_i$, as specified by the recursive rules. Formally,~$\bigcap^\infty_{n=1} [0,1]^T_{w_1w_2\cdots w_n} = \{f(w_1w_2\cdots)\}$. Note that the width of $[0,1]^T_{w_1\cdots w_n}$ is less than $\tau^{\frac{n}{2}}$, we add a factor of $\frac{1}{2}$ to compensate for the action $[x^+,y^-]^{M_1}_{0\omega}$ that does not alter the interval. 

By construction, for any $w \in P(\TT)$, $f(w) \in \DD_\tau$. Therefore, $f$ is well-defined.

The inverse map $g: \DD_\tau \to P(\TT)$ is defined as follows. Given $x \in \DD_\tau$, we express $x$ as a sequence of nested intervals, and we recover the sequence $w = (w_1, w_2, w_3, \dots)$ in $ P(\TT)$ such that:
\[
     w_n = \begin{cases}   
     0 &\text{ when } x \in [0,1]^T_{w_1w_2\cdots w_{n-1}0};\\
     1 &\text{ when } x \in [0,1]^T_{w_1w_2\cdots w_{n-1}1};\\
     2 &\text{ when } x \in [0,1]^T_{w_1w_2\cdots w_{n-1}2}.\\
    \end{cases}
\]
This defines a sequence $w = (w_1, w_2, w_3, \dots)$ that corresponds to $x$. Since $g$ is the inverse of $f$, we have~$g\big(f(w)\big) = w$ for all $w \in P(\TT)$. Hence,~$f$ is a bijection.

We will now show that open sets $U \in P(\TT)$ are mapped to open sets $f(U) \in \DD_\tau$. By construction, every open set in $P(\TT)$ is a cylinder set. Recall that a cylinder set is defined by $C_\alpha = \{w \in P(\TT) : w \text{ starts with } \alpha\}$ for finite paths $\alpha$. Hence, we only need to show that every finite path in $P(\TT)$ is mapped by $f$ to an open set in $D_\TT$. Due to the structure of $D_\TT$, every point and closed interval contained in $D_\TT$ can be expressed as an open interval. We will change the representation of equations~$\ref{eq: f function}$ to show that every image of a finite path $\alpha \in P(\TT)$ under $f$ is an open set:
\begin{equation}\label{eq: f' function}
    \begin{aligned}
    [0,1]^T_{0\omega} &= \big({-}\infty,({\tau^2})^+\big)^{L_1}_{\omega};\\
    [0,1]^T_{1\omega} &= (\tau^2, \tau)^{M_1}_{\omega};\\
    [0,1]^T_{2\omega} &= (\tau^-,+\infty)^{R_1}_{\omega};\\
    [x,y]^{L_1}_{0\omega} &= \big({-}\infty,(\tau y)^+\big)^{L_1}_{\omega};\\
    [x,y]^{L_1}_{1\omega} &= (\tau y,y)^{M_1}_{\omega};\\
    [x,y]^{L_1}_{2\omega} &= (y^-,y^+);\\
    [x^+,y^-]^{M_1}_{0\omega} &= (x,y)^{M_2}_{\omega};\\
    [x,y]^{R_1}_{0\omega} &= (x^-,x^+);\\  
    \end{aligned}
    \quad
    \begin{aligned}
    [x,y]^{R_1}_{1\omega} &= \big(x, y-\tau(y-x)\big)^{M_1}_{\omega};\\
    [x,y]^{R_1}_{2\omega} &=  \Big(\big(y-\tau(y-x)\big)^-,\infty\Big)^{R_1}_{\omega};\\
    [x^+,y^-]^{M_2}_{0\omega} &=  \big(x, x+\tau^2(y-x)\big)^{M_1}_{\omega};\\
    [x^+,y^-]^{M_2}_{1\omega} &=  \Big(\big(x+\tau^2(y-x)\big)^-,\big(y-\tau^2(y-x)\big)^+\Big)^{M_3}_{\omega};\\
    [x^+,y^-]^{M_2}_{2\omega} &=  \big(y-\tau^2(y-x),y\big)^{M_1}_{\omega};\\
    [x,y]^{M_3}_{0\omega} &= (x^-,x^+);\\
    [x,y]^{M_3}_{1\omega} &= (x,y)^{M_1}_{\omega};\\
    [x,y]^{M_3}_{2\omega} &= (y^-,y^+).\\
    \end{aligned}
\end{equation}
Formally, if $\alpha$ is a finite path in $P(\TT)$ and $f(\alpha)$ transforms the interval $[0,1]$ $n$ times, then we use \ref{eq: f function} for the first $n-1$ transformations and \ref{eq: f' function} for the last $n^{th}$ transformation. This approach guaranties that the function remains well-defined and that the resulting interval is open. Hence, the inverse $g$ is a continuous function.

By \cite[Theorem 2.1]{Webster}, which states that the path space of any directed graph is a locally compact Hausdorff space, we  conclude that $P(\TT)$ is a Hausdorff space.

Since the Cantor-like space $D_\TT$ is a union of a Cantor set and isolated points, it is closed and bounded, and therefore compact.

Finally, we use the fact that a bijective continuous function from a compact space to a Hausdorff space is a homeomorphism. The map $f: P(\TT) \rightarrow D_\TT$ is such a function; therefore,~$f$ is indeed a homeomorphism.
\end{proof}

\noindent We now have enough tools to prove the main result of this section.
\begin{proof} [Proof of Theorem \ref{horofunction boundary of $M$}]
    By Proposition \ref{tree homeomorphism} the space of the path space~$P(\TT)$ of the type graph~$\TT$ (see Figure \ref{fig: type graph}) is homeomorphic to $\DD_\tau$. According to \cite[Proposition 2.21]{Belk Rational embeddings} the boundary of the tree of atoms~$\AA(M)$ (see Figure \ref{fig: tree of atoms}) is isomorphic to the path space of the type graph~$\TT$. Finally, by Theorem~\ref{equivalence boundary = horofunction boundary} the boundary of the tree of atoms $\AA(M)$ is homeomorphic to the horofunction boundary~$\partial_h \MM$ of~$\MM$.

    Hence, horofunction boundary $\partial_h \MM$ of $\MM$ is homeomoprhic to $\DD_\tau$.
\end{proof}

\begin{remark}
    It was an open question whether any non-elementary hyperbolic groups have isolated points in their horofunction boundaries. In the next section, we show that $M$ is hyperbolic, so this gives an example of a non-elemenraty hyperbolic monoid whose horofunction boundary has isolated points.
\end{remark}

\section{Gromov hyperbolicity}\label{sec: hyperbolicity} Hyperbolicity in graphs is a fundamental concept in geometric group theory. It indicates tree-like properties on a large scale and often leads to efficient solutions to word problems \cite{Gromov}. In the context of monoids, a hyperbolic Cayley graph suggests negatively curved geometry. This has profound implications for the structure's behavior, including efficient word problems and connections to finitely generated groups. This section proves the following.
\begin{thm} \label{M hyperbolic}
    The Cayley graph $\MM = (V,E,r)$ of the monoid $M=\langle L, R : LR^2 = RL^2\rangle$ is hyperbolic.
\end{thm}
\subsection{Preliminaries on horizontal graphs} Following \cite[Sections 1-2]{Kong}, we introduce the reader with the notion of horizontal edges, which play a key role in the last section of this paper.

Let $\Gamma = (V,E,r)$ be a rooted graph. The \textit{vertical edge set} $E_v \subseteq E$ is defined as:
\[E_v = \big\{\{x,y\} \in E : |x| - |y| = \pm 1\big\}.\] 
The \textit{horizontal edge set} $E_h \subset E$ is defined as: 
\[E_h = \big\{\{x,y\} \in E : |x| = |y|\big\}.\] 
Therefore, the edge set $E$ can be partitioned as $E = E_v \cup E_h$. Note that in a connected graph with~$E \neq \emptyset$, $E_h$ may be empty, but $E_v$ is always non-empty.

For $\Gamma = (V, E, r)$, we define the \textit{horizontal distance} $d_h(\cdot, \cdot)$ as the graph distance on the subgraph induced by $(V, E_h)$. For any pair of vertices $x, y \in V$, $d_h(x,y) = \infty$ whenever~$|x| \neq |y|$. Moreover, the inequality~$d(x,y) \leq d_h(x,y)$ holds for all $x, y \in V$. When equality occurs, i.e., $d(x,y) = d_h(x,y)$, there exists a geodesic $\pi(x,y)$ that is entirely contained in $(V, E_h)$. We refer to it as the \textit{horizontal geodesic} of $\Gamma$ denoted by $\pi_h(x,y)$.

Let's briefly recall from Section \ref{graph back} the notation for descendant and predecessor sets. Let $m \geq 0$ and~$x \in V$, then
\[\begin{aligned}
    J_m(x) &:= \{y \in V : x \preceq y, |y| = |x| + m\};\\
    J_{-m}(x) &:= \{z \in V : x \in J_m(z)\};\\
\end{aligned}\]
are the $m$-th descendant and $m$-th predecessor sets of $x$, respectively. 

\begin{definition}
   \cite[Definition 2.1]{Kong} A rooted graph $\Gamma = (V, E,r)$ is called \textit{expansive} if it satisfies the following condition:
\[\forall x, y \in V, \ \forall u \in J_1(x), \ \forall v \in J_1(y), \ d_h(x, y) > 1 \implies d_h(u, v) > 1.\]
\end{definition}
\begin{definition}
    \cite[Definition 2.5]{Kong} A rooted graph $\Gamma=(V, E,r)$ is called \textit{$(m,k)$-departing} if there exists $m, k \in \mathbb{N}$ such that:
    \[
   \forall x, y \in V, \ d_h(x, y) > k \implies \forall u \in J_m(x), v \in J_m(y), \ d_h(u, v) > 2k.
    \]
\end{definition}
\noindent Let's consider a short example of an expansive graph.
\begin{example}
Let $\Gamma_n = (V, E, r)$ be an infinite $n$-ary tree (where each vertex has exactly $n$ descendants) with additional horizontal edges between every pair of vertices in lexicographic order on the same level. 

This family of graphs $\Gamma_n$ is expansive. To prove this, consider two vertices $x, y \in V$ such that~$d_h(x,y) > 1$, and let $u \in J_1(x)$ and $v \in J_1(y)$. If $|x| = |y|$, then $|u| = |v| = |x| + 1$, and $d_h(u,v) > 1$ because the horizontal distance between parents/children of non-adjacent vertices at the same level is always greater than 1. If $|x| \neq |y|$, then $d_h(u,v) = \infty > 1$ by the definition of horizontal distance. Therefore, $\Gamma_n$ satisfies the expansive condition for all $n \geq 2$.
\end{example}

Let $\Gamma = (V,E)$ be a locally finite connected graph. Let $x, y$ and $z$ be vertices in $V$, then the three geodesics that join them are called \textit{sides} and form a \textit{geodesic triangle}. If each of the sides of the geodesic triangle is contained in the $\delta$-neighborhood of the union of the other two sides, for some non-negative $\delta$, then such a triangle is called \textit{$\delta$-thin}. A locally finite connected graph is called \textit{$\delta$-hyperbolic} if each geodesic triangle in $\Gamma$ is $\delta$-thin. We call the smallest such $\delta \geq 0$ the \textit{hyperbolicity constant} of $\Gamma$. A graph is called \textit{hyperbolic in the Gromov sense} or simply \textit{hyperbolic} if there exists a $\delta \geq 0$ such that each subgraph of~$\Gamma$ is $\delta$-hyperbolic. \cite{Gromov}.

\begin{thm}\label{hyperbolic theorem}
    \cite[Theorem 1.1]{Kong} Let $\Gamma = (V, E,r)$ be an expansive rooted graph. The following statements are equivalent:
    \begin{enumerate}
        \item $(V, E,r)$ is hyperbolic;
        \item There exists a constant $P < \infty$ such that the lengths of all horizontal geodesics are bounded by $P$;
        \item $(V, E,r)$ is $(m,k)$-departing for $m,k \in \NN$.
    \end{enumerate}
\end{thm}

\begin{definition}\cite[Definition 5.1.6]{Loh}
Let $(X, d_X)$ and $(Y, d_Y)$ be metric spaces, and let $f \colon X \to Y$ be a map. We say that $f$ is a \textit{quasi-isometry} if there exist constants~$A \geq 1$,~$B \geq 0$ and~$C\geq0$ such that the following conditions hold:
\begin{enumerate}
    \item $\forall x, y \in X :\frac{1}{A} d_X(x, y) - B \leq d_Y\big(f(x), f(y)\big) \leq A d_X(x, y) + B;$
    \item  $\forall z\in Y:  \exists x \in X : d_Y\big(z,f(x)\big)\leq C.$
\end{enumerate}
\noindent The constants $A$ and $B$ are called the \textit{quasi-isometric constants} of the embedding $f$. When only the first condition is satisfied $f$ is called a \textit{quasi-isometric embedding}.

Note that if $f$ is a surjective quasi-isometric embedding from $X$ to $Y$ then $f$ is a quasi-isometry.
\end{definition}

\begin{thm}\label{quasi invariant hyperbolic theorem}
\cite[Corollary 7.2.13]{Loh} Let $X$ and $Y$ be geodesic metric spaces. Let $f : X \to Y$ be a quasi-isometry. Then $X$ is hyperbolic if and only if $Y$ is hyperbolic.
\end{thm}

\subsection{Hyperbolicity of the \textit{Cay(M)}} We will now prove theorem \ref{M hyperbolic} by a number of lemmas. The first step is to construct a modified graph $\MM'$ that includes additional horizontal edges between every pair of adjacent vertices that are not present in~$Cay(M)$. The next step will be to establish that the two graphs are quasi-isometric. Relying on Theorem \ref{hyperbolic theorem} from the results of Kong, Lau, and Wang \cite{Kong}, we prove that the newly introduced graph is hyperbolic, which is sufficient since quasi-isometries preserve hyperbolicity

We will now construct the modified graph $\MM'$. Since~$Cay(M)= (M,E,1)$ has no horizontal edges in, we allow~$Cay(M) = (M, E_v, 1)$. Now, define $\MM' = (M, E_v\cup E_h, 1)$, where $E_h$ is the set of horizontal edges defined by
\[E_h=\{(x,y) :\exists m \in V, \text{ } x = mL, y=mR\} \cup \{ (x,y) : \exists m\in V, \text{ } x=mLR, y=mRL\}.\]
See Figure \ref{fig:M M' graphs} for an illustration.

\begin{figure}[!ht]
    \centering
    \includegraphics[width=1\linewidth]{"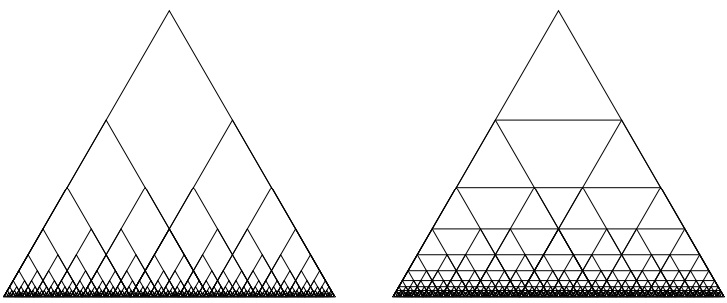"}
\caption{The graphs $Cay(M)$ and $\MM'$}
    \label{fig:M M' graphs}
\end{figure}

Observe that the graph $Cay(M)$ is contained in $\MM'$. We now have enough tools to prove that the large-scale geometry of the graphs $Cay(M)$ and $\MM'$ is essentially the same.
\begin{prop} \label{prop:quasi}
Let $f: Cay(M) \rightarrow \MM'$ be the natural inclusion map, then $f$ is a quasi-isometry.
\end{prop}

\begin{proof}
Let $f \colon Cay(M) \to \MM'$ be the natural inclusion map. Since $\MM'$ contains every path of~$Cay(M)$, we conclude that for any pair of vertices $x, y \in Cay(M)$, the following holds:
\[
    d_{\MM'}\big(f(x), f(y)\big) \leq d_{Cay(M)}(x, y).
\]
On the other hand, every added horizontal edge in $\MM'$ decreases the minimal distance between the points it connects by a factor of $2$. Hence, for every pair of vertices $x, y \in Cay(M)$, we have
\[
    \frac{1}{2} d_{Cay(M)}(x, y) \leq d_{\MM'}\big(f(x), f(y)\big).
\]
Combining these two inequalities, we obtain
\[
    \frac{1}{2} d_{Cay(M)}(x, y) \leq d_{\MM'}\big(f(x), f(y)\big) \leq 2 d_{Cay(M)}(x, y),
\]
which shows that the natural inclusion map $f \colon Cay(M) \to \MM'$ is a quasi-isometric embedding with quasi-isometric constants $A=2$ and $B = 0$. Therefore, $Cay(M)$ embeds quasi-isometrically in $\MM'$. Since $f$ is surjective, we conclude that it is a quasi-isometry.
\end{proof}

\begin{prop}\label{expansive}
    The graph $\MM' = (V,E_v\cup E_h, r)$ is expansive.
\end{prop}

\begin{proof}
    We will prove this using the contrapositive statement
    \[\forall u, v \in V, \text{ and }\forall x \in J_{-1}(u), \text{ } \forall y \in J_{-1}(v) \text{ s.t. } d_h(u,v) \leq 1 \Rightarrow d_h(x,y) \leq 1.\]

    We will consider two possible cases corresponding to the possible horizontal distances between~$u$ and~$v$. 

    Let $d_h(u,v)=0$, i.e. $u$ and $v$ be the same vertex. In this case $u$ and $v$ have the same predecessors due to the monoid property $LR^2 = RL^2$. If for some predecessor pair~$x$ and~$y$,~$d_h(x,y) > 1$, the restriction~$LR^2=RL^2$ would force~$d_h(u,v)>1$, contradicting our assumption that $d_h(u,v) = 0$.

    Now, let $d_h(u,v)=1$. Due to the monoid property $LR^2=RL^2$, within the next two levels above~$u$ and~$v$, there always exists a common ancestor $w \in J_{-2}(u)$, $w \in J_{-2}(v)$. This common ancestor $w$ ensures that all predecessors $x$ and $y$ are at most 1 apart, as any greater distance would contradict the existence of the common ancestor within two levels.

    Both cases show that $d_h(u,v) \leq 1 \implies d_h(x,y) \leq 1$. Hence, this proves the contrapositive statement, which is equivalent to the definition of expansiveness. Hence, the graph~$\MM'$ is expansive.
\end{proof}
Our next goal is to prove that the graph $\MM'$ is $(m,k)$-departing. But first we need some definitions. In the graph $\MM'$, we define two types of horizontal edges of length 1:
\begin{itemize}
    \item let $U$ be a horizontal edge where the right predecessor of the left vertex and the left predecessor of the right vertex coincide;
    \item let $D$ be a horizontal edge where the right descendant of the left vertex and the left descendant of the right vertex coincide.
\end{itemize}
\[
\begin{tikzpicture}[scale=2]
\draw (4.32624*2-1.6524745577599187*2+1.41139-5.19852,0) node[anchor=north]{}
-- (8.65248-2.28366-5.19852, 1.76893-1.09326) node[anchor=south]{}
-- (5.97871*4-8.65248*2+2.28366*2-4.32624*2+3.45396-5.19852, 0) node[anchor=north]{}
--cycle;

\draw(6.3688004422-5.19852,0) node[anchor=south]{$U$};
\end{tikzpicture}
\qquad
\begin{tikzpicture}[scale=2]
\draw (8.65248-2.28366-5.19852, 1.76893-1.09326) node[anchor=south]{}
-- (5.97871*4-8.65248*2+2.28366*2-4.32624*2+3.45396-5.19852, 0) node[anchor=north]{}
--(5.97871*4-8.65248*2+2.28366*2-4.32624*2+3.45396-5.19852,0) node[anchor=north]{}
--(5.97871*2-8.65248+2.28366-5.19852, 1.76893-1.09326) node[anchor=south]{}
--cycle;

\draw(5.97871*4-8.65248*2+2.28366*2-4.32624*2+3.45396-5.19852, 1.76893-1.09326) node[anchor=south]{$D$};

\end{tikzpicture}
\]

In the graph $M'$, the following properties hold:
\begin{enumerate}
    \item an edge cannot be both $U$ and $D$ simultaneously;
    \item the pattern of the edges $U$ and $D$ follows a specific structure:
    \begin{itemize}
        \item Level 1 (from the root): $U$
        \item Level 2: $UDU$
        \item Level 3: $UDUUDU$
        \item Level 4: $UDUUDUDUUDU$
    \end{itemize}
\end{enumerate}
The patterns above are obtained from the substitution structure. By analyzing the structure, we obtain the following substitution rules:
\begin{enumerate}
    \item a single edge $U$, that is not adjacent to any other $U$'s at level $n$ is transformed into a sequence of~$UDU$ edges at level $n+1$;
    \item a sequence of $UU$ edges at level $n$ is transformed into a sequence of $UDUDU$ edges at level $n+1$;
    \item an  edge $D$ is ignored in the transformation to the next level. We can view $D$'s as having the purpose of separating~$U$'s into singles and doubles.
\end{enumerate}
\begin{figure}[H]
    \centering
    \includegraphics[width=0.8\linewidth]{"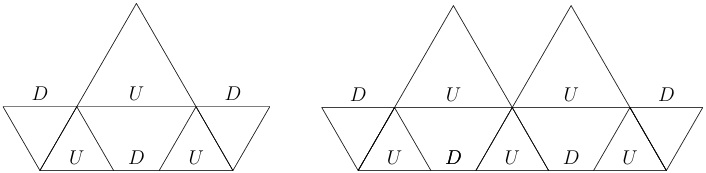"}
\end{figure}
We invite the reader to expand these pictures with various combinations of types of horizontal edges.

Note that  we can define $U$ and $D$ for the graph $\MM$ by considering the same pairs of vertices.

\begin{prop}\label{M' (m,k)}
    The graph $\MM'$ is $(m,k)$-departing.
\end{prop}

\begin{proof}
    We will show that $M'$ is at most $(3,2)$-departing, i.e. 
    \[
   \forall x,y \in X \text{ and } \forall u \in J_3(x), v \in J_3(y) \text{ s.t. } d_h(x,y) > 2 \Rightarrow d_h(u,v) >4, 
    \]
    We will prove by proving the contrapositive statement:
    \[
   \forall u,v \in X \text{ and } \forall x \in J_{-3}(u), y \in J_{-3}(v) \text{ s.t. } d_h(u,v) \leq 4 \Rightarrow d_h(x,y) \leq2, 
    \]

Without loss of generality, let $u$ be located to the left of $v$ on the graph, such that $d_h(u,v) \leq 4$. Due to lexicographic order, we only need to consider the extreme case where~$d_h(u,v) = 4$. There are 3 possible combinations of horizontal edges of length 4:~$UDUD$, $UUDU$, and~$DUUD$. We ignore cases like~$UDUU$ due to symmetry.

We start by analyzing the $UDUD$ case. The key here is to observe that in recovering the predecessors a horizontal edge of type $U$ contributes to the appearance of 2 vertical edges, whereas type $D$ lacks contribution. Note that the adjacent edges of the sequence $UDUD$ can also contribute to the formation of horizontal edges at the above level; therefore, we have to look at all possible combinations in which $UDUD$ can appear. By analyzing the substitution structure, we observe that there are a total of 3 such subcases: $UUDUDU$, $UDUDU$, and $DUDUDU$.

The following picture shows two levels of edges above the sequence of edges $UDUD$ inside $UUDUDU$:
\begin{figure}[H]
    \centering
    \includegraphics[width=0.7\linewidth]{"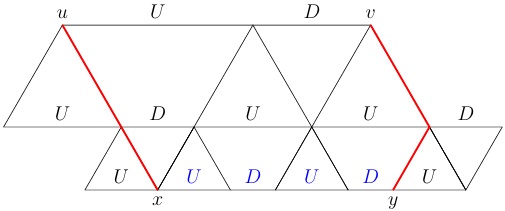"}
\end{figure}
\noindent Let $x$ and $y$ be the extreme vertices in the sequence of edges $UDUD$. Then the distance between the extreme predecessors $J_{-2}(x) = u$ and $J_{-2}(y) = v$ is 2. By the monoid property, the distance of $J_{-1}(u)$ and $J_{-1}(v)$ cannot be greater than 2. Note that the sequence~$UDUD$ is always followed by a~$U$, however, it may or may not have a prefix of $U$ or $D$. The absence or presence of a $D$ in front of the sequence will not contribute to the edges in the previous level. The $UDUDU$ and $DUDUDU$ subcases are proven in a similar way.

The next case is $UUDU$. Once again we only look at $UUDU$ inside: $DUUDUU$, $UUDUU$, $UUUDUD$, $UUDUD$, $DUUDU$ and $UUDU$. The following picture shows two levels of edges above the sequence of edges $UUDU$ inside $DUUDUU$: 
\begin{figure}[H]
    \centering
    \includegraphics[width=0.6\linewidth]{"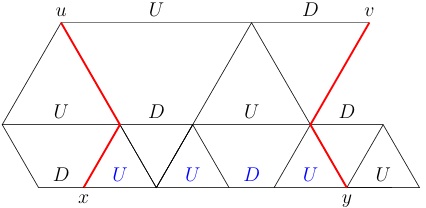"}
\end{figure}
\noindent Observe that we obtain the same situation for the extreme predecessors $J_{-2}(x) = u$ and~$J_{-2}(y) = v$ as before. The subcases $UUDUU$, $UUUDUD$, $UUDUD$, $DUUDU$, and $UUDU$ are proven in a similar way.

We are left with the $DUUD$ case. Since $D$'s are not allowed to be adjacent to each other or at the boundary, we look at $DUUD$ inside $UDUUDU$. 
\begin{figure}[H]
    \centering
    \includegraphics[width=0.65\linewidth]{"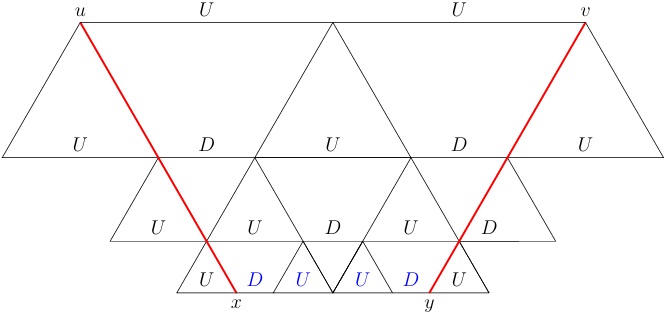"}
\end{figure}
\noindent For the extreme case $J_{-3}(x) = u$ and~$J_{-3}(y) = v$ we find that the distance is at most 2.

Note that we proved only the extreme cases. All other cases are proven in a similar way; moreover, most of them are contained in the pictures provided above.
Hence, this proves the contrapositive statement, which implies that $\MM'$ is at most $(3,2)$-departing.
\end{proof}

\begin{cor}\label{M' hyperbolic}
    The graph $\MM'$ is hyperbolic.
\end{cor}

\begin{proof}
    Proposition \ref{expansive} shows that $\MM'$ is expansive. Hence, we can apply Theorem \ref{hyperbolic theorem} to Proposition \ref{M' (m,k)}, which concludes that $\MM'$ being $(m,k)$-departing is equivalent to $\MM'$ being hyperbolic.
\end{proof}

\begin{proof}[Proof of Theorem \ref{M hyperbolic}]
    Proposition \ref{prop:quasi} states that $f:Cay(M) \rightarrow \MM'$ is a quasi-isometry. Corollary \ref{M' hyperbolic} shows that $\MM'$ is hyperbolic. Then by Theorem \ref{quasi invariant hyperbolic theorem} $Cay(M)$ must also be hyperbolic.
\end{proof}

\end{document}